\documentclass[11pt]{amsart}
\usepackage[utf8]{inputenc}  

 \usepackage[all]{xy}
 \usepackage{amscd,amsfonts,amssymb,amsmath,amsthm,latexsym}
\usepackage[colorlinks=true, linktocpage=true, linkcolor=blue, citecolor=blue]{hyperref}
\usepackage{cleveref}
\usepackage[utf8]{inputenc}
\usepackage{float}

\usepackage{tikz-cd}
\usepackage{xcolor}

\usepackage{soul}

\newcommand{\suchthat}{\ | \ }

\newcommand{\marked}{\mathbb{M}}
\newcommand{\punct}{\mathbb{P}}
\newcommand{\bSigma}{\mathbf{\Sigma}}
\newcommand{\surf}{(\Sigma,\marked,\punct)}

\newcommand{\Sh}{\operatorname{Sh}}
\newcommand{\ind}{{\mathrm{ind}}}
\newcommand{\KRS}{\operatorname{KRS}}
\newcommand{\supp}{\operatorname{supp}}
\newcommand{\abs}[1]{\lvert #1\rvert}
\newcommand{\ebrace}[1]{\langle\!\langle #1\rangle\!\rangle}
\newcommand{\field}{\Bbbk}
\newcommand{\DecIrr}{\operatorname{DecIrr}}
\newcommand{\dimv}{\underline{\dim}}   
\newcommand{\Hom}{\operatorname{Hom}}

\newcommand{\prj}{{\operatorname{proj}}}
\newcommand{\op}{{\operatorname{op}}}  
\newcommand{\cLC}{\mathcal{LC}}
\newcommand{\Intn}{\operatorname{Int}}
\newcommand{\Lamin}{\operatorname{Lam}}

\newcommand{\Ker}{\operatorname{ker}}
\newcommand{\End}{\operatorname{End}}
\newcommand{\jacobalg}[1]{\mathcal{P}(#1)}
\newcommand{\Gr}{\operatorname{Gr}}
\newcommand{\FST}{\operatorname{FST}}
\newcommand{\MSW}{\operatorname{MSW}}

\usepackage{fancyhdr}

\fancypagestyle{plain}{
 \fancyhf{}
 \fancyfoot[C]{\thepage}%
}
\fancypagestyle{fockgonchorov}{
 \fancyhf{}
 \fancyfoot[L]{\footnotesize{
$^{\dagger}$ Musiker-Schiffler-Williams' result in the coefficient-free setting heavily relies on the work of Fock and Gonchorov~\cite[Theorem 12.3(ii)]{fock2006moduli}}}}

\newcommand{\sgn}{\text{sgn}}

\newcommand{\RN}[1]{%
 \textup{\uppercase\expandafter{\romannumeral#1}}%
}

\newcommand{\bd}{\mathbf{d}}
\newcommand{\be}{\mathbf{e}}  
\newcommand{\bg}{\mathbf{g}}

\newcommand{\bv}{\mathbf{v}}
\newcommand{\bw}{\mathbf{w}}
\newcommand{\bx}{\mathbf{x}}
\newcommand{\by}{\mathbf{y}}

\newcommand{\cA}{\mathcal{A}} 
\newcommand{\cB}{\mathcal{B}}

\newcommand{\cC}{\mathcal{C}}

\newcommand{\cP}{\mathcal{P}}

\newcommand{\cU}{\mathcal{U}}

\newcommand{\CC}{\mathbb{C}}     

\newcommand{\NN}{\mathbb{N}}

\newcommand{\ZZ}{\mathbb{Z}}

\newcommand{\alp}{\alpha}      
\newcommand{\bet}{\beta}
\newcommand{\gam}{\gamma}      
\newcommand{\del}{\delta}

\newcommand{\lam}{\lambda}      
\newcommand{\sig}{\sigma}

\newcommand{\Lam}{\Lambda}
\newcommand{\Sig}{\Sigma}      
\newcommand{\bSig}{\mathbf{\Sig}}      

\newcommand{\Ma}{\mathbb{M}}
\newcommand{\Pu}{\mathbb{P}}

\newcommand{\ra}{\rightarrow}

\newtheorem{thm}{Theorem}[section]
\newtheorem*{thm*}{Theorem}
\newtheorem{lem}[thm]{Lemma}
\newtheorem{prop}[thm]{Proposition}
\newtheorem{cor}[thm]{Corollary}

\theoremstyle{definition}
\newtheorem{defn}[thm]{Definition}

\newtheorem{prob}[thm]{Problem}
\newtheorem{exmp}[thm]{Example}

\newtheorem{rmk}[thm]{Remark}
\newtheorem{case}{Case}
\newtheorem{subcase}{Subcase}[case]

\numberwithin{equation}{section}

\title[Bangle functions are the generic basis for punctured surfaces]{Bangle functions are the generic basis for cluster algebras from punctured surfaces with boundary}

\author{Christof Geiss}
\address{Christof Geiss\newline
Instituto de Matem\'aticas, UNAM, Mexico}
\email{christof.geiss@im.unam.mx}

\author{Daniel Labardini-Fragoso}
\address{Daniel Labardini-Fragoso\newline
Dipartimento di Matematica ``Tullio Levi-Civita'',
Universit\'a a degli Studi di Padova, Italy}
\email{daniel.labardinifragoso@unipd.it, labardini@math.unipd.it}

\author{Jon Wilson}
\address{Jon Wilson\newline
Jeremiah Horrocks Institute, University of Lancashire, UK}
\email{jwilson30@lancashire.ac.uk}

\date{September 18, 2025}

\begin{document}
\begin{abstract}
\noindent 
Let $\mathbf{\Sigma}=(\Sigma, \mathbb{M}, \mathbb{P})$
be a surface with marked points and non-empty boundary.
We prove that for any tagged triangulation $T$ of $\mathbf{\Sigma}$ in the sense of Fomin--Shapiro--Thurston, the coefficient-free bangle functions of Musiker--Schiffler--Williams coincide with the coefficient-free generic Caldero--Chapoton functions arising from the Jacobian algebra of the quiver with potential $(Q(T), W(T))$ associated to $T$ by Cerulli Irelli and the second named author.

When the set of boundary marked points $\mathbb{M}$ has at least two elements, Schröer and the first two authors have shown, relying heavily on results of Mills, Muller and Qin, that the generic coefficient-free Caldero-Chapoton functions form a basis of the coefficient-free (upper) cluster algebra
$\mathcal{A}(\mathbf{\Sigma})=\mathcal{U}(\mathbf{\Sigma})$. So, the set of bangle functions proposed by Musiker--Schiffler--Williams over ten years ago is indeed a basis. Previously, this
was only known in the unpunctured case. 

\end{abstract}

\maketitle

\tableofcontents

\section{Introduction}
\subsection{Bases for cluster algebras} Fomin and Zelevinsky introduced around 2001 cluster algebras
\cite{fomin2002cluster,fomin2003cluster,fomin2007cluster} 
as a device to study dual canonical bases in quantum groups and
the closely related theory of total positivity in complex reductive
groups.  They expected that in this context all cluster
monomials should belong to the dual  canonical basis. This question
was only recently settled by Kang--Kashiwara--Kim--Oh~\cite{kang2018monidal} for the case of generalized symmetric Cartan matrices and by Qin~\cite{qin2020dual}  for 
generalized symmetrizable Cartan matrices.  Previously, the
first named author together with Leclerc and Schröer~\cite{geiss2012generic} had identified
the dual of Lusztig's  semicanonical basis with the generic Caldero-Chapoton functions, which also contain all cluster monomials.  

Since the set of cluster monomials is always linearly 
independent~\cite{cerulli2013linear} this led more generally to
the search of bases for (upper) cluster algebras which contain the
cluster monomials, and which ideally expand positively 
with respect to each cluster.  Moreover,  such ``good'' bases are 
typically parametrized by the by tropical points of the corresponding cluster Poisson variety.  By the landmark work of 
Gross--Hacking--Keel--Kontsevich~\cite{gross2018canonical} on
theta functions at least one such basis exists under rather weak
hypothesis. See the remarkable paper~\cite{qin2019bases} for a thorough discussion. 

Parallel to these developments, cluster algebras have found applications in a surprising variety of subjects, see for example 
the surveys~\cite{fomin2010totalp},\cite{keller2012cluster},\cite{keller2020green},\cite{leclerc2010cluster}.

\subsection{Surface cluster algebras and bangle functions.} 
A prominent class of cluster algebras stems from surfaces with marked points, thanks to works of
Fock--Goncharov \cite{fock2007dual}, Fomin--Shapiro--Thurston \cite{fomin2008cluster}, Fomin--Thurston \cite{fomin2018cluster}, Gekhtman--Shapiro--Vainshtein \cite{gekhtman2005cluster} and Penner \cite{penner2012decorated}.

Let  $\bSig=(\Sigma, \Ma, \Pu)$ be a (bordered) surface with marked points in the sense of~\cite{fomin2008cluster}, where
$\Ma$ denotes the set of marked points on the boundary and $\Pu$ is the set of punctures. 
See Definition~\ref{def:surface} for more details. 
To associate a (coefficient-free) cluster algebra to $\bSig$   one considers tagged triangulations of $\bSigma$. 
Following~\cite[Def.~9.6]{fomin2008cluster}  each tagged 
triangulation $T$ has an associated skew-symmetric matrix $B(T)$, 
see also Definition~\ref{adjacency matrix} for details.
The matrix $B(T)$ changes under any flip of an arc in $T$ according to the Fomin--Zelevinsky mutation rule.  This allows us already to define a cluster algebra $\cA(\bSigma)=\cA(B(T),\mathbf{x})$ which is actually independent of the choice of
the triangulation $T$. In fact, a fundamental result of 
Fomin--Shapiro--Thurston~\cite[Cor.~6.2]{fomin2018cluster} and~\cite[Thm.~7.11]{fomin2008cluster} states that provided $\bSigma$ is not a closed surface with exactly one puncture, there is a bijection $\gamma\mapsto x_\gamma$ between the set of tagged arcs on $\bSigma$ and the set of cluster variables of
$\mathcal{A}(B(T_0),\mathbf{x})$, such that the induced assignment $T\mapsto(B(T),(x_\gamma))_{\gamma_\in T}$ is a bijection between the set of triangulations of 
$\bSigma$ and the set of (unlabelled) seeds of $\mathcal{A}(B(T_0),\mathbf{x})$, such that whenever two tagged triangulations are related by the flip of a tagged arc, 
the seeds corresponding to them are related by the mutation of seeds corresponding to the arc being flipped.

Given a surface $\bSigma=\surf$ and an ideal triangulation $T$ of $\bSigma$, Musiker--Schiffler--Williams associated to each plain arc (resp. simple closed curve) $\alpha$ on $\bSigma$, a bipartite graph $G(T,\alpha)$, which they called a \emph{snake graph} (resp. \emph{band graph}) \cite{musiker2011positivity,musiker2013bases}, with edges and \emph{tiles} naturally labelled by the arcs in $T$ (and the boundary segments of $\bSigma$). Later, Wilson extended this construction to all tagged arcs $\alpha$, with respect to any tagged triangulation $T$, introducing the like-minded \emph{loop graph} $G(T,\alpha)$ \cite{wilson2020surface}. Considering a certain subset $\mathcal{P}(G(T,\alpha))$ of the set of all perfect matchings of the loop or band graph $G(T,\alpha)$, called \emph{good} matchings, and using the natural labels mentioned above, they assigned a \emph{weight} monomial $w(P):=x(P)y(P)$ to each $P \in \mathcal{P}(G(T,\alpha))$, thus producing a generating function $\sum_{P}w(P)$ for the good matchings of $G(T,\alpha)$. Dividing this generating function by the monomial $\operatorname{cross}(T,\alpha)$ that records the number of crossings of $\alpha$ with each arc in $T$, they define the \emph{bangle function}
\begin{equation}\label{eq:general-MSW-expression-intro}
\MSW(G(T,\alp)):=\frac{\sum_{P}w(P)}{\operatorname{cross}(T,\alp)}.
\end{equation}
Furthermore, they define the bangle function associated to a lamination 
$L\in\Lamin(\bSig)$ (see~Section~\ref{ssec:KRS-Lamin} for details) to be
\[
\MSW(T, L):=\prod_{\alp\in L} \MSW(G(T,\alp))^{m_{\alpha}},
\]
where $m_\alpha$ is the multiplicity of the single laminate $\alpha$ as a member of~$L$. The set of bangle functions is
$$
\mathcal{B}^\circ(\bSigma, T):=\{\MSW(T,L)\suchthat \text{$L$ is a lamination of $\bSigma$}\}.
$$

Musiker--Schiffler--Williams show in \cite{musiker2013bases} that if the underlying surface $\surf$ does not have punctures and has at least two marked points (i.e., $\punct=\varnothing$ and $|\marked|\geq 2$), then for the coefficient-free cluster algebra $\cA(\bSigma)$ associated to $\surf$, and for any cluster algebra associated to $\surf$ under full-rank extended exchange matrices {\tiny $\left[\begin{array}{c}B(T)\\ B'\end{array}\right]$}, that (after appropriate specialisations in the sense of Fomin-Zelevinsky's separation of additions formula [Theorem 3.7, \cite{fomin2007cluster}]) both the set of bangle functions and the set of bracelet functions form bases of the associated cluster algebra, each containing the cluster monomials by \cite{musiker2011positivity}$^{\dagger}$. They conjecture in \cite{musiker2013bases} that $\mathcal{B}^{\circ}$ 

forms a basis also in the punctured case.

\thispagestyle{fockgonchorov}

\subsection{Generic Caldero-Chapoton functions.} \label{ssec:gen-CC}
Let $Q$ be a 2-acyclic quiver with vertex set 
$Q_0=\{1,2,\ldots, n\}$,
and  $B(Q)=(b_{ij})\in\ZZ^{Q_0\times Q_0}$ 
the corresponding skew-symmetric matrix with entries
\begin{equation} \label{eq:B(Q)}
b_{ij}=
\abs{\{\alp\in Q_1\mid \alp:j\rightarrow i \ \text{in} \ Q\}} -\abs{\{ \bet\in Q_0\mid \bet:i\rightarrow j \ \text{in} \ Q\}}.
\end{equation}
as in~\cite[Equation (1.4)]{derksen2010quivers}. In this context we abbreviate $\cA(Q):= \cA(B(Q))$ for the corresponding
cluster algebra with trivial coefficients.

Recall that for each basic $\CC$-algebra $A=\CC\ebrace{Q}/I$
and a representation $M$ of $A$ we have the \emph{$F$-polynomial}

\[
F_M(y_1,\ldots y_n) := 
\sum_{\be\leq\dimv(M)} \chi(\Gr^A_\be(M)) \by^{\be}
\in\ZZ[y_1,\ldots y_n]
\]

where $\by^{\be}:= \displaystyle \prod_{i=1}^n y_i^{e_i}$ and $\chi(\Gr^A_\be(M))$ is the topological Euler characteristic 
of the quiver Grassmannian of subrepresentations of $M$ with
dimension vector $\be$.

Building on this, one may define the \emph{Caldero--Chapoton} 
function
\[
CC_A(M):= \bx^{\bg(M)} F_M(\hat{y}_1,\ldots \hat{y}_n)
\in\ZZ[x^{\pm 1}_1,\ldots x^{\pm 1}_n]
\]
where $\hat{y}_j:= \displaystyle \prod_{i=1}^n x_i^{b_{ij}}$ for each $j \in Q_0$ and $\bg(M)\in\ZZ^{Q_0}$ is the injective g-vector of $M$.

Together with Leclerc and Schr\"{o}er, the first named author explored in \cite{geiss2012generic} the algebraic geometry of (affine) varieties of representations of Jacobi-finite non-degenerate quivers with potential $(Q,S)$ which appear in algebraic Lie theory.  They isolated a class of irreducible components of these representation spaces with good geometric-homological properties and called them \emph{strongly reduced} irreducible components. Later on, the adjective used was changed to \emph{generically $\tau^-$}-reduced, or simply, \emph{$\tau^-$-reduced}, 
cf.~\cite{geiss2022schemes}. Schröer and Bobiński made in a recent preprint~\cite{bobiński2025genericallytauregularirreduciblecomponents} a very good point that these
components should in fact be called \emph{$\tau^-$-regular components}. We decided to use this new terminology.
Roughly speaking, a component $Z\subseteq \operatorname{rep}(\mathcal{P}(Q,S),\mathbf{d})$ is $\tau^-$-regular if the codimension in $Z$ of a top dimensional $\operatorname{GL}_{\mathbf{d}}(\CC)$-orbit $\mathcal{O}\subseteq Z$ is equal to the minimum value taken on $Z$ by Derksen-Weyman-Zelevinsky's  (injective) $E$-invariant. See~\cite[Sec.~3 and 5]{cerulli2015caldero-chapoton} for more details.

It is easy to see that each $\tau^-$-regular irreducible component $Z$ has an open dense subset where the  constructible function
$M\mapsto CC_{\mathcal{P}(Q,S)}(M)$
takes a constant value, which is then denoted $CC_{\mathcal{P}(Q,S)}(Z)$. Resorting to decorated representations and spaces of decorated representations in order to be able to hit initial cluster variables through the Caldero-Chapoton function, the set of \emph{generic Caldero-Chapoton functions} is
\[
\mathcal{B}_{\mathcal{P}(Q,S)}:=\{CC_{\mathcal{P}(Q,S)}(Z)\suchthat \text{$Z$ is a decorated $\tau^-$-regular irreducible component}\}.
\]
It is in fact quite easy to see with the help of~\cite{derksen2010quivers} that for each Jacobi-finite non-degenerate quiver with potential $(Q,S)$ the set 
$\mathcal{B}_{\mathcal{P}(Q,S)}$ contains all cluster monomials of the cluster algebra $\cA(Q)$, and that all elements of $\mathcal{B}_{\mathcal{P}(Q,S)}$ belong to
the upper cluster algebra $\cU(Q)$ corresponding  to $Q$.
If the set $\mathcal{B}_{\mathcal{P}(Q,S)}$
is a basis of the (coefficient-free) upper cluster algebra then it is called the \emph{generic basis} of $\cU(Q)$.  

One of the main results of~\cite{geiss2012generic} is, 
that in the setting of unipotent cells for symmetric Kac-Moody groups, the generic Caldero--Chapoton
functions can be identified with the dual of Lusztig's semicanonical basis~\cite{lusztig2000semicanonical},
and that this basis contains in particular all cluster monomials.  
This led the authors of~\cite{geiss2012generic} to conjecture that the set $\mathcal{B}_{\mathcal{P}(Q,S)}$ might
be a basis of the upper cluster algebra in more situations beyond the algebraic Lie theoretic context.

The set $\mathcal{B}_{\mathcal{P}(Q,S)}$ does not change (up to cluster automorphism) under mutations as a subset of $\mathcal{U}(Q)$. This was shown by Plamondon when $(Q,S)$ is Jacobi-finite non-degenerate \cite{plamondon2013generic}, and by Schröer and the first two authors when $(Q,S)$ is arbitrary non-degenerate \cite{geiss2020generic}. 

\subsection{Statement of the Main Result}
Let $\bSig$ be a surface with marked points.  To each tagged 
triangulation $T$ of $\bSig$ we associate a quiver $Q'(T)$ such
that $B(Q'(T))=B(T)$.  This convention allows us to state and
prove our results without the roundabout with dual CC-functions
as for example in~\cite{geiss2022schemes}. 
Schröer and the first two named authors showed~\cite[Thm.~1.18]{geiss2020generic}, relying heavily on results of Mills \cite{mills2017maximal}, Muller \cite{muller2013locally,muller2014A=U} and Qin \cite{qin2019bases}, that when $(Q,S)=(Q'(T),S'(T))$ for some triangulation $T$ of a surface $\bSigma=\surf$ with at least two marked points on the boundary, the set of generic CC-functions
$\cB_{\cP(Q,S)}$ is indeed a basis for the (upper) cluster algebra $\cA(\bSigma)=\cU(\bSigma)$,
where $S'(T)$ is the potential defined by Cerulli Irelli and the second author in \cite{cerulli2012quivers}.
Let us write from now on $A(T)$ for the corresponding
complex Jacobian algebra $\cP_\CC(Q'(T), W'(T))$.
Note that due to our convention here, the algebra $A(T)$
is opposite to the one which was considered for example in~\cite{geiss2023laminations}.  We are now ready to state our main result, Theorem~\ref{thm:bangle-equals-generic-basis}, in a slightly
simplified version.

\begin{thm}\label{thm:bangle-equals-generic-basis-simp}
Let $\bSigma=\surf$ be a surface with marked points such that $|\marked|\geq 2$. For each tagged triangulation $T$ of $\bSigma$, the set $\mathcal{B}^\circ(\bSigma, T)$ of coefficient-free bangle functions is equal to the generic basis  
$\cB_{A(T)}$ of the coefficient-free (upper) cluster algebra $\mathcal{A}(\bSigma)=\mathcal{U}(\bSigma)$.
In particular, the set $\cB^\circ(\bSig, T)$ is a basis
of the cluster algebra.
\end{thm}

\begin{rmk} 

(1) The proof of our theorem  relies heavily on \cite{geiss2023laminations} and our Combinatorial Key Lemma \ref{thm: comb key lemma}.  These arguments imply 
$\cB^\circ(\bSig, T)=\cB_{A(T)}$.  
See Section~\ref{ssec:struct} for more details. 
By the above mentioned result from ~\cite[Thm.~1.18]{geiss2020generic}, the set $\cB^\circ(\bSig, T)$ is indeed a basis of the upper cluster algebra.

(2) As we have mentioned already, Musiker--Schiffler--Williams showed in~\cite{musiker2013bases} that in the special case
$\Pu=\emptyset$ the bangle functions form a basis of $\cA(\bSig)$,
and they conjectured in Appendix A of~\emph{loc.~cit.}\ that the
restriction to $\Pu=\emptyset$ can be removed.

(3) Schröer and the first two named authors showed
in~\cite[Thm.~11.9]{geiss2022schemes} this theorem for the
special case $\Pu=\emptyset$. Their proof relies heavily on the
fact that for $\Pu=\emptyset$ the Jacobian algebra $A(T)$ is
gentle, which allows for quite direct calculations.  In our more
general situation the algebras $A(T)$ are still tame~\cite{geiss2016the},  but the description of their 
indecomposable representations is in general unknown.  
See for example~\cite{domínguez2017arc} for partial results into
this direction.  Our approach here exploits the use of mutations in the spirit of~\cite{derksen2010quivers}.

(4) Gross--Hacking--Keel--Kontsevich construct in their ground breaking paper~\cite{gross2018canonical} a \emph{theta basis}
for a large class of (upper) cluster algebras with skew-symmetrizable exchange matrix $B$.  This basis has remarkable
positivity properties. 

Recently, Mandel--Qin \cite{mandel2023bracelets} have shown that when $B=B(T)$ for some triangulation of a surface $\bSigma=\surf$, Gross--Hacking--Keel--Kontsevich's \emph{theta basis}
is equal to the set of bracelet functions $\mathcal{B}(\bSigma)$
from~\cite{musiker2013bases}, thus proving, through a combination with \cite{gross2018canonical}, that the set of bracelet functions is indeed a basis. The strategy of their proof is surprisingly similar to ours.
\end{rmk}

\subsection{Structure of the paper} \label{ssec:struct}
The paper is organized as follows. In Section \ref{sec:preliminaries} we provide the basics about mutations and cluster algebras from surfaces. In Section \ref{sec:laminations-as-components} we summarize the main constructions and results from \cite{geiss2023laminations} upon which we will rely in Section \ref{sec:bangle-functions-are-the-generic-basis}, the most relevant one being \cite[Proposition 2.8]{geiss2023laminations} (stated as Proposition \ref{prp:JaMut} below), which gives a precise recursive formula for the change that the generic projective $g$-vector of any $\tau$-regular component undergoes under mutation of $\tau$-regular components.

It should be stressed that, since elements of $\mathcal{B}^{\circ}$ were defined directly from the associated loop or band graphs, then, a priori, this viewpoint says nothing about whether $\mathcal{B}^{\circ}$ is canonical in a cluster-theoretical sense. That is, whether or not $\mathcal{B}^{\circ}$ is independent of the choice of triangulation, up to a cluster isomorphism. 

Throughout this paper, our approach works by flipping the perspective of \cite{musiker2011positivity} — we fix a curve on the surface and alter the underlying triangulations through the process of cluster mutation. Since we wish to work with principal coefficients, extra care is needed here. Indeed, we consequently require a framework that describes how cluster variables transform from one triangulation to the next, with respect to principal coefficients at those triangulations. This is achieved through the following lemma:

\begin{lem}[Combinatorial Key Lemma]
Let $T=\{\tau_1,\ldots,\tau_n\}$ be a tagged triangulation of $\bSig$,  $\tau_k$  a tagged arc belonging to $T$, and $\alpha$  a simple closed curve. We set $T'$ to be the tagged triangulation obtained from $T$ by flipping $\tau_k$. The following equality holds:

\begin{equation}
(y_k +1)^{h_k}F_{G(T,\alpha)}(y_1,\ldots, y_n) = (y_k' +1)^{h_k'}F_{G(T',\alpha)}(y_1',\ldots, y_n')
\end{equation}

where \begin{itemize}

\item $(\mathbf{y}',B') \in \mathbb{Q}(y_1,\ldots, y_n)$ is obtained from $(\mathbf{y},B(T))$ by mutation at $k$,

\item the \emph{snake $\mathbf{h}$-vector} $\mathbf{h}_{G(T,\alpha)}$ is defined by setting
$$u^{h_i} := {F_{G(T,\alpha)}}_{{\vert}_{\operatorname{Trop}(u)}}(u^{[-b_{i1}]_{+}}, \ldots, u^{-1}, \ldots, u^{[-b_{in}]_{+}}),$$ where $u^{-1}$ is in the $i^{th}$ component. The snake $\mathbf{h}$-vector $\mathbf{h}_{G(T',\alpha)}$ is defined analogously with respect to $T'$ and $\mu_{k}(B(T))=B(T')$.

\end{itemize}

\noindent Moreover, the snake $\mathbf{g}$-vector $\mathbf{g}_{G(T,\alpha)} = (g_1,\ldots, g_n)$ satisfies 
\begin{equation}
g_k = h_k - h_k'
\end{equation} 
and $\mathbf{g}_{G(T,\alpha)} = (g_1,\ldots, g_n)$ is related to $\mathbf{g}_{G(T',\alpha)} = (g_1',\ldots, g_n')$ by the following rule:
\begin{equation}
g_j'  := \left\{
\begin{array}{ll}
        -g_k, &\text{if $j =k$}; \\
        \\
        g_j + [b_{jk}]_{+}g_k - b_{jk}h_k, &\text{if $j \neq k$}.\\
\end{array} 
\right.
\end{equation}

\end{lem}

As an immediate corollary, we obtain that the set of bangle functions 
$\mathcal{B}^{\circ}(\bSigma)=\cB^\circ(\bSig, T)$ 
is, up to cluster automorphism, independent of the choice of $T$. Note that a version of this lemma first appeared in the work of Derksen-Weyman-Zelevinsky [\cite{derksen2010quivers}, Lemma 5.2] which was restricted the setting of cluster variables — our result shows this lemma holds on all elements of the basis $\mathcal{B}^{\circ}$.

In Section \ref{base case section} we show that for simple closed curve $\alpha$, there exists a tagged $T_\alpha$ for which one can explicitly compute a (decorated) representation $\mathcal{M}(T_\alpha,\alpha)$ such that 
\begin{itemize}
\item the projective $g$-vector
$\bg_{A(T_\alpha)}(\mathcal{M}(T_\alpha,\alpha))$ 
is equal to the vector of dual shear coordinates $\Sh_{T_\alpha}(\alpha)$ and to the snake $g$-vector
$\mathbf{g}_{G(T_\alpha,\alpha)}$ coming from the snake or band graph $G(T_\alpha,\alpha)$;
\item the representation-theoretic (dual) $F$-polynomial $F_{M(T_\alpha,\alpha)}$ is equal to the snake $F$-polynomial $F_{G(T_\alpha,\alpha)}$ coming from the snake or band graph $G(T_\alpha,\alpha)$;
\item the orbit closure of $\mathcal{M}(T_\alpha,\alpha)$ is a $\tau$-regular component $Z_{T_\alpha,\alpha}$ and any decorated representation 
mutation-equivalent to $\mathcal{M}(T_\alpha,\alpha)$ is generic, for the values of the projective $g$-vector, the representation-theoretic (dual) $F$-polynomial and Derksen-Weyman-Zelevinsky's $E$-invariant, in the corresponding $\tau$-regular component.
\end{itemize}

In Section \ref{CKL} we show that every time a flip of tagged triangulations is applied, the changes undergone by the snake $g$-vectors and snake $F$-polynomials of the band graphs constructed from $\alpha$, obey the same recursive formulas that govern the changes undergone by the projective $g$-vectors and representation-theoretic (dual) $F$-polynomials. In Section \ref{sec:bangle-functions-are-the-generic-basis} we combine the results from Sections \ref{base case section} and \ref{CKL} to deduce our main result. In Section \ref{sec:remarks-and-problems} we make some remarks about the extent to which the representations that yield the bangle functions have been really explicitly computed, and state a couple of open problems.

\section{Preliminaries}\label{sec:preliminaries}

\subsection{Cluster algebras}
This section provides a brief review of (skew-symmetric) cluster algebras of geometric type. Let $n \leq m$ be positive integers. Furthermore, let $\mathcal{F}$ be the field of rational functions in $m$ independent variables. Fix a collection $X_1,\ldots,X_n, x_{n+1},\ldots, x_m$ of algebraically independent variables in $\mathcal{F}$. We define the \emph{coefficient ring} to be $\mathbb{ZP}:=\mathbb{Z}[x^{\pm 1}_{n+1}\ldots x_m^{\pm 1}]$.

\begin{defn}
A \emph{(labelled) seed} consists of a pair, $(\mathbf{x},\mathbf{y}, B)$, where
\begin{itemize}
\item 
$\mathbf{x} = (x_1,\ldots, x_n)$ is a collection of variables in $\mathcal{F}$ which are algebraically independent over $\mathbb{ZP}$,
\item 
$\mathbf{y} = (y_1,\ldots y_n)$ where $y_k = \displaystyle \prod_{j=n+1}^{m} x_j^{b_{jk}}$ for some $b_{jk} \in \mathbb{Z}$,
\item 
$B = (b_{jk})_{j,k \in \{1,\ldots, n\}}$ is an $n \times n$ skew-symmetric integer matrix.
\end{itemize}
The variables in any seed are called \emph{cluster variables}. The variables $x_{n+1},\ldots,x_m$ are called \emph{frozen variables}. We refer to $\mathbf{y}$ as the \emph{choice of coefficients}.
\end{defn}

\begin{defn}
Let $(\mathbf{x},\mathbf{y},B)$ be a seed and let $i \in \{1,\ldots,n\}$. 

We define a new seed $\mu_{i}(\mathbf{x},\mathbf{y},B) := (\mathbf{x}',\mathbf{y}',B')$, called the \emph{mutation} of $(\mathbf{x},\mathbf{y},B)$ at $i$ where:
\begin{itemize}
\item 
$\mathbf{x}' = (x'_1,\ldots, x'_n)$ is defined by 
\[
x'_i = \frac{\displaystyle \prod_{b_{ki} >0} x_k^{b_{ki}} + \prod_{b_{ki} <0} x_k^{-b_{ki}}}{x_i}
\]
and setting $x_j' = x_j$ when $j \neq i$;
\item 
$\mathbf{y}'$ and $B' = (b'_{jk})$ are defined by the following rule:
\[   
b'_{jk} = 
     \begin{cases}
       -b_{jk},& \text{if $j=i$ or $k=i$,}\\
       b_{jk} + \max(0,-b_{ji})b_{ik} + \max(0,b_{ik})b_{ji},& \text{otherwise.}\\
     \end{cases}
\]
\end{itemize}
\end{defn}
A \emph{quiver} is a finite directed (multi) graph $Q= (Q_0,Q_1)$ where $Q_0$ is the set of vertices and $Q_1$ is the set of directed edges. It will often be convenient to encode (extended) skew-symmetric matrices as quivers. We describe this simple relationship in the definition below, which follows the conventions set out in the work of Derksen, Weyman and Zelevinsky~\cite{derksen2010quivers}.
\begin{defn}
Given a skew-symmetric $n\times n$ matrix $B$ we define a quiver $Q(B)= (Q_0, Q_1)$ by setting $Q_0 =\{1,\ldots, n\}$ and demanding that for any $i,j \in Q_0$ there are $[b_{ij}]_{+}$ arrows from $j$ to $i$ in $Q_1$.

We generalise the definition to extended $m\times n$ skew-symmetric matrices $\tilde{B}$ in the obvious way.
\end{defn}

\begin{defn}
\label{defn: cluster algebra}
Fix an  $(\mathbf{x},\mathbf{y}, B)$. If we label the {initial cluster variables} of $ \mathbf{x}$ from $1,\ldots, n $ then we may consider the labelled n-regular tree $\mathbb{T}_n $. Each vertex in $\mathbb{T}_n$ has $n$ incident edges labelled $1, \ldots, n$. Vertices of $\mathbb{T}_n$ represent seeds and the edges correspond to mutation. In particular, the label of the edge indicates which direction the seed is being mutated in. 

Let $\mathcal{X} $ be the set of all cluster variables appearing in the seeds of $ \mathbb{T}_n$. The \emph{cluster algebra} of the seed $(\mathbf{x},\mathbf{y}, B)$ is defined as $\mathcal{A}(\mathbf{x},\mathbf{y}, B ) := \mathbb{ZP}[\mathcal{X}]$. 

We say $\mathcal{A}(\mathbf{x},\mathbf{y}, B)$ is the \emph{cluster algebra with principal coefficients} if $m = 2n$ and $\mathbf{y} = (y_1, \ldots, y_n)$ satisfies $y_k = x_{n+k}$ for any $k \in \{1,\ldots, n\}$.
\end{defn}

\subsection{Cluster algebras from surfaces}
\label{Cluster algebras from surfaces}

In this subsection we recall the work of Fomin, Shapiro and Thurston \cite{fomin2008cluster}, which establishes a cluster structure for triangulated orientable surfaces. 

\begin{defn} \label{def:surface}
A \emph{surface with marked points} is a triple $\bSigma=(\Sigma,\marked,\punct)$ consisting of a compact, connected, oriented two-dimensional real differentiable manifold $\Sigma$ with (possibly empty) boundary $\partial\Sigma$, a finite set $\marked\subseteq\partial\Sigma$ containing at least one point from each boundary component of  $\Sigma$, and a finite set $\punct\subseteq\Sigma\setminus\partial\Sigma$. 
It is required that $\marked\cup\punct\neq\varnothing$. 
The elements of $\marked\cup\punct$ are called \emph{marked points} and the elements of $\punct$ are called \emph{punctures}. 
If $\punct=\varnothing$, it is said that $\bSigma$ is \emph{unpunctured}, whereas if $\punct\neq\varnothing$, one says that $\bSigma$ is punctured.
For technical reasons we exclude the cases where $\bSigma$ is an unpunctured or once-punctured monogon; a digon; a triangle; or a once, twice or thrice punctured sphere.  
In~\cite[Def.~2.1]{fomin2008cluster} the term \emph{bordered surfaces with marked points} was used instead of the more recent
terminology~\emph{surface with marked points} or even~\emph{marked surface}.  
\end{defn}

\begin{defn}
An \emph{arc} of $\bSig$ is a simple curve in $\Sig$ connecting two marked points of $\Ma\cup\Pu$, which is not homotopic  
to a boundary segment or a marked point.  We consider arcs up to homotopy (relative to its endpoints) and up to inversion of the orientation. 

A \emph{tagged arc} $\gamma$ is an arc whose endpoints have been `tagged' in one of two ways; \emph{plain} or \emph{notched}. Moreover, this tagging must satisfy the following conditions: if the endpoints of $\gamma$ share a common marked point, they must receive the same tagging; and an endpoint of $\gamma$ lying on the boundary $\partial S$ must always receive a plain tagging. 
In this paper we shall always consider tagged arcs up to the equivalence induced
from the equivalence relation for plain arcs. 
\end{defn}

\begin{defn}
Let $\alpha$ and $\beta$ be two tagged arcs of $\bSig$. We say $\alpha$ and $\beta$ are \emph{compatible} if and only if the following conditions are satisfied:

\begin{itemize}
\item There exist homotopic representatives of $\alpha $ and $\beta$ that don't intersect in the interior of $\Sig$. 

\item Suppose the untagged versions of $ \alpha$ and $\beta $ do not coincide. 
If $\alpha$ and $\beta$ share an endpoint $p$ then the ends of $\alpha$ and 
$\beta$ at $p$ must be tagged in the same way.

\item Suppose the untagged versions of $\alpha$ and $\beta$ do coincide. Then precisely one end of $\alpha$ must be tagged in the same way as the corresponding end of $\beta$.
\end{itemize}

A \emph{tagged triangulation} of $\bSig$ is a maximal collection of pairwise compatible tagged arcs of $\bSig$. Moreover, this collection is forbidden to contain any tagged arc that enclose a once-punctured monogon.

An \emph{ideal triangulation} of $\bSig$ is a maximal collection of pairwise compatible plain arcs. Note that ideal triangulations decompose $\bSigma$ into triangles, however, the sides of these triangles may not be distinct; two sides of the same triangle may be glued together, resulting in a \emph{self-folded triangle}.
\end{defn}

\begin{rmk}
To each tagged triangulation $T$ we may uniquely assign an ideal triangulation $T^{\circ}$ as follows:

\begin{itemize}
\item  
If $p$ is a puncture with more than one incident notch, then replace all these notches with plain taggings.
\item  
If $ p$ is a puncture with precisely one incident notch, and this notch belongs to $\beta \in T$, then replace $\beta$ with the unique arc $\gamma$ of $\bSig$ which encloses $\beta$ and $p$ in a monogon.
\end{itemize}

Conversely, to each ideal triangulation $T$ we may uniquely assign a tagged triangulation $ \iota(T)$ by reversing the second procedure described above.
\end{rmk}

\begin{defn}
\label{adjacency matrix}
Let $T$ be a tagged triangulation, and consider its associated ideal triangulation $ T^{\circ}$. We may label the arcs of $T^{\circ} $ from $ 1,\ldots, n$ (note this also induces a natural labelling of the arcs in $ T$). We define a function, $ \pi_{T} : \{1,\ldots,n\} \rightarrow \{1,\ldots,n\} $, on this labelling as follows:
\[   
\pi_{T}(i) = 
     \begin{cases}
       j & \text{if $ i $ is the folded side of a self-folded triangle in $ T^{\circ}$,}\\
       &\text{and $j$ is the remaining side;}\\
       i & \text{otherwise.}\\
     \end{cases}
\]

For each non-self-folded triangle $\Delta$ in $T^{\circ}$, as an intermediary step, define the matrix $ B_T^{\Delta} = (b^{\Delta}_{jk}) $ by setting 
\[   
b^{\Delta}_{jk} = 
     \begin{cases}
       1 & \text{if $\Delta$ has sides $ \pi_{T}(j)$ and $ \pi_{T}(k)$, and $ \pi_{T}(k) $ follows $ \pi_{T}(j) $}\\
       & \text{in the clockwise sense;}\\
       -1 & \text{if $ \Delta$ has sides $ \pi_{T}(j)$ and $\pi_{T}(k) $, and $ \pi_{T}(k) $ precedes $\pi_{T}(j)$}\\
       & \text{in the clockwise sense;}\\
       0 & \text{otherwise.}\\
     \end{cases}
\]
The \emph{adjacency matrix} $B(T) = (b_{ij})$ of $T$ is then defined to be the following summation, taken over all non-self-folded triangles $\Delta$ in $T^{\circ}$:
\[
B(T):=  \sum\limits_{\Delta} B_T^{\Delta}
\]
\end{defn}

\begin{defn}
Let $T$ be a tagged triangulation of a bordered surface $\bSig$. Consider the initial seed $ (\mathbf{x},\mathbf{y},B_T) $, where: $\mathbf{x} $ contains a cluster variable for each arc in $T$; $B(T)$ is the matrix defined in Definition \ref{adjacency matrix}; and $\mathbf{y}$ is any choice of coefficients. We call $\mathcal{A}(\mathbf{x}, \mathbf{y}, B(T))$ a \emph{surface cluster algebra}. 
\end{defn}

\begin{prop}[Theorem 7.9, \cite{fomin2008cluster}]

Let $T $ be a tagged triangulation of a bordered surface $\bSig$. Then for any $\gamma \in T$ there exists a unique tagged arc $\gamma'$ on $\bSig$ such that $f_{\gamma}(T) := (T\setminus \{\gamma\}) \cup \{\gamma'\}$ is a tagged triangulation. We call $f_{\gamma}(T)$ the \emph{flip} of $T$ with respect to $\gamma$.

\end{prop}

\begin{thm}[Theorem 6.1, \cite{fomin2008cluster}]
\label{surface correspondence}
Let $\bSig$ be a  surface  with marked points. If $\bSig$ is not a once punctured closed surface, then in the cluster algebra $\mathcal{A}(\mathbf{x}, \mathbf{y}, B(T))$, the following correspondence holds:
\begin{align*}
 &\hspace{23mm} \mathcal{A}(\bx, \by, B(T)) & &  &\bSig \hspace{32mm}&   \\ 
 &\hspace{18mm} \emph{Cluster variables}  &\longleftrightarrow&  &\emph{Tagged arcs} \hspace{23mm} &  \\
 &\hspace{25mm}\emph{Clusters}   &\longleftrightarrow&  &\emph{Tagged triangulations} \hspace{14mm}&  \\
 &\hspace{24mm} \emph{Mutation}    &\longleftrightarrow&   &\emph{Flips of tagged arcs} \hspace{16mm}&  \\
\end{align*}

When $\bSig$ is a once-punctured closed surface then cluster variables are in bijection with all plain arcs or all notched arcs depending on whether $T$ consists solely of plain arcs or notched arcs, respectively.
\end{thm}


\section{Laminations as \texorpdfstring{$\tau$}{τ}-regular components}
\label{sec:laminations-as-components}

\subsection{Tame partial KRS-monoids}

The simple-minded notion of \emph{partial Krull-Remak-Schmidt monoid} was introduced in \cite{geiss2023laminations} as a convenient abstract framework to state the naturalness of the bijection between laminations and $\tau$-regular components proved therein.

\begin{defn}\cite[Definition 2.1]{geiss2023laminations} A \emph{partial monoid} is a triple $(X, e, \oplus)$ consisting of a set
  $X$, a symmetric function $e: X\times X\ra\NN$ and a partially-defined sum
  $\oplus: \{(x,y)\in X\times X\mid e(x,y)=0\}\ra X$ such that:
  \begin{itemize}
  \item[(s)] if $e(x,y)=0$ we have $x\oplus y =y\oplus x$;
  \item[($0$)] there exists a unique element $0\in X$ with $e(0,x)=0$
    and  $0\oplus x=x$ for all $x\in X$;
  \item[(d)] if $e(y,z)=0$ we have $e(x, y\oplus z)=e(x,y)+e(x,z)$ for all
    $x\in Z$;
  \item[(a)] $(x\oplus y)\oplus z= z\oplus (y\oplus z)$ whenever
    one side of the equation is defined.
  \end{itemize}  
A \emph{morphism} of partial monoids from $X=(X, e,\oplus)$
to $X'=(X',e',\oplus')$ is a function $f: X\ra Y$, such that
$e'(f(x), f(y))=e(x,y)$ for all $x,y\in X$ and
$f(x\oplus y)= f(x)\oplus' f(y)$ whenever $e(x,y)=0$.
\end{defn}

\begin{rmk} (1) Suppose $(x_1\oplus x_2)\oplus x_3$ is defined. By $(d)$, we have $e(x_i, x_j)=0$ for $1\leq i<j\leq 3$, hence
  $x_1\oplus (x_2\oplus x_3)$ is defined.

  (2)  Suppose that we have $x_1, x_2, \ldots, x_n\in X$ with $e(x_i, x_j)=0$ for
  all $i<j$,  then $x_1\oplus x_2\oplus\cdots\oplus x_n\in X$ is well-defined,
  and  for each permutation $\sig\in\mathfrak{S}_n$ we have
  \[
    x_1\oplus x_2\oplus\cdots\oplus x_n=
    x_{\sig(1)}\oplus x_{\sig(2)}\oplus\cdots\oplus x_{\sig(n)}.
  \]
\end{rmk}  

\begin{defn}\cite[Definition 2.3]{geiss2023laminations} \label{def:KRS2}
Suppose $X=(X, e, \oplus)$ is a partial monoid.
\begin{itemize}
\item
The elements of the set
  \[
    X_\ind:=\{x\in X\setminus\{0\}
    \mid x= y\oplus z \text{ implies } y=0 \text{ or } z=0\}
  \]
are called \emph{indecomposable elements} of $X$.
\item
We say that $X$ is a \emph{partial KRS-monoid}
if every $x\in X$ is equal to a finite direct sum of indecomposable elements, and
whenever
\[
  x_1\oplus x_2\oplus\cdots\oplus x_m= y_1\oplus y_2\oplus\cdots\oplus y_n
\]
with $x_1, \ldots, x_m, y_1, \ldots, y_m\in X_\ind$, necessarily $m=n$ and, moreover, there exists
a permutation $\sig\in\mathfrak{S}_n$ such that $y_i=x_{\sig(i)}$ for all
$i=1,2,\ldots, n$.
\item
We say that $X$ is \emph{tame} if $e(x,x)=0$ for all $x\in X$.
\item A \emph{framing} for $X$ is
 a map $\bg: X\ra\ZZ^n$ (for some non-negative integer $n$) such that
$\bg(x\oplus y)=\bg(x)+\bg(y)$ for all
$x, y\in C$ with $e(x,y)=0$. A framing is \emph{faithful} if it is an injective function.
\item
A framed partial monoid $X=(X, e, \oplus,\bg)$ is \emph{free of rank} $n$ if the framing
$\bg: X\ra\ZZ^n$ is bijective.  
\end{itemize}  
\end{defn}

\begin{exmp}\cite[Example 2.4]{geiss2023laminations}\label{expl:KRS1}
Suppose $C$ is a set equipped and $e: C\times C\ra\NN$ is a function
such that $e(c,c)=0$ for all $c\in C$. Then
\[
  \KRS(C,e):=\{f: C\ra\NN\mid c_1, c_2\in\supp(f)\Rightarrow e(c_1, c_2)=0
  \text{ and } \abs{\supp(f)}<\infty \}
\]
is a tame partial KRS-monoid with
\[
e(f,g):= \sum_{c,d\in C} f(c)\cdot g(d)\cdot e(c,d),
\qquad
\text{and}
\qquad
  (f\oplus g)(c):= f(c)+g(c).
\]
Notice that
\[
  \KRS(C,e)_\ind=\{\del_c:C\ra\NN\mid c\in C\},\quad\text{where}\quad
  \del_c(d):=\del_{c,d} \quad\text{(Kronecker's \emph{delta})}.
\]
With this notation we have $f=\oplus_{c\in C} \del_c^{\oplus f(c)}$ for all
$f\in\KRS(C,e)$.
\end{exmp}

\begin{rmk}
For a tame partial KRS-monoid $X=(X, e, \oplus)$
we have $X\cong\KRS(X_\ind, e')$, where $e'$ is the restriction of
$e$ to $X_\ind\times X_\ind$. 
In this situation, any map $\bg: C\ra\ZZ^n$ can be extended to a framing
$\bg:\KRS(C, e)\ra\ZZ^n$ by simply setting
$\bg(f):=\sum_{c\in C} f(c)\bg(c)$. 
\end{rmk}

\subsection{KRS-monoids of \texorpdfstring{$\tau^-$}{τ-}-regular components} \label{ssec:FD-KRS}
Let $A$ be a basic finite dimensional algebra over an algebraically closed field $\field$. We may identify the Grothendieck group $K_0(A)$ of $A$ with $\ZZ^n$ for some non-negative integer $n$. 
Consider the set $\DecIrr^{\tau^-}(A)$ of decorated, generically $\tau^-$-regular,  irreducible components of the representation varieties of $A$. See Section~\ref{ssec:gen-CC} 
for the terminology, and~\cite[Sec.~3 and~5]{cerulli2015caldero-chapoton}  for more details.  The set $\DecIrr^{\tau^-}(A)$ becomes a partial KRS-monoid
under the function 
\[
e_A: \DecIrr^{\tau^-}(A)\times \DecIrr^{\tau^-}(A)\ra\NN
\]
that to each pair $(X,Y)\in \DecIrr^\tau(A)\times \DecIrr^\tau(A)$ associates the generic value $e_A(X,Y)$ that on the Zariski product $X\times Y$ takes the symmetrized (injective)
$E$-invariant
\begin{multline}\label{eq:def-of-sym-proj-E-inv}
E_A((M,\bv),(N,\bw)):=\\
\dim\Hom_A(\tau^{-1}M, N)+\dim\Hom_A(\tau^{-1}N, M)+ 
\dimv(M)\cdot\bw + \bv\cdot\dimv(N);
\end{multline}
and the partially-defined sum
\[
\oplus: \{(X,Y)\in \DecIrr^{\tau^-}(A)\times \DecIrr^{\tau^-}(A)\mid e(x,y)=0\}\ra \DecIrr^{\tau^-}(A)
\]
given by the direct sum of $e_A$-orthogonal irreducible components, $ X\oplus Y:=\overline{X\oplus Y}$.

The function $\bg_A:\DecIrr^{\tau^-}(A)\ra\ZZ^n$ that to each $X\in \DecIrr^\tau(A)$ associates the generic value taken on $X$ by the (injective)
$g$-vector 
\[
\bg_A(M,\bv):= (\dim\Hom(\tau^{-1} M, S_i) - \dim(S_i, M)+ v_i)_{i\in Q_0}
\]
is a framing for the partial KRS-monoid $(\DecIrr^{\tau^-}(A),e_A,\oplus)$.
More precisely:  

\begin{thm} \label{thm:tau-red-comps-form-a-KRS-monoid}
  Let $A$ a finite-dimensional basic algebra over an algebraically closed field. Let $n$ be the rank of the Grothendieck group of $A$.
\begin{itemize}
\item[(a)]     
  $\DecIrr^{\tau^-}(A)=(\DecIrr^{\tau^-}(A), e_A, \oplus, \bg_A)$ is a framed, free KRS-monoid of rank $n$.  The subset $\DecIrr^\tau_\ind(A)$ of components,
  which contain a dense set of indecomposable representations is precisely
  the set of indecomposable elements in the sense of Definition~\ref{def:KRS2}.
\item[(b)]
  With the framing from the generic (injective) $g$-vector, $\DecIrr^{\tau^-}(A)$ is  a free partial KRS-monoid of rank $n$.  
\item[(c)] 
  If $A$ is of tame representation type, then $\DecIrr^{\tau^-}(A)$ is tame in the sense of Definition~\ref{def:KRS2}. Consequently, there is an isomorphism of partial KRS-monoids
\[
  \DecIrr^{\tau^-}(A)\cong \KRS(\DecIrr^{\tau^-}_\ind(A), e_A).  
\]
Moreover in this case each $Z\in\DecIrr^\tau_\ind(A)$ contains either a
dense orbit, or a one-parameter family of bricks.  
\end{itemize}
\end{thm}  

\begin{rmk}\label{rem:about-KRS-monoids-of-tau-red-comps}
\begin{enumerate}
\item 
In the form it is stated, Theorem~\ref{thm:tau-red-comps-form-a-KRS-monoid} appeared (up to duality) in \cite{geiss2023laminations}. Its assertions are restatements of previous results by various authors:
Part~(a) is a well-known combination of~\cite[Theorem~1.2]{crawley2002irreducible}
and~\cite[Theorems~1.3 and 1.5]{cerulli2015caldero-chapoton}.
Part~(b) is a theorem of Plamondon~\cite{plamondon2013generic}.  
Part~(c) is~\cite[Corollary 1.7]{geiss2023semicontinuous} and \cite[Theorem~3.2]{geiss2022schemes}.
\item 
More precisely, Theorem~\ref{thm:tau-red-comps-form-a-KRS-monoid} is just
the dual version of~\cite[Thm.~2.5]{geiss2023laminations}.
There, the result was stated in terms of generically $\tau$-regular
components with the function $e_A^\prj$ which is defined in terms
of the projective symmetrized E-invariant 
\begin{multline*}
\qquad\quad E_A^\prj((M,\bv), (N,\bw)):=\\
\dim\Hom_A(M, \tau N) + \dim\Hom_A(N, \tau M) 
 +\bv\cdot\dimv(N)+ \dimv(M)\cdot\bw,
\end{multline*}
with framing defined by the generic values of the projective $g$-vector
\[
\bg_A^\prj(M,\bv):= (\dim\Hom_A(S_i, \tau M, S_i)-\dim\Hom_A(M, S_i) + v_i)_{i\in Q_0}
\]
Duality over the ground field $\field$ induces a natural isomorphism of framed partial KRS-monoids 
\[
(\DecIrr^{\tau^-}(A), e_{A}, \oplus, \bg_{A})\cong (\DecIrr^\tau(A^\op), e^\prj_{A^\op}, \oplus, \bg^\prj).
\]
\end{enumerate}
\end{rmk}

Suppose now that the finite-dimensional $\field$-algebra $A$ is the Jacobian algebra of a non-degenerate quiver with potential $(Q,W)$. 
Following~\cite[Section~10]{derksen2008quivers}, for each representation $M$ of 
$A=\cP_{\field}(Q, W)$ and each vertex $k\in Q_0$ we consider the triangle of linear maps
\[\xymatrix{
& M(k) \ar[rd]^{M(\beta_k)}\\
M_{\operatorname{in}}(k)\ar[ru]^{M(\alpha_k)}&&\ar[ll]^{M(\gamma_k)} M_{\operatorname{out}}(k)
}.
\]
In view of~\cite[Proposition~10.4 and Remark~10.8]{derksen2010quivers}, we have 

\begin{align} \label{eq:gvec}
\bg_A(M)= (\dim\Ker(M(\gamma_k))-\dim M(k))_{k\in Q_0}
=:\bg_{A}^{\mathrm{DWZ}}(M,\mathbf{0}).
\end{align}
Here $\bg_{A}^{\mathrm{DWZ}}(M)$ denotes the $g$-vector of
a QP-representation in the sense of~\cite[(1.13)]{derksen2010quivers}.  This can be easily extended to
decorated QP-representations.

Now, write $A':=\cP_{\field}(\mu_k(Q,W))$, where $\mu_k(Q,W)$ is the QP-mutation defined in \cite{derksen2008quivers}, and  
consider the piecewise linear transformation of integer vectors 
\begin{equation} \label{eq:gproj-transf}
\gamma_k^B: \ZZ^{Q_0}\ra \ZZ^{Q_0}\quad\text{with}\quad 
\gamma_k^B(\bg)_i = \begin{cases}
        -g_i &\text{if } i=k,\\
         g_i + \sgn(g_k)[b_{ik}\cdot g_k]_+ &\text{else},
    \end{cases}
\end{equation}
where $B(Q)=(b_{ij})$ is the matrix defined in~\eqref{eq:B(Q)}.
Note that, this is just another way of writing the transformation
rule for g-vectors from~\cite[Conj.~7.12]{fomin2007cluster},
which was proved in~\cite[Thm.~1.7]{derksen2010quivers}.
Compare also with Equation~\eqref{comb gvector mutation}.

\begin{prop}\cite[Proposition 2.8]{geiss2023laminations} \label{prp:JaMut}
    For each $Z\in\DecIrr^{\tau^-}(A)$  there exists a dense open subset $U_Z\subset Z$,
    a unique irreducible component $\tilde{\mu}_k(Z)\in\DecIrr^{\tau^-}(A')$ and a regular
    map $\nu_Z: U_Z\ra\tilde{\mu}_k(Z)$ with the following properties:
    \begin{itemize}
    \item[(a)]
        For each $X\in U_Z$ we have $\nu_Z(X)\cong\mu_k(X)$, where $\mu_k$ denotes the
        mutation of the decorated QP-representation $X$ in direction $k$, as defined 
        in~\cite{derksen2008quivers}.
    \item[(b)]  The morphism of affine varieties
    $G_{\bd(\tilde{\mu}_k(Z))}\times U_Z \ra \tilde{\mu}_k(Z), (g, X)\mapsto g.\nu_Z(X)$
    is dominant.
    \item[(c)]
    For each $Z\in\DecIrr^\tau(A)$ we have 
    \[
    \bg_{A'}(\tilde{\mu}_k(Z))=\gam_k^B(\bg_A(Z)).
    \]
   \item[(d)] $\tilde{\mu}_k(\tilde{\mu_k}(Z))=Z$ for all 
    $Z\in\DecIrr^{\tau^-}(A)$.
    \item[(e)] The map 
    \[
    \tilde{\mu}_k:\DecIrr^{\tau^-}(A)\ra\DecIrr^{\tau^-}(A')
    \]
    is an isomorphism of partial KRS-monoids.  In particular, $Z\in\DecIrr^{\tau^-}(A)$ is 
    indecomposable if and only if $\tilde{\mu}_k(Z)\in\DecIrr^{\tau^-}(A')$ is indecomposable.
    \end{itemize}    
\end{prop}

\begin{rmk} \label{rk:Mut-Comp}
(1) Similar to Theorem~\ref{thm:tau-red-comps-form-a-KRS-monoid},
the above Proposition \ref{prp:JaMut} is the dual version of the original one  in~\cite{geiss2023laminations}, see also the discussion in Remark~\ref{rem:about-KRS-monoids-of-tau-red-comps}.

(2) In view of the definition of the piecewise linear map
$\gam_k^B$ in Equation~\eqref{eq:gproj-transf} and part~(c) of Proposition \ref{prp:JaMut} above we have the following interesting equality
\[
    \left[\begin{array}{c}-\mu_k(B(Q))\\
    \phantom{-}\bg_A'(\tilde{\mu}_k(Z))\end{array}\right]
    =\mu_k\left(
    \left[\begin{array}{c}-B(Q)\\
    \phantom{-}\bg_A(Z)\end{array}\right]\right)
\]
of $(n+1)\times n$ matrices, where $\mu_k$ is Fomin-Zelevinsky's matrix mutation.
Thus, Proposition~\ref{prp:JaMut} is a generalization of Nakanishi-Zelevinsky's result \cite[Equation (4.1)]{nakanishi2012on}. 
\end{rmk}

\subsection{KRS-monoids of laminations} \label{ssec:KRS-Lamin}
Let $\bSigma=\surf$ be a possibly-punctured surface with marked points.
In \cite[Section~4]{geiss2023laminations}, we considered
the set $\cLC^*(\bSigma)$ of homotopy classes of marked curves and loops on $\bSigma$
that have no kinks. Inspired by \cite{qiu2017cluster}, we  introduce a symmetric marked
intersection number
\[
\Intn^*_\bSigma: \cLC^*(\bSigma)\times\cLC^*(\bSigma)\ra\mathbb{Z}_{\geq 0}.
\] 
Given a marked curve (or  loop)
$(\gam, c)\in\cLC^*(\bSigma)$, let $(\gam,c)^{-1}$ denote the inversely oriented
curve $\gam$, with accordingly swapped decoration $c$.  The involution
$(\gam, c)\mapsto (\gam, c)^{-1}$ induces an equivalence relation $\simeq$ on $\cLC^*(\bSigma)$. For us, as in~\cite{qiu2017cluster}, a \emph{simple marked curve} is a marked curve
$(\gam, c)\in\cLC^*(\bSigma)$ which has self-intersection number $0$. We let $\cLC^*_{\tau}(\bSigma)$ denote the set of all simple marked curves, and define 
the tame partial KRS-monoid of laminations
\[
\Lamin(\bSigma):= \KRS(\cLC^*_{\tau}(\bSigma)/_{\simeq}, \Intn^*_\bSigma).
\]
See Example \ref{expl:KRS1} and \cite[Section~4.3]{geiss2023laminations} for more details.  The set $\cLC^*_{\tau}(\bSigma)/_{\simeq}$ can be identified
with the set of laminations $\mathcal{C}^\circ(\bSigma)$ considered
in~\cite{musiker2013bases} (two marked curves having intersection number equal to $0$ corresponds to them being \emph{$\mathcal{C}^\circ$-compatible} in Musiker-Schiffler-Williams' nomenclature). This is also compatible with the treatment of laminations in~\cite{fomin2018cluster}.

The tame partial KRS-monoid $\Lamin(\bSigma)$ can be equipped with a framing coming from (dual) shear coordinates. We shall describe this framing only for the indecomposable elements of $\Lamin(\bSigma)$, which we call \emph{laminates}. For the reader's convenience, we fully recall from \cite{geiss2023laminations} the definition of the vector of dual shear coordinates of a laminate with respect to an arbitrary tagged triangulation. Our treatment essentially follows \cite{fomin2018cluster}, but with slight changes in conventions --hence the adjective \emph{dual} in the term \emph{dual shear coordinates}. Let $T$ be a tagged triangulation of $\surf$, and let $\lam$ be either a tagged arc or a simple closed curve on $\surf$. We define the vector $\Sh_T(\lam)=(\Sh_T(\lam)_{\gam})_{\gam\in T}$ of dual shear coordinates of $\lam$ with respect to $T$ in steps as follows. 

\begin{case} Assume that $T$ has non-negative signature $\delta_T:\punct\rightarrow\{1,0,-1\}$. Then, at any given puncture $p$, either all tagged arcs in $T$ incident to $p$ are tagged plain, or exactly two tagged arcs in $T$ are incident to $p$, their underlying ordinary arcs being isotopic to each other, their tags at $p$ differing from one another, and the tags at their other endpoint being plain. 

Following \cite{fomin2008cluster}, represent $T$ through an ideal triangulation $T^\circ$ defined as follows. For each $p\in\mathbb{P}$ such that $\delta_T(p)=0$, let $i_p,j_p,$ be the two tagged arcs in $T$ that are incident to $p$, with $i_p$ tagged \emph{plain} at $p$, and $j_p$ tagged \emph{notched} at $p$. Then $T^\circ:=\{\gam^\circ\ | \ \gam\in T\}$ is obtained from $T$ by setting $\gam^\circ:=\gam$ for every $\gam\in T$ both of whose ends are tagged \emph{plain}, and by replacing each $j_p$ with a loop $j_p^\circ$ closely surrounding $i_p^\circ$ for $i$ and $j$ as above. Thus, the corresponding ordinary arcs $i_p^\circ,j_p^\circ\in T^\circ$ form a self folded triangle that has $i_p^\circ$ as the folded side, and $j_p^\circ$ as the loop closely surrounding $i_p^\circ$. See \cite[Sections 9.1 and 9.2]{fomin2008cluster}.

If $\lam$ is a simple closed curve, set $L:=\lam$. Otherwise, let $L$ be the curve obtained from $\lam$ by modifying its two ending segments according to the following rules:
\begin{itemize}
\item any endpoint incident to a marked point in the boundary is slightly slided along the boundary segment lying immediately to its right as in Figure \ref{Fig_rhohalf} (here, we stand upon the surface using its orientation, and look from the marked point towards the interior of surface, note that we use the orientation of $\Sig$ to determine what is right and what is left);
        \begin{figure}[!h]
        \caption{Slightly sliding endpoints lying on the boundary}\label{Fig_rhohalf}
                \centering
                \includegraphics[scale=.1]{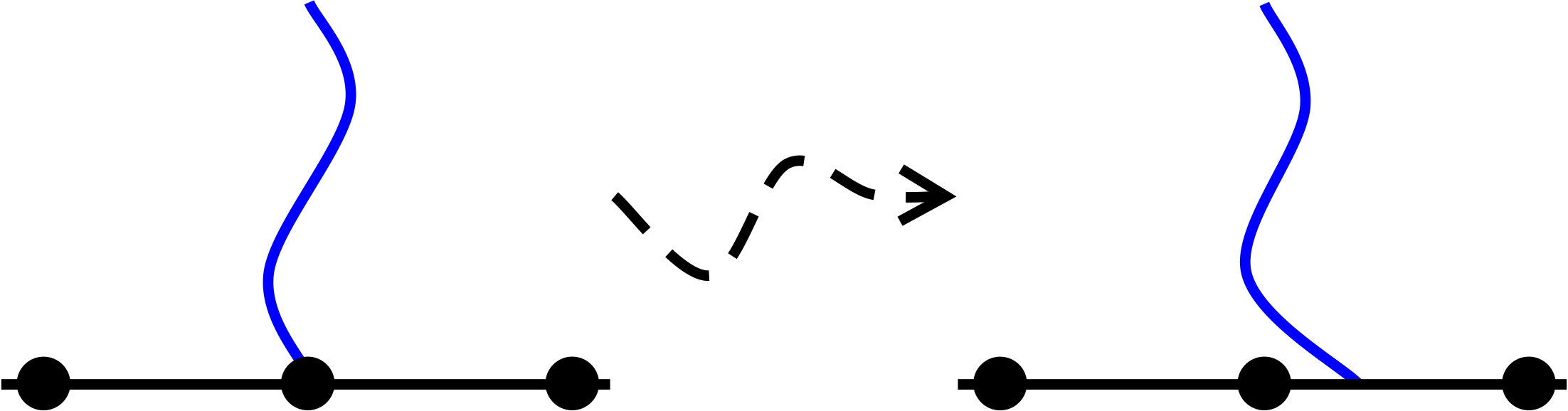}
        \end{figure}
\item any ending segment of $\lam$ tagged \emph{plain} at a puncture $q$ is replaced with a non-compact curve segment spiralling towards $q$ in the clockwise sense;
\item any ending segment of $\lam$ tagged \emph{plain} at a puncture $q$ is replaced with a non-compact curve segment spiralling towards $q$ in the counter-clockwise sense.
\end{itemize}

Take an arc $\gam\in T$. In order to define the shear coordinate $\Sh_T(\lam)_{\gam}$ we need to consider two subcases.

\begin{subcase}
Suppose that the ordinary arc $\gam^{\circ}\in T^\circ$ is not the folded side of a self-folded triangle of $T^\circ$. Then $\gam^\circ$ is contained in exactly two ideal triangles of $T^\circ$, and the union $\overline{\lozenge}_{\gam^\circ}$ of these two triangles is either a quadrilateral (if $\gamma^\circ$ does not enclose a self-folded triangle) or a digon (if $\gamma^\circ$ encloses a self-folded triangle). In any of these two situations, the complement in $\overline{\lozenge}_{\gam^\circ}$ of the union of the arcs belonging to $T^\circ\setminus\{\gam^\circ\}$ can be thought to be an open quadrilateral $\lozenge_{\gam^\circ}$ in which $\gamma^\circ$ sits as a diagonal. The shear coordinate $\Sh_T(\lam)_{\gam}$ is defined to be the number of segments  of $\lozenge_{\gam^\circ}\cap L$ that form the shape of a letter $Z$ when crossing $\gam^\circ$ minus the  number of segments  of $\lozenge_{\gam^\circ}\cap L$ that form the shape of a letter $S$ when crossing $\gam^\circ$.
\end{subcase}

\begin{subcase}
Suppose that the ordinary arc $\gam^{\circ}\in T^\circ$ is the folded side of a self-folded triangle of $T^\circ$. Then there is a puncture $p\in\punct$ such that $\gam=i_p$ and $\gam^\circ=i_p^\circ$. Define 
\[
\Sh_T(\lam)_{i_p}:=\Sh_T(\lam')_{j_p}.
\]
where $\lam'$ is obtained from $\lam$ by switching the tags of $\lam$ at the puncture $p$.
\end{subcase}
\end{case}

\begin{case} Assume now that $T$ is an arbitrary tagged triangulation of $\surf$. The set of punctures at which $T$ has negative signature is the inverse image $\delta_T^{-1}(-1)$. Set $T'$ to be the tagged triangulation obtained from $T$ by changing from notched to plain all the tags incident to punctures in~$\delta_T^{-1}(-1)$. Thus, $T'$ is a tagged triangulation of signature zero, so dual shear coordinates with respect to $T'$ have already been defined. Set
\[
\Sh_T(\lam):=\Sh_{T'}(\lam'),
\]
where $\lam'$ is obtained from $\lam$ by switching all the tags of $\lam$ at the punctures belonging to the set $\delta_T^{-1}(-1)$.
\end{case}

\begin{exmp}\label{ex:FT-shear-coords-VS-dual-shear-coords} In Figure \ref{Fig_dualShearCoords_vs_FSTShearCoords},
        \begin{figure}[!h]
        \caption{Left: computation of the vector of dual shear coordinates $\Sh_T(\lam)$. Right: computation of the vector of Fomin--Thurston's shear coordinates $\Sh_{\overline{T}}^{\FST}(\overline{\lam})$.}\label{Fig_dualShearCoords_vs_FSTShearCoords}
                \centering
                \includegraphics[scale=.075]{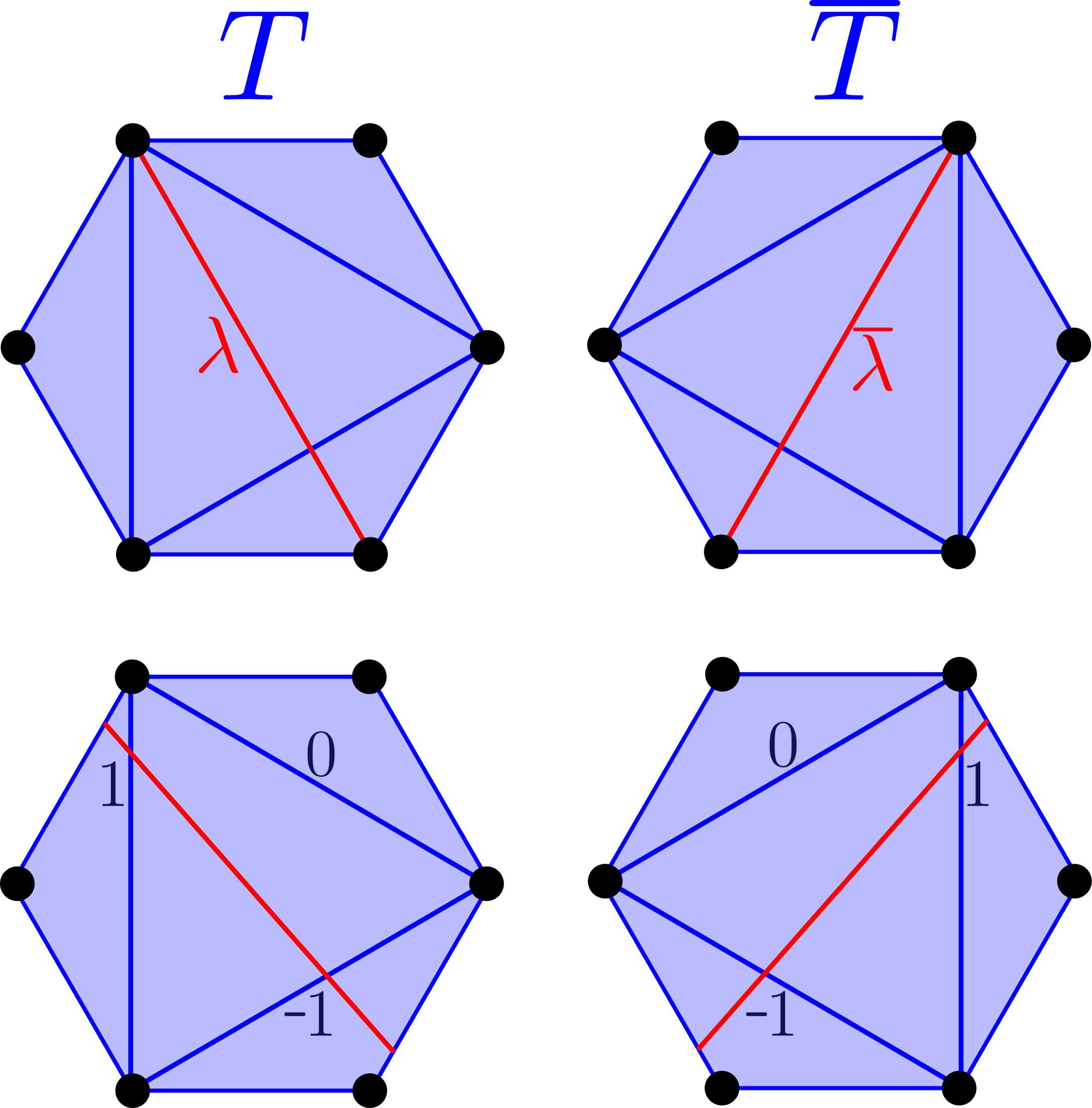}
        \end{figure}
the reader can glimpse the relation between the dual shear coordinates we have defined above and the shear coordinates used by Fomin--Thurston, namely,
\begin{equation}\label{eq:relation-shear-coord-vs-dual-shear-coord}
\Sh_T(\lam)=\Sh_{\overline{T}}^{\FST}(\overline{\lam}),
\end{equation}
where $\Sh_T(\lam)$ is the vector of dual shear coordinates we have defined above, $\overline{T}$ and $\overline{\lambda}$ are the images of $T$ and $\lambda$ in the surface obtained as the \emph{mirror image} of $\surf$, and $\Sh_{\overline{T}}^{\FST}(\overline{\lam})$ is the vector of shear coordinates of Fomin--Thurston, \cite[Definition 13.1]{fomin2018cluster}.
\end{exmp}

With these definitions in place we can restate~\cite[Theorem 13.5]{fomin2018cluster} for dual shear coordinates as follows.

\begin{thm}\label{thm:behavior-of-dual-shear-coords-under-matrix-mutation}
    Let $T$ and $T'$ be tagged triangulations of $\surf$ related by the flip of a tagged arc $k\in T$, and let $\lam$ be either a tagged arc or a simple closed curve on $\surf$. Then
\[    
    \left[\begin{array}{c}-B(T')\\
    \Sh_{T'}(\lam)\end{array}\right]
    =\mu_k\left(
    \left[\begin{array}{c}-B(T)\\
    \Sh_T(\lam)\end{array}\right]\right).
\]
\end{thm}

In view of the definition of the piecewise linear map
$\gam_k^B$ in  Equation~\eqref{eq:gproj-transf} and the
definition of the Fomin-Zelevinsky matrix mutation rule,
Theorem~\ref{thm:behavior-of-dual-shear-coords-under-matrix-mutation} implies in particular 
\begin{equation} \label{eq:unraveled-recursion-satisfied-by-dual-shear-coords}
\Sh_{T'}(\lam)=\gam^{B(T)}_k(\Sh_T(\lam)).    
\end{equation}
See also Remark~\ref{rk:Mut-Comp}~(2).

\subsection{The isomorphism of framed KRS-monoids 
\texorpdfstring{$\Lamin(\bSigma)\cong\DecIrr^\tau(A)$}{DecIrrτ(A)=Lam(Σ)}}
Generalizing \cite[Theorem 10.13 and Proposition 10.14]{geiss2022schemes} from the unpunctured to the punctured case, the main result of \cite{geiss2023laminations} states with our conventions

\begin{thm}\cite[Theorem 1.1.]{geiss2023laminations}\label{thm:Lams-vs-taured-comps-iso}
    Let $\bSigma=\surf$ be a surface with non-empty boundary. For each  
tagged triangulation $T$ of $\bSigma$ there is a unique isomorphism   
of partial KRS-monoids
\[
  \pi_T: (\Lamin(\bSigma),\Intn^*,+) \ra
  (\DecIrr^{\tau-}(A(T)), e_{A(T)},\oplus), 
\]
such that the diagram of functions and sets
$$
\xymatrix{
\Lamin(\bSigma) \ar[rr]^{\pi_T} \ar[dr]_{\Sh_T} & & \DecIrr^{\tau^-}(A(T)) \ar[dl]^{\mathbf{g}_{A(T)}} \\
& \mathbb{Z}^{T} &
}
$$
where $A(T)$ is the dual Jacobian algebra of the quiver with potential associated to $T$ in \cite[Definition 4.1]{cerulli2012quivers}, \cite[Definitions 3.1 and 3.2]{labardini2016quivers}, and $\bg_{A(T)}$ is the generic
injective g-vector.
\end{thm}

Roughly speaking, the proof strategy followed in \cite{geiss2023laminations} consists of two steps:
\begin{enumerate}
    \item the statement of the theorem is shown to hold for tagged triangulations of signature zero, exploiting  heavily the fact that for those triangulations the Jacobian algebra is skewed-gentle, and that in this situation a combinatorial
    description of the generically $\tau^-$-regular components
    is now available~\cite{geiss2023onhomomorphisms};
    \item given a tagged triangulation $T$ for which the theorem holds, and another tagged triangulation $T'$ obtained by flipping an arc $k\in T$, Proposition~\ref{prp:JaMut} and Theorem~\ref{thm:behavior-of-dual-shear-coords-under-matrix-mutation} are applied to produce a commutative diagram
\[
    \xymatrix{
    & & & & \DecIrr^{\tau^-}(A(T)) \ar[dl]^{\mathbf{g}_{A(T)}}\ar[dddd]^{\cong}_{\widetilde{\mu}_k}\\
     & & & \mathbb{Z}^{T} \ar[dd]|-{\quad\gamma_k^{B(T)}} & \\
    \Lamin(\bSigma) \ar@/^2.5pc/[uurrrr]|-{\pi_T}^{\cong\quad} \ar@/_2.5pc/@{.>}[ddrrrr]|-{\quad \pi_{T'}:=\widetilde{\mu}_k\circ \pi_T}  \ar[urrr]_{\Sh_T} \ar[drrr]^{\Sh_{T'}} & & & & \\
     & & & \mathbb{Z}^{T} &\\
     & & &  & \DecIrr^{\tau^-}(A(T')) \ar[ul]_{\mathbf{g}_{A(T')}}
    }
\]
and this allows to deduce that the theorem holds then for $T'$ as well.
\end{enumerate}
Since any two tagged triangulations of a surface with non-empty boundary are related by a finite sequence of flips, and since any such surface certainly admits at least one triangulation of signature zero, these two steps yield a proof of Theorem \ref{thm:Lams-vs-taured-comps-iso}.

\section{Triangulations adapted to closed curves}
\label{base case section}
Let $\alpha$ be an \emph{essential loop} on $\bSigma=\surf$, i.e., a non-contractible simple closed curve that is furthermore not contractible to a puncture, and let $T=\{\tau_1,\ldots,\tau_n\}$ be a tagged triangulation of $\bSigma$. The \emph{band graph} associated to $\alpha$ with respect to $T$, cf. \cite[Definition 3.4 and \S8.3]{musiker2013bases}, will be denoted $G(T,\alpha)$.
The polynomial $F_\zeta^T$ defined in \cite[Definition 3.14]{musiker2013bases} (see also \cite[\S8.3]{musiker2013bases}), with $\zeta:=\alpha$, will be called \emph{snake $F$-polynomial} and denoted $F_{G(T,\alpha)}$. Similarly, the integer vector from \cite[Definition 6.1]{musiker2013bases} will be called the \emph{snake $g$-vector} and denoted $\bg_{G(T,\alpha)}$. Notice that $$F_{G(T,\alpha)} = \displaystyle \sum_{P \in \mathcal{P}(G(T,\alpha))} y(P)$$ where  $\mathcal{P}(G(T,\alpha))$ is the collection of all good matchings of the band graph $G(T,\alpha)$.
Furthermore, by \cite[Proposition 6.2]{musiker2013bases}, we have
\[
\bg_{G(T,\alpha)}=\deg\left(\frac{x(P_-)}{\operatorname{cross}(T,\alpha)}\right),
\]
where $P_-$ is the \emph{minimal matching} of the band graph $G(T,\alpha)$, cf. \cite[Definition 3.7]{musiker2013bases}, and $\deg:\mathbb{Z}[x_1^{\pm1},\ldots,x_n^{\pm1},y_1,\ldots,y_n]\rightarrow \mathbb{Z}^n$ is the grading defined by
\[
\deg(x_i)=\mathbf{e}_i \qquad \text{and} \qquad \deg(y_j):=-\mathbf{b}(T)_j
\]
as in \cite[Proposition 6.1]{fomin2007cluster}, $\mathbf{b}(T)_j$ being the $j^{\operatorname{th}}$ column of $B(T)$, 
see Definition~\ref{adjacency matrix}.

\begin{prop} 
\label{prop:existence-ad-hoc-triangulation-for-closed-curve}
For each non-contractible simple closed curve $\alpha$ on $\bSig$ there exists a triangulation $T_{\alpha}$ without self-folded triangles such that:
\begin{itemize}
\item the Jacobian algebra of the restriction $(Q'(T_\alpha)|_I,{ S'}(T_\alpha)|_I)$ is gentle, where $I$ is the set of arcs crossed by $\alpha$;
\item the band module $M(T_\alpha,\alpha,\lambda,1)$ over $\Lambda_I:=\jacobalg{Q(T_\alpha)|_I,S(T_\alpha)|_I}$ of quasi-length $1$ associated to the parameter $\lambda\in\CC^*$ is a module over the Jacobian algebra $A(T_{\alp}):=\cP_\CC(Q'(T_{\alp}), S'(T_{\alp}))$ as well and satisfies:
\begin{align*}
     F_{M(T_\alpha,\alpha,\lambda,1)} &= F_{G(T,\alpha)} & E_{A(T_\alpha)}(M(T_\alpha,\alpha,\lambda,1))&=1   \\
     \bg_{A(T_\alpha)}(M(T_\alpha,\alpha,\lambda,1)) &= \Sh_{T_\alpha}(\alpha) = \bg_{T_\alpha}(G(T_\alpha,\alpha)).
\end{align*}
\end{itemize}
\end{prop}
\begin{proof}
Cutting $\bSigma$ along $\alpha$ we are left with at most two marked surfaces $(S_1,M_1)$ and $(S_2,M_2)$ -- note that if $\alpha$ is non-separating then one $S_i$ will be empty. To construct $T_{\alpha}$ we split the task into two cases. 

\setcounter{case}{0}
\begin{case}
If $M_i \neq \emptyset$ for any $i \in \{1,2\}$ then there exist compatible arcs $\gamma_1$ and $\gamma_2$ in $\bSig$ which bound $\alpha$ in a cylinder $C_{1,1}$. Define $T_{\alpha}$ to be any triangulation containing both $\gamma_1$ and $\gamma_2$, with the extra property of not having self-folded triangles. See Figure \ref{case1triangulation}.

\begin{figure}[ht]
\caption{The two types of $\alpha$ occurring in Case 1, together with the arcs $\gamma_1$ and $\gamma_2$ bounding it in a cylinder.}
\label{case1triangulation}
\begin{center}
\includegraphics[width=13cm]{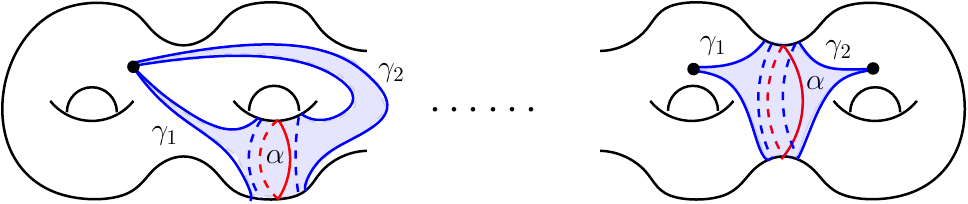}
\end{center}
\end{figure}
\end{case}

\begin{case} Otherwise, $\alpha$ is a separating curve where $S_i \neq \emptyset$ and $M_i = \emptyset$ for some $i \in \{1,2
\}$. In this case there exists an arc $\gamma \in \bSig$ bounding $\alpha$ in a bordered surface $\bSig_{g,1}$ of genus $g$ with one boundary component and one marked point. Define $T_{\alpha}$ to be any triangulation containing $\gamma$, with the extra property of not having self-folded triangles. See Figures \ref{case2triangulation} and \ref{bandgraphform2}.

\begin{figure}[ht]
\caption{The type of $\alpha$ occurring in Case 2, together with the arc $\gamma$ bounding it in a genus $g$ surface with one boundary component and one marked point.}
\label{case2triangulation}
\begin{center}
\includegraphics[width=13cm]{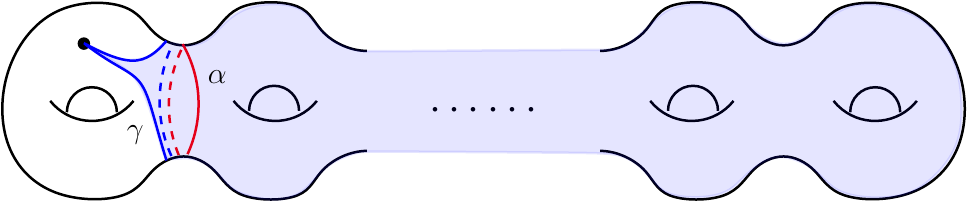}
\end{center}
\end{figure}
\end{case}

In both cases, it is clear that $\Lambda_I:=\jacobalg{Q(T_\alpha)|_I,S(T_\alpha)|_I}$ is a gentle algebra. That $M(T_\alpha,\alpha,\lambda,1)$ is a module also over $A(T_{\alp})$ follows immediately from \cite[Proposition 8.9]{derksen2008quivers} or by direct inspection.

It is easy to check that $A(\overline{T_\alpha})=A(T_\alpha)^{\operatorname{op}}$ and $M(\overline{T_\alpha},\overline{\alpha},\lambda,1)\cong D(M(T_\alpha,\alpha,\lambda,1))$ as representations of $A(\overline{T_\alpha})$. Therefore, by \eqref{eq:gvec}, \cite[proof of Theorem 10.0.5]{labardini2010quivers} and \eqref{eq:relation-shear-coord-vs-dual-shear-coord}, we have 
\[
    \bg_{A(T_\alpha)}(M(T_\alpha,\alpha,\lambda,1))  = \bg_{A(\overline{T_\alpha})}^{\prj}(M(\overline{T_\alpha},\overline{\alpha},\lambda,1))=
    \Sh_{\overline{T_\alpha}}^{\FST}(\overline{\alpha}) = \Sh_{T_\alpha}(\alpha).
\] 
The equality $\Sh_{T_\alpha}(\alpha)=\bg_{G(T_\alpha,\alpha)}$ follows from \cite[Proposition 10.14]{geiss2022schemes}. 

Recall, that our quiver $Q'(T_\alp)$ is opposite to 
the quiver $Q(T)$
which was used in~\cite{labardini2010quivers} and in~\cite{geiss2022schemes}.

\begin{lem}
\label{bandgraphform}
The band graph $G(T_\alpha,\alpha)$ is a `zig-zag' band graph whose corresponding band is given in Figure \ref{bandgraph}.

\end{lem}

\begin{figure}[ht]
\caption{Here we depict the possible shapes of what we call `zig-zag' band graphs. The only restriction is that $n \geq 2$.}
\label{bandgraph}
\begin{center}
\includegraphics[width=13cm]{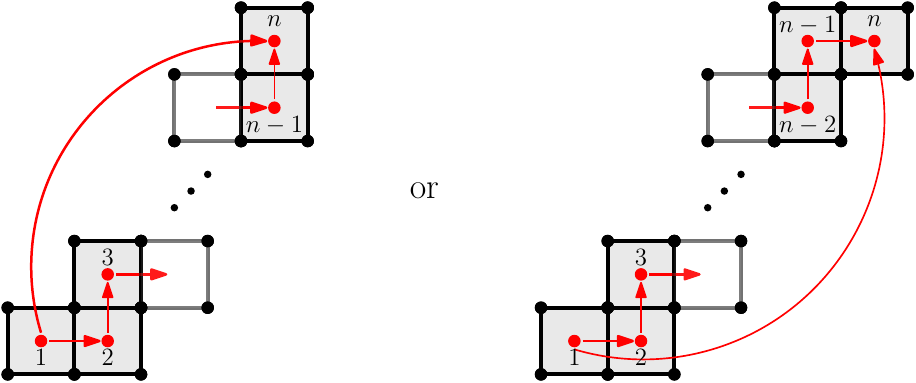}
\end{center}
\end{figure}

\begin{proof}

When $T_{\alpha}$ is defined via case $1$, then a direct computation obtains the $n=2$ band graph in Figure \ref{bandgraph}.

So, suppose $T_{\alpha}$ is defined via case $2$. Then there exists a unique triangle $\Delta$ of $T_{\alpha}$, lying inside $\bSigma_{g,1}$, such that $\gamma$ is an edge of $\Delta$. Let us denote the remaining two edges of $\Delta$ by $\gamma_1$ and $\gamma_2$. For some $\beta$ and $\delta$ in $T$ we see that $\alpha$ has the following sequences of `zig-zag' intersections: $\gamma_1, \gamma_2, \beta$ and $\delta, \gamma_1, \gamma_2$. All other sequences of intersections (of $\alpha$ with $T$) of length $3$ are `fan' intersections. This completes the proof.
\end{proof}

\begin{figure}[ht]
\caption{The green shaded area denotes the triangle $\Delta$ used in the proof of Lemma \ref{bandgraphform}.}
\label{bandgraphform2}
\begin{center}
\includegraphics[width=8cm]{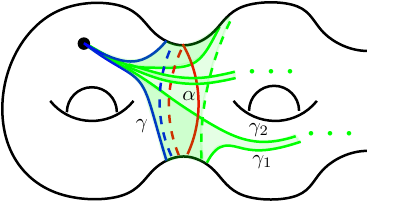}
\end{center}
\end{figure}

\begin{lem} \label{lem:TrivEnd}
For
 $\lam\in\CC^*$, the endomorphism ring 
 $\End_{A(T_\alp)}(M(T_\alp,\alp,\lam,1))$ of the regular quasi-simple 
 band module $M(T_\alp, \alp,\lam,1)$ is isomorphic to $\CC$.  
\end{lem}
\begin{proof}
The band graph of $\alp$ with respect to $T_\alp$ is of 
`zig-zag' type by construction.  
As shown in Lemma~\ref{bandgraphform}, each `zig-zag' band graph gives rise to a band $b$ of the form shown in \eqref{band}.  Moreover, the "support" $\Lam_I$ of
$b$ is a special biserial (in fact hereditary) algebra. Thus we have trivially
$\End_{A(T_\alpha)}(M(T_\alpha,\alpha,\lambda,1))=
\End_{\Lambda_I}(M(T_\alpha,\alpha,\lambda,1))$
We shall follow  the setup and terminology used by Butler and Ringel in~\cite{butler1987auslander} for representations of band modules. 

Let $M= M(b,\lam, 1)$ be a regular simple 
band module corresponding to the following band $b$: 
\begin{equation} \label{band}
\begin{tikzpicture}
\matrix(m)[matrix of math nodes, row sep=3em, column sep=2em, text height=1.5ex, text depth=0.25ex]{t_{i_n} &t_{i_{n-1}} & t_{i_{n-2}} & \ldots &t_{i_3} & t_{i_2} & t_{i_1} \\};
\path[->](m-1-2) edge node [above]{$\alpha_{n-1}$} (m-1-1);
\path[->](m-1-3) edge node [above]{$\alpha_{n-2}$} (m-1-2);
\path[->](m-1-6) edge node [above]{$\alpha_{2}$} (m-1-5);
\path[->](m-1-7) edge node [above]{$\alpha_{1}$} (m-1-6);
\path[->](m-1-7) edge [bend right=-20] node [below]{$\alpha_{n}$} (m-1-1);
\end{tikzpicture}
\end{equation}
Here, we may think of vertices $t_{i_j}\in Q_0(T_\alp)$ as the arcs of the triangulation  $T_\alp$ which intersect
consecutively with the loop $\alp$.  Note that possibly $t_{i_j}=t_{i_k}$ for
certain $i\neq k$.  However, $t_{i_1}\neq t_{i_n}$ since the arrow $\alp_n$ is not
a loop.  
Specifically, $M$ is defined by the following collection of vector spaces and maps:
\begin{center}
\begin{tikzpicture}
\matrix(m)[matrix of math nodes, row sep=3em, column sep=2em, text height=1ex, text depth=0.25ex]{\mathbb{C} &\mathbb{C} & \mathbb{C} & \ldots &\mathbb{C} & \mathbb{C} &\mathbb{C} \\};
\path[->](m-1-2) edge node [above]{$I$} (m-1-1);
\path[->](m-1-3) edge node [above]{$I$} (m-1-2);
\path[->](m-1-6) edge node [above]{$I$} (m-1-5);
\path[->](m-1-7) edge node [above]{$I$} (m-1-6);
\path[->](m-1-7) edge [bend right=-15] node [below]{$\lam^{-1}$} (m-1-1);
\end{tikzpicture}
\end{center}
Now,  we can use Krause's formalism from~\cite{krause1991maps}.  With the notation
from \emph{loc.~cit.}~p.~191 we denote by $\mathcal{A}(b,b)$ the set of isomorphism classes of admissable (\emph{sic}) triples for the pair of bands  $(b,b)$.  
Since the band $b$ has a unique source, which is mapped to $t_{i_1}$ and a unique sink, which is mapped to $t_{i_n}\neq t_{i_1}$ we have
$\mathcal{A}(b,b)=\{[(b, \operatorname{id}_b,\operatorname{id}_b)]\}$.  
Thus, $\End_{\Lam_I}(M)\cong\CC$ follows directly
from the Theorem on~\cite[p.~191]{krause1991maps}.
\end{proof}

By Lemma~\ref{lem:TrivEnd} we have
$\End_{A(T_\alpha)}(M(T_\alpha,\alpha,\lambda,1))=\CC$. 
On the other hand, 
one can directly check that $\dimv(M(T_\alpha,\alpha,\lambda,1))\cdot\bg_{A(T_\alpha)}(M(T_\alpha,\alpha,\lambda,1))=0$. Hence
\begin{align*}
E_{A(T_\alpha)}(M(T_\alpha,\alpha,\lambda,1)) &= \dim(\End_{A(T_\alpha)}(M(T_\alpha,\alpha,\lambda,1)))+\\
&\quad +\dimv(M(T_\alpha,\alpha,\lambda,1))\cdot\bg_{A(T_\alpha)}(M(T_\alpha,\alpha,\lambda,1))=1.
\end{align*}

The equality $F_{M(T_\alpha,\alpha,\lambda,1)} = F_{G(T_{\alpha},\alpha)}$ follows from~\cite[Lemma 11.4 and Remark 11.7]{geiss2022schemes} and~\cite[Theorem 1.2]{haupt2012euler} as in the proof of
\cite[Proposition 11.8]{geiss2022schemes}.
Note however, that due to our convention for $A(T_\alp)$ resp.\ for $\cP_\CC(Q'(T_\alp)\mid_I, S'(T_\alp)_I)$, we can avoid the
use of dual CC-functions.
This finishes the proof of Proposition \ref{prop:existence-ad-hoc-triangulation-for-closed-curve}.
\end{proof}

\begin{rmk}
    In the proof of Proposition \ref{prop:existence-ad-hoc-triangulation-for-closed-curve} we have partially used the explicit definition of the potential $S(T)$: we know that the restriction of $S(T)$ to the set $I$ is a gentle algebra.
\end{rmk}

\begin{cor}\label{coro:generic-values-on-adhoc-irred-comp-coincide-with-MSW}
In the situation of Proposition \ref{prop:existence-ad-hoc-triangulation-for-closed-curve},
\begin{itemize}
\item 
the Zariski closure
\[
    Z_{T_\alpha,\alpha}:=\overline{\bigcup_{\lambda\in \CC^*}\operatorname{GL}_{\mathbf{d}}(\CC)\cdot M(T_\alpha,\alpha,\lambda,1)}
\]    
is a generically $\tau^-$-regular indecomposable irreducible component of the representation space $\operatorname{rep}(A(T_\alpha),\mathbf{d})$, where $\mathbf{d}:=\dimv(M(T_\alpha,\alpha,\lambda,1))$;
\item 
the generic (injective) $g$-vector $\bg_{A(T_\alpha)}(Z_{T_\alpha,\alpha})$ is equal to the snake $g$-vector $\bg_{G(T_\alpha,\alpha)}$;
\item 
the generic value $CC_{A(T_\alpha)}(Z_{T_\alpha,\alpha})$ is equal to Musiker-Schiffler-Williams' expansion $\MSW(G(T_\alpha,\alpha))$.
\end{itemize}
\end{cor}

\begin{proof}
By \cite[Theorem 7.1]{geiss2016the} (see also \cite[Theorem 3.5]{labardini2016on}), the algebra $A(T_\alpha)$ is tame. Furthermore, for $\lambda_1\neq\lambda_2$, the band modules $M(T_\alpha,\alpha,\lambda_1,1)$ and $M(T_\alpha,\alpha,\lambda_2,1)$ are non-isomorphic, and by Proposition \ref{prop:existence-ad-hoc-triangulation-for-closed-curve}, for every point $M$ in the dense open subset $\bigcup_{\lambda\in\mathbb{C}\setminus\{0\}}\operatorname{GL}_{\mathbf{d}}(\mathbb{C})\cdot M(T_\alpha,\alpha,\lambda,1)$ we have $E_{A(T_\alpha)}(M)=1$. Hence
$Z_{T_\alpha,\alpha}$ is a generically $\tau^-$-regular indecomposable irreducible component by \cite[Lemma 3.1 and Theorem 3.2]{geiss2022schemes} (see also \cite[Section~2.2]{carroll2015on}).

By Proposition \ref{prop:existence-ad-hoc-triangulation-for-closed-curve}, for every point $M\in \bigcup_{\lambda\in\mathbb{C}\setminus\{0\}}\operatorname{GL}_{\mathbf{d}}(\mathbb{C})\cdot M(T_\alpha,\alpha,\lambda,1)$ we have
$$
F_{M} = F_{G(T_\alpha,\alpha)} \qquad \text{and} \qquad
     \bg_{A(T_\alpha)}(M) = \bg_{G(T_\alpha,\alpha)}.
$$
This implies the second and third assertions.
\end{proof}


\section{The Combinatorial Key Lemma}
\label{CKL}
\subsection{Derksen--Weyman--Zelevinsky's representation-theoretic Key Lemma}
The following result is the Key Lemma of Derksen, Weyman and Zelevinsky. 

\begin{thm}[Lemma 5.2, \cite{derksen2010quivers}]
\label{keylemma}
Let $\mathcal{M}$ be a finite-dimensional decorated representation of a non-degenerate quiver with potential $(Q,S)$, and define $\overline{\mathcal{M}}$ to be the mutation of $\mathcal{M}$ in direction $k \in \{1,\ldots, n\}$. Then 

\begin{equation}
\label{Fmutation}
(y_k +1)^{h_k}F_{\mathcal{M}}(y_1,\ldots, y_n) = (y_k' +1)^{h_k'}F_{\overline{\mathcal{M}}}(y_1',\ldots, y_n'),
\end{equation}

where 
\begin{itemize}
\item $(\mathbf{y}',B') \in \mathbb{Q}(y_1,\ldots, y_n)$ is obtained from $(\mathbf{y},B_{\operatorname{DWZ}}(Q))$ by $Y$-seed mutation at $k$,
\item $h_k$ and $h_k'$ are the $k^{th}$ components of the $\mathbf{h}$-vectors $\mathbf{h}_{\mathcal{M}}$ and $\mathbf{h}_{\overline{\mathcal{M}}}$, respectively, defined by \cite[Equations (1.8) and (3.2)]{derksen2010quivers}.
\end{itemize}
Moreover, the $\mathbf{g}$-vector $\mathbf{g}_{\jacobalg{Q,S}}^{\operatorname{DWZ}}(\mathcal{M}) = (g_1,\ldots, g_n)$ defined by \cite[Equation (1.13)]{derksen2010quivers} satisfies 
\begin{equation}
\label{kth gvector}
g_k = h_k - h_k'
\end{equation} 
and is related to the $\mathbf{g}$-vector $\mathbf{g}_{\jacobalg{\mu_k(Q,S)}}^{\operatorname{DWZ}}(\overline{\mathcal{M}}) = (g_1',\ldots, g_n')$  via 
\begin{equation}
\label{gvector mutation}
g_j'  := \left\{
\begin{array}{ll}
        -g_k, &\text{if $j =k$}; \\
        \\
        g_j + [b_{jk}]_{+}g_k - b_{jk}h_k, &\text{if $j \neq k$}.\\
\end{array} 
\right.
\end{equation}
\end{thm}

\begin{rmk}
By~\cite[Proposition~10.4 and Remark~10.8]{derksen2010quivers}, $\mathbf{g}_{\mathcal{M}}^{\operatorname{DWZ}}=\mathbf{g}_{\mathcal{M}}^{\operatorname{inj}}$
    (see also \eqref{eq:gvec}).
\end{rmk}

\subsection{Statement of the Combinatorial Key Lemma}

\begin{thm}[Combinatorial Key Lemma]
\label{thm: comb key lemma}
Let $T=\{\tau_1,\ldots,\tau_n\}$ be a tagged triangulation of $\bSig$,  $\tau_k$  a tagged arc belonging to $T$, and $\alpha$  a simple closed curve. We set $T'$ to be the tagged triangulation obtained from $T$ by flipping $\tau_k$. The following equality holds:

\begin{equation}
\label{combinatorial Fmutation}
(y_k +1)^{h_k}F_{G(T,\alpha)}(y_1,\ldots, y_n) = (y_k' +1)^{h_k'}F_{G(T',\alpha)}(y_1',\ldots, y_n')
\end{equation}

where \begin{itemize}

\item $(\mathbf{y}',B') \in \mathbb{Q}(y_1,\ldots, y_n)$ is obtained from $(\mathbf{y},B(T))$ by mutation at $k$,

\item the \emph{snake $\mathbf{h}$-vector} $\mathbf{h}_{G(T,\alpha)}$ is defined by setting
\[
u^{h_i} := {F_{G(T,\alpha)}}_{{\vert}_{\operatorname{Trop}(u)}}(u^{[-b_{i1}]_{+}}, \ldots, u^{-1}, \ldots, u^{[-b_{in}]_{+}}), 
\]
where $u^{-1}$ is in the $i^{th}$ component. The snake $\mathbf{h}$-vector $\mathbf{h}_{G(T',\alpha)}$ is defined analogously with respect to $T'$ and $\mu_{k}(B(T))=B(T')$.

\end{itemize}

\noindent Moreover, the snake $\mathbf{g}$-vector $\mathbf{g}_{G(T,\alpha)} = (g_1,\ldots, g_n)$ satisfies 
\begin{equation}
\label{comb kth gvector}
g_k = h_k - h_k'
\end{equation} 
and $\mathbf{g}_{G(T,\alpha)} = (g_1,\ldots, g_n)$ is related to $\mathbf{g}_{G(T',\alpha)} = (g_1',\ldots, g_n')$ by the following rule:
\begin{equation}
\label{comb gvector mutation}
g_j'  := \left\{
\begin{array}{ll}
        -g_k, &\text{if $j =k$}; \\
        \\
        g_j + [b_{jk}]_{+}g_k - b_{jk}h_k, &\text{if $j \neq k$}.\\
\end{array} 
\right.
\end{equation}

\end{thm}

\begin{rmk}\label{rem:snake-h-vector-is-non-positive} Since the band graph $G(T,\alpha)$ always admits a minimal matching, the snake $F$-polynomial $F_{G(T,\alpha)}$ has constant term $1$, hence every entry of the snake $h$-vector $\mathbf{h}_{G(T,\alpha)}$ is non-positive.
\end{rmk}

\begin{exmp}
As shown in Figure \ref{exmp: cylinder}, let $T$ and $T'$ be triangulations of the cylinder $C_{1,1}$, where $T'$ is obtained from $T$ by the flip of $\tau_1$. Furthermore, let $\alpha$ be the unique simple closed curve of $C_{1,1}$.

Directly from the definition, let us compute the snake $F$-polynomials $F_{G(T,\alpha)}$ and $F_{G(T',\alpha)}$ of $\alpha$. Namely, considering the good matchings associated to the band graphs $G(T,\alpha)$ and $G(T',\alpha)$ shown in Figure \ref{exmp: cylinder} we get: $$F_{G(T,\alpha)}(y_1,y_2) = 1 + y_2 + y_1y_2 \hspace{10mm} \text{and} \hspace{10mm} F_{G(T',\alpha)}(y_1,y_2) = 1 + y_1 + y_1y_2.$$

Now we wish to calculate the components $h_1$ and $h_1'$ corresponding to $\tau_1$ in the respective $\mathbf{h}$-vectors $\mathbf{h}_{\alpha, T}$ and $\mathbf{h}_{\alpha, T'}$.
Since $b_{12} = -2$ and $b_{12}' = 2$ we get: $${F_{G(T,\alpha)}}_{\vert_{Trop(u)}}(u^{-1},u^2) = 1\oplus u^2\oplus u = 1$$ and $${F_{G(T',\alpha)}}_{\vert_{\operatorname{Trop}(u)}}(u^{-1},1) = 1\oplus u^{-1} \oplus u^{-1} = u^{-1}.$$

Consequently, we have $h_1 = 0$ and $h_1' = -1$. As such, \eqref{combinatorial Fmutation} now reduces to: $$F_{G(T,\alpha)}(y_1,y_2) = (y_1'+1)^{-1}F_{G(T',\alpha)}(y_1',y_2'),$$
which follows directly from the definition of $Y$-seed mutation. Indeed,  $y_1' := y_1^{-1}$ and $y_2' := y_2(1+y_1)^2$.

\end{exmp}

\begin{figure}[H]
\caption{Left: triangulations $T$ and $T'$ of $C_{1,1}$ related by the flip of $\tau_1$. Right: the band graphs $G(T,\alpha)$ and $G(T',\alpha)$.}
\label{exmp: cylinder}
\begin{center}
\includegraphics[width=14cm]{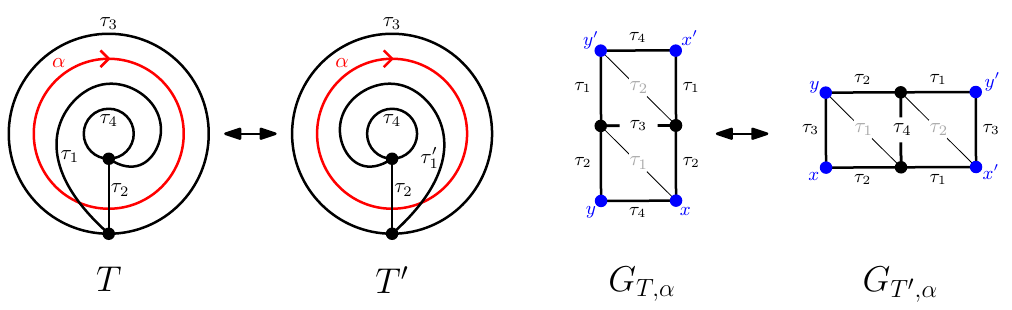}
\end{center}
\end{figure}

The rest of Section is devoted to proving Theorem \ref{thm: comb key lemma}.

\subsection{Local segments and their local flips}\label{subsec:local-segments}
Note that, in general, the band graphs $G(T,\alpha)$ and $G(T',\alpha)$ will have very different shapes, consequently, their collections of good matchings seem to differ considerably. This is due to the fact that $\alpha$ may intersect the flip region multiple times, and in many different combinatorial ways. Nevertheless, we shall show that the relationships between the snake $\mathbf{g}$-vectors, snake $\mathbf{h}$-vectors and snake $F$-polynomials of $G(T,\alpha)$ and $G(T',\alpha)$ are governed by `local' considerations.

Let us (cyclically) label the intersections points between $\alpha$ and $T^{\circ}$ by $p_1, p_2, \ldots, p_d$. By convention, indices of these intersection points will always be taken modulo $d$.

\begin{defn}
Consider a subcurve $\alpha_{ij}$ of $\alpha$ with intersection points $p_i, \ldots, p_j$. We call $\alpha_{ij}$ a \textbf{local segment} of $\alpha$ \textbf{with respect to $\gamma \in T$} if: 
\begin{itemize}
\item 
$\alpha_{ij}$ intersects $\gamma$ or $\gamma'$
\item 
neither $p_{i-1}$ or $p_{j+1}$ lie on $\gamma$ or $\gamma'$
\item 
$\alpha_{ij}$ is minimal with the above two properties (with respect to subcurve inclusion).
\end{itemize}
\end{defn}

We list all possible local segments of $\alpha$, with respect to a tagged arc $\gamma \in T$, in Figure \ref{flips}.

\begin{figure}[H]\caption{The complete list of local segments, considered up to rotations and reflections.}
\label{flips}
\begin{center}
\includegraphics[width=12cm]{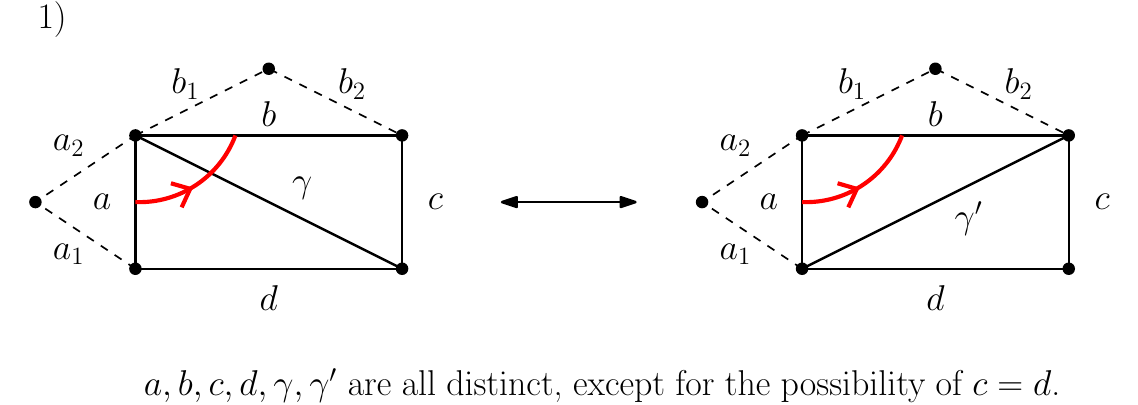}
\end{center}
\end{figure}
\begin{center}
\includegraphics[width=12cm]{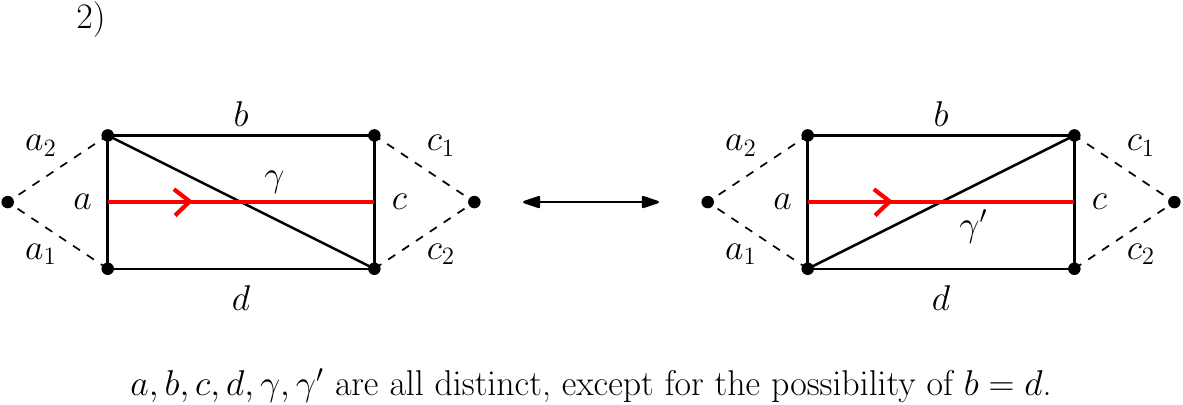}
\includegraphics[width=12cm]{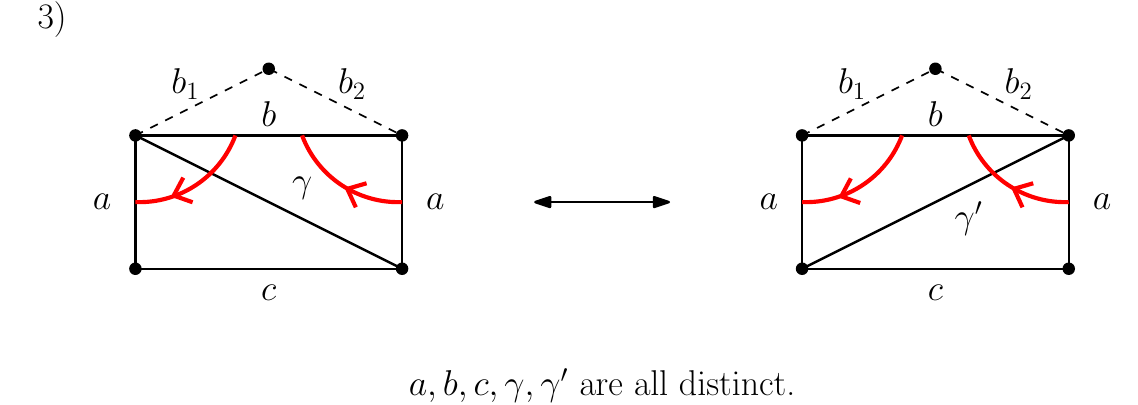}
\includegraphics[width=12cm]{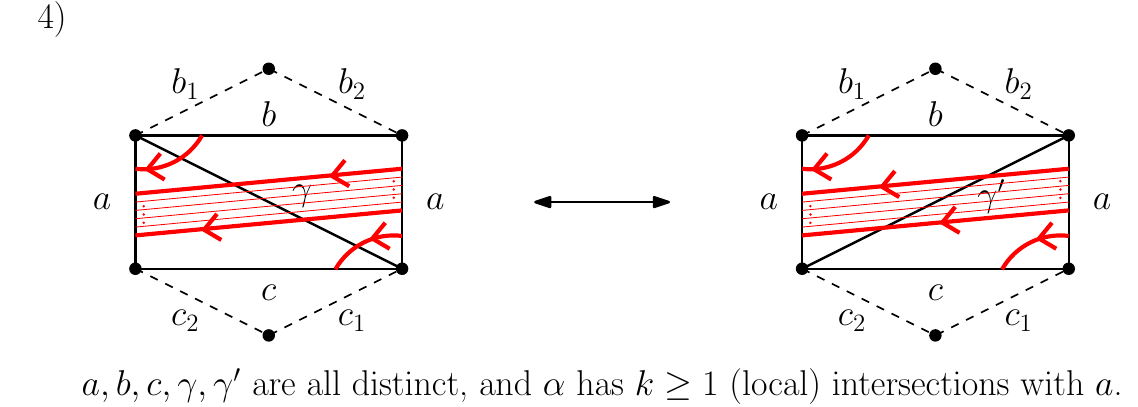}
\includegraphics[width=12cm]{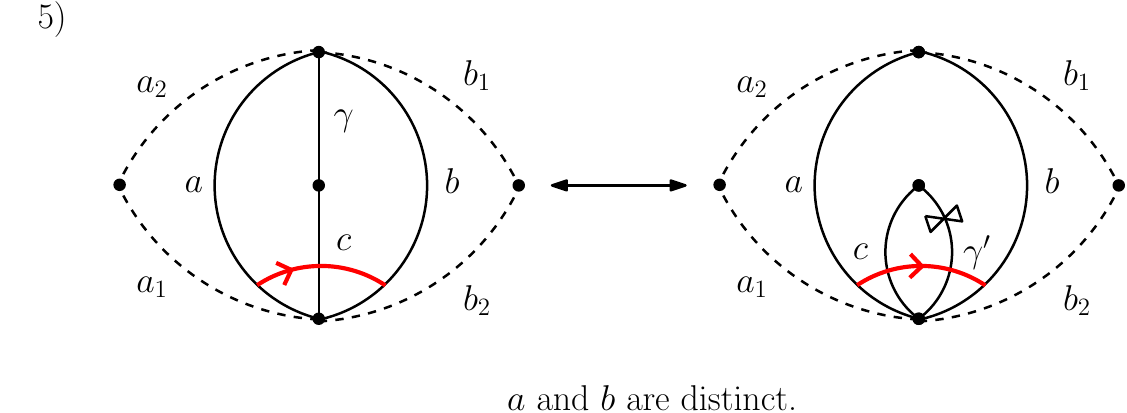}
\includegraphics[width=12cm]{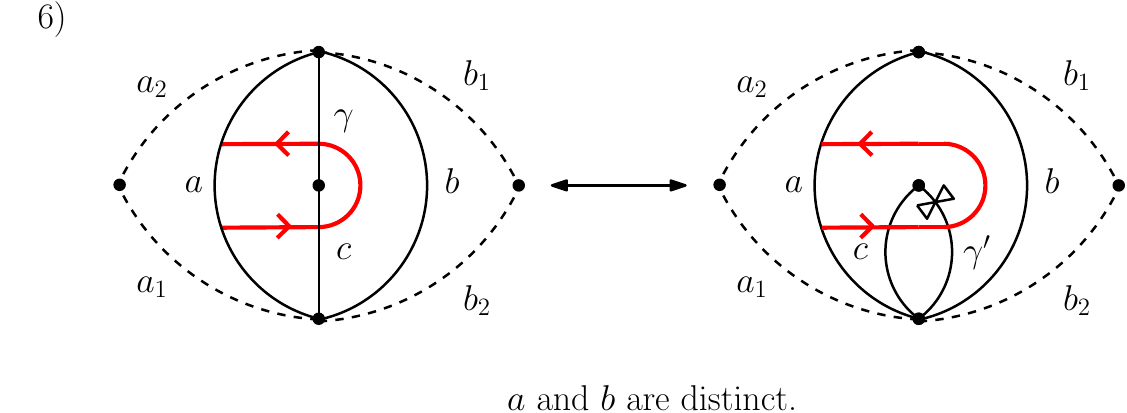}
\includegraphics[width=12cm]{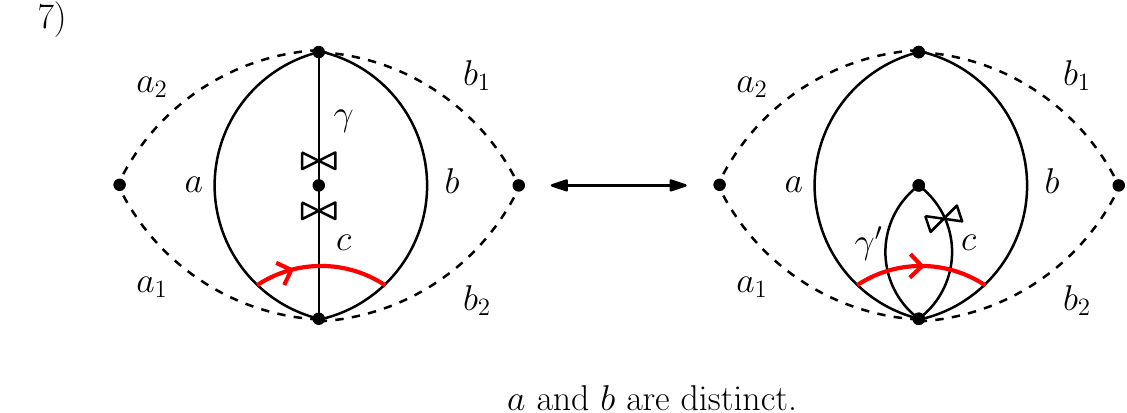}
\includegraphics[width=12cm]{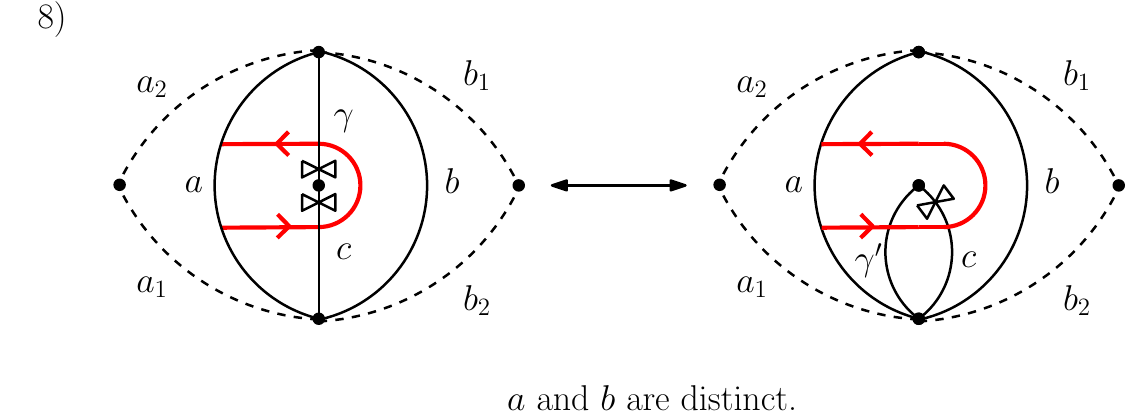}
\end{center}

\begin{rmk}
\label{disjointness} 
Note that in all of the configurations listed in Figure \ref{flips}, we have 
$$
a_1,a_2,b_1,b_2,c_1,c_2 \notin \{\gamma, \gamma'\}.
$$
\end{rmk}

\begin{rmk}

We have not included the case that $\alpha$ is a closed curve enclosed by $T$ in a cylinder $C_{1,1}$, but this is an exceptional case and can easily be checked separately.

\end{rmk}

We now list the snake graphs corresponding to the flips illustrated in Figure \ref{flips}.

\begin{figure}[H]\caption{The list of snake graphs corresponding to Figure \ref{flips}.}
\label{SG}
\begin{center}
\includegraphics[width=17cm]{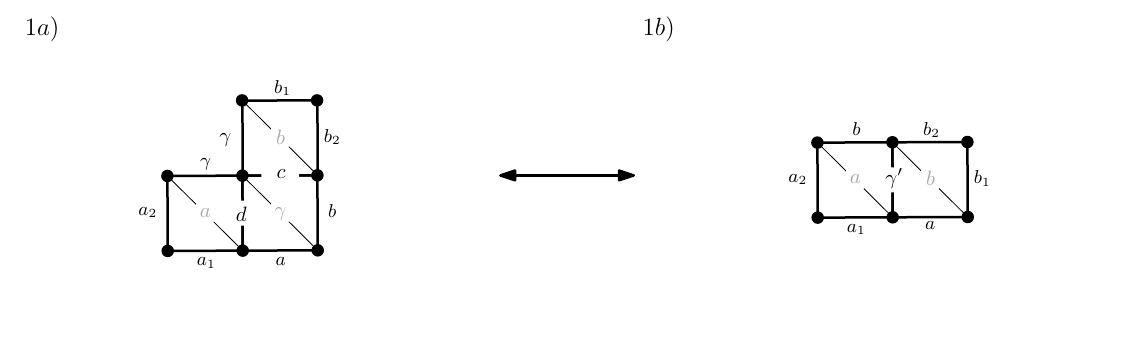}
\end{center}
\end{figure}

\begin{center}
\includegraphics[width=17cm]{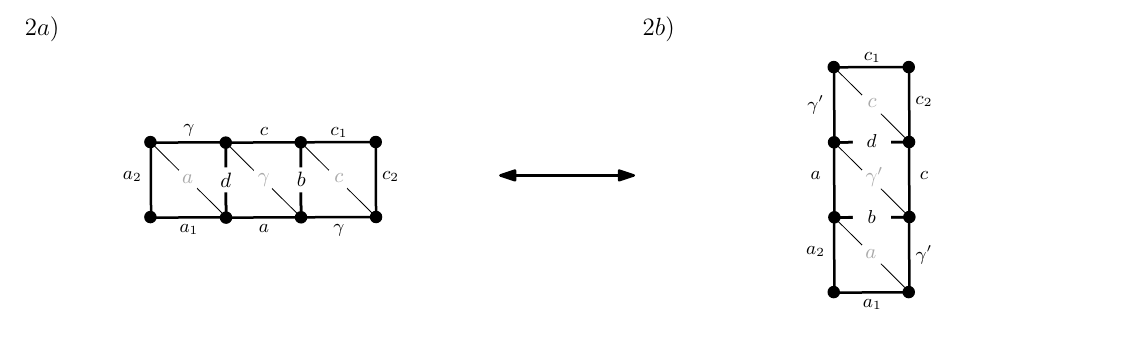}
\includegraphics[width=17cm]{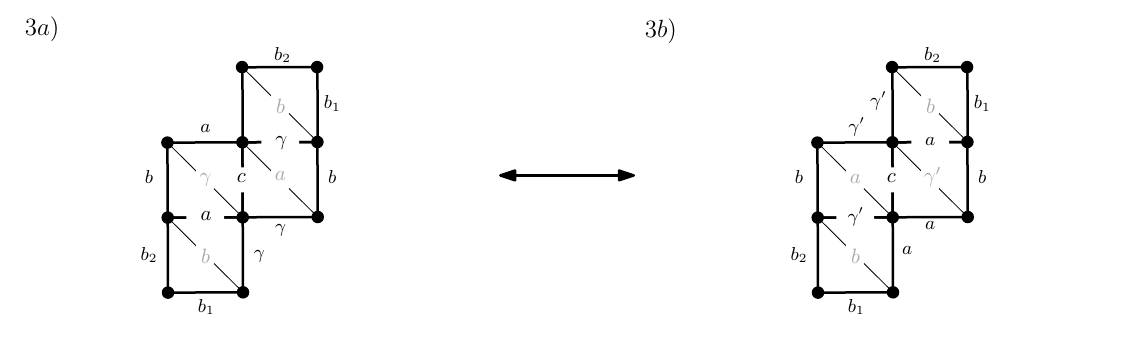}
\includegraphics[width=17cm]{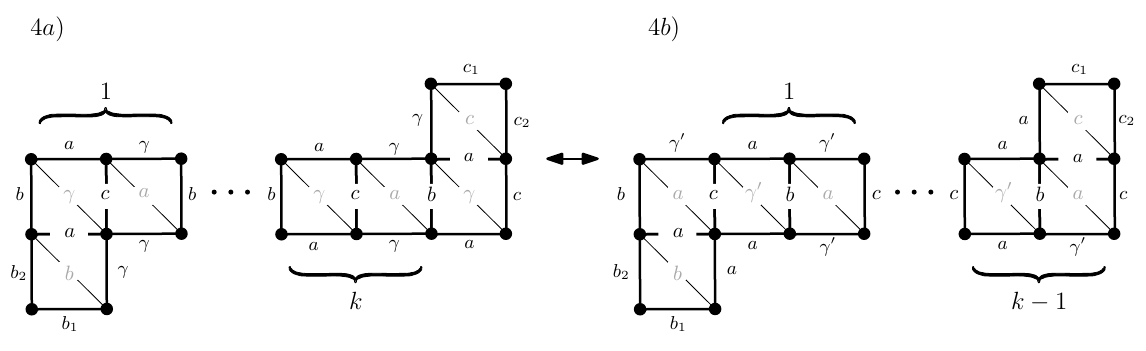}
\includegraphics[width=17cm]{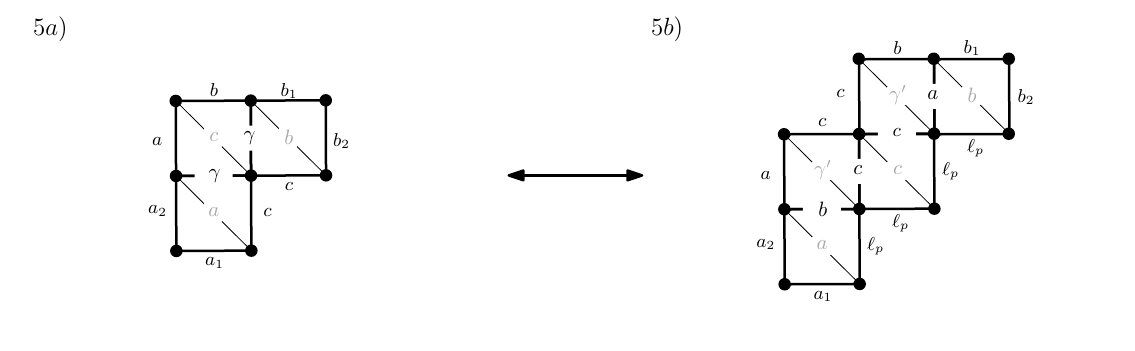}
\includegraphics[width=17cm]{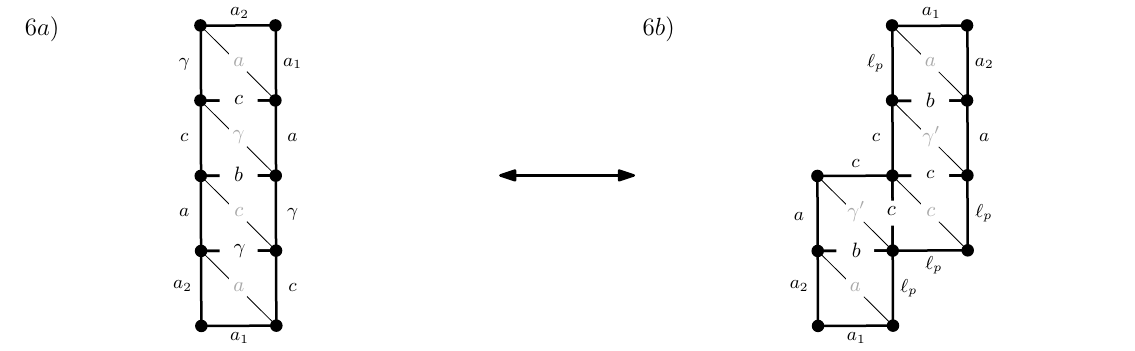}
\includegraphics[width=17cm]{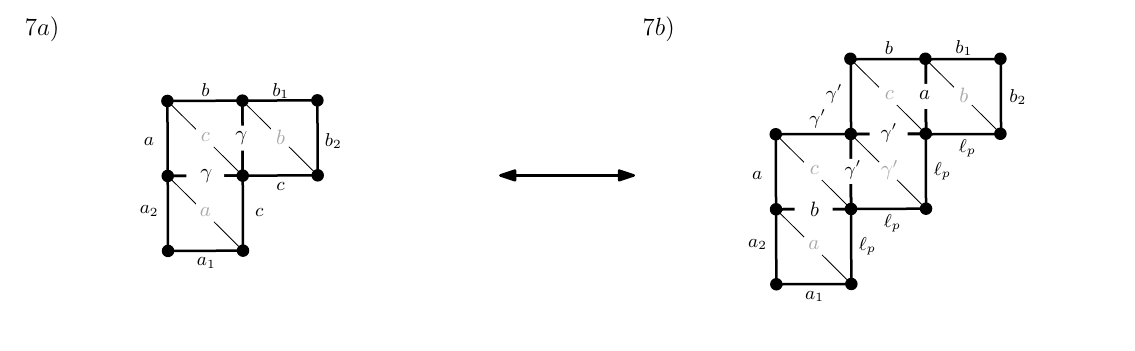}
\includegraphics[width=17cm]{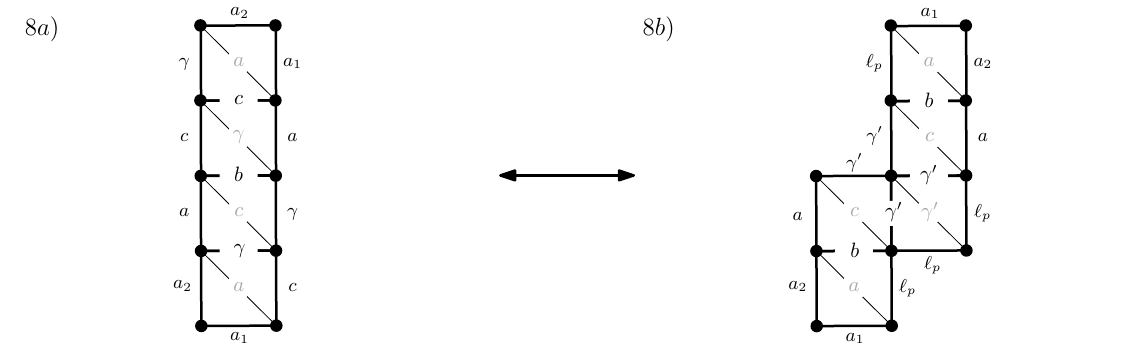}
\end{center}

As mentioned above, the strategy in proving Theorem \ref{thm: comb key lemma} is to reduce the problem to checking that the theorem holds on the flips listed in Figure \ref{flips}. To do this, we shall need a better understanding of the associated $\mathbf{h}$-vectors.

\subsection{Local \texorpdfstring{$\mathbf{h}$}{h}-vectors}
\begin{defn}
Let $\alpha_{ij}$ be a subcurve of $\alpha$ present in Figure \ref{flips}. One can define a snake graph corresponding to the configuration, and thus obtain a \emph{local} snake $F$-polynomial $F_{G(T,\alpha_{ij})}$. We call the $\mathbf{h}$-vector $\mathbf{h}_{G(T,\alpha_{ij})}$ defined through $F_{G(T,\alpha_{ij})}$ the \emph{local snake $\mathbf{h}$-vector} of $\alpha$ with respect to $\alpha_{ij}$ and $T$.
\end{defn}

The idea is to show that local snake $\mathbf{h}$-vectors naturally determine the whole snake $\mathbf{h}$-vector.

\begin{prop}
\label{local hsum}
Let $\mathcal{L}$ be the set of all local segments of $\alpha$ with respect to $\gamma \in T$. Then for each $k \in \{1,\ldots, n\}$ we have the following equality: $$h_{\alpha; k} = \displaystyle \sum_{\alpha_{ij} \in \mathcal{L}} h_{\alpha_{ij}; k}.$$
\end{prop}
To prove Proposition \ref{local hsum}, we first show:
\begin{lem}
\label{leq}
Keeping the terminology used in Proposition \ref{local hsum}, the following inequality holds.
\[
\sum_{\alpha_{ij} \in \mathcal{L}} h_{\alpha_{ij}; k} \leq h_{\alpha; k}.
\]
\end{lem}

\begin{proof}
Let $\gamma$ be any arc in $T$. Note that if $d$ is a diagonal of $G(T,\alpha)$ such that $b_{d\gamma} \neq 0$ then that diagonal appears as the diagonal of some tile in $G(T,\alpha_{ij})$. The lemma then follows by Remarks \ref{rem:snake-h-vector-is-non-positive} and \ref{disjointness}.

\end{proof}

By Lemma \ref{leq}, to prove Proposition \ref{local hsum} we need to show that there exists a good matching $P$ of $G(T,\alpha)$ such that the restriction of $P$ to any $G(T,\alpha_{ij})$ of $G(T,\alpha)$ achieves the corresponding local $h_k$. Namely, we shall introduce the idea of \emph{gluing} together perfect matchings. With respect to this notion of gluing, the basic idea is show that for any $G(T,\alpha_{ij})$ and $G(T,\alpha_{lm})$, there exists perfect matchings $P_{ij}$ and $P_{lm}$, respectively, which can be glued together and which achieve their corresponding local $h_k$'s.

\subsection{Extendability of perfect matchings}

\begin{defn}
\label{extend}
Let $\mathcal{H} = [H_1,\ldots, H_s]$ and $\mathcal{K} = [K_1,\ldots, K_t]$ be snake graphs contained in a band graph $\mathcal{G}$, such that $H_i \neq K_j$ for all $i,j$. Let $P_{\mathcal{H}}$ and $P_{\mathcal{K}}$ be perfect matchings of $\mathcal{H}$ and $\mathcal{K}$, respectively. We say that $P_{\mathcal{H}}$ is \emph{extendable} to $P_{\mathcal{K}}$ in $\mathcal{G}$ if there exists a perfect matching $P$ of the union $[H_1, H_2, \ldots, K_{t-1}, K_t]$ such that the orientation induced by $P$ on the diagonals of $\mathcal{H}$ and $\mathcal{K}$ agrees with the orientation induced by $P_{\mathcal{H}}$ and $P_{\mathcal{K}}$, respectively.

\end{defn}

\begin{figure}[H]
\caption{An example and non-example of extendability. With respect to Definition \ref{extend}, the tiles of $\mathcal{H}$ and $\mathcal{K}$ are shaded green and blue, respectively.}
\label{glueability}
\begin{center}
\includegraphics[width=13.5cm]{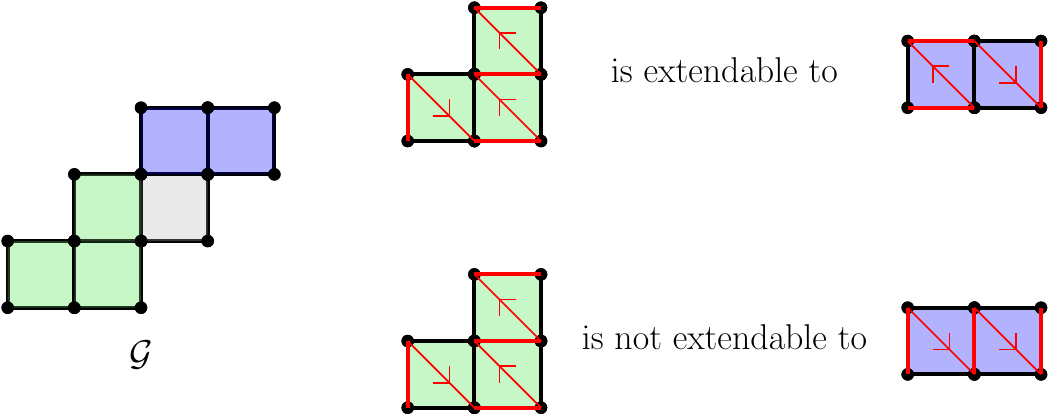}
\end{center}
\end{figure}

\begin{defn}

Let $P$ be a perfect (resp. good) matching of a snake (resp. band) graph~$\mathcal{G}$, and let $T$ be a tile of $\mathcal{G}$. We say that the diagonal of $T$ is \emph{positive} with respect to $P$ if one of the following holds: \begin{itemize}
\item 
$T$ is an odd tile and the diagonal is oriented downwards;
\item 
$T$ is an even tile and the diagonal is oriented upwards.
\end{itemize}
Otherwise we say that the diagonal of $T$ is \emph{negative}.
\end{defn}

\begin{lem}
\label{negative diagonals}
If $P_{\mathcal{H}}$ and $P_{\mathcal{K}}$ induce negative diagonals on the tiles $H_s$ and $K_1$, respectively, then $P_{\mathcal{H}}$ is extendable to $P_{\mathcal{K}}$.

\end{lem}

\begin{proof}

This follows from the existence of a minimal matching on $[H_s, \ldots, K_1]$. 
\end{proof}

\begin{proof}[Proof of Proposition \ref{local hsum}]
Lemma \ref{negative diagonals} enables us to prove Proposition \ref{local hsum} with the following strategy. We shall show that for each $G(T,\alpha_i)$ appearing in Figure \ref{SG} there exists a perfect matching $P_i$ such that: 
\begin{itemize}
\item 
$P_i$ achieves the corresponding local $h_k$;
\item 
$P_i$ induces a negative diagonal on the first and last tile of $G(T,\alpha_i)$.
\end{itemize}

\underline{\textbf{Case 1a)}}: Here $b_{\gamma a} = -1$ and $b_{\gamma b} =1$. If the diagonal $\gamma$ is positive then so is the diagonal $a$. Therefore $h_k = 0$. An alternative way to achieve $h_k = 0$ is taking $P_i$ to be the minimal matching. \newline

\underline{\textbf{Case 1b)}}: There is no diagonal labeled by $\gamma'$ so $h_k \geq 0$. Taking $P_i$ to be the minimal matching achieves $h_k = 0$. \newline

\underline{\textbf{Case 2a)}}: Here $b_{\gamma a} = b_{\gamma c} = -1$. If the diagonal $\gamma$ is positive then so are the diagonals $a$ and $c$, hence $h_k = 0$. So take $P_i$ to be the minimal matching. \newline

\underline{\textbf{Case 2b)}}: Here $b_{\gamma' a} = b_{\gamma' c} = 1$. Hence $h_k = -1$ is achieved by taking $P_i$ such that the diagonal $\gamma'$ is positive, and all diagonals are negative. \newline

\underline{\textbf{Case 3a)}}: Here $b_{\gamma a} = -2$ and $b_{\gamma b} = 1$. If the diagonal $\gamma$ is positive then so is $a$, hence $h_k \geq 0$. Taking $P_i$ to be the minimal matching achieves $h_k = 0$. \newline

\underline{\textbf{Case 3b)}}: Here $b_{\gamma' a} = 2$ and $b_{\gamma' b} = -1$. If the diagonal $\gamma'$ is positive then so is the subsequent diagonal $b$, hence $h_k \geq 0$. Taking $P_i$ to be the minimal matching achieves $h_k = 0$. \newline

\underline{\textbf{Case 4a)}}: Here $b_{\gamma a} = -2$ and $b_{\gamma b} = b_{\gamma c} = 1$. If any diagonal $\gamma$ is positive then so is the adjacent diagonal $a$, hence $h_k \geq 0$. Taking $P_i$ to be the minimal matching achieves $h_k = 0$. \newline

\underline{\textbf{Case 4b)}}: Here $b_{\gamma' a} = 2$ and $b_{\gamma' b} = b_{\gamma' c} = -1$. Taking $P_i$ such that diagonals $b$ and $c$ are negative, and all other diagonals are positive shows $h_k \leq 1-k$. Moreover, since there are only $k-1$ diagonals labelled by $\gamma'$ then $h_k = 1-k$. \newline

\underline{\textbf{Case 5a)}}: There is no diagonal labelled $\gamma$ so $h_k = 0$, and we can take $P_i$ to be the minimal matching. \newline

\underline{\textbf{Case 5b)}}: Here $b_{\gamma' b} = -1$ and $b_{\gamma' a} = 1$. If $\gamma'$ is positive then so is $b$, hence $h_k = 0$, and we can take $P_i$ to be the minimal matching. \newline

\underline{\textbf{Case 6a)}}: Here $b_{\gamma a} = -1$. If $\gamma$ is positive then so is the subsequent diagonal $a$, hence $h_k = 0$, and we can take $P_i$ to be the minimal matching. \newline

\underline{\textbf{Case 6b)}}: If both diagonals $\gamma$ are positive, then so is the diagonal $c$. Thus, recalling the definition of the height function, we have that $h_k \geq -1$. Taking $P_i$ such that both diagonals $a$ are negative, and all other diagonals are positive achieves $h_k = -1$. \newline

\underline{\textbf{Case 7a)}}: There is no diagonal labelled $\gamma$ so $h_k = 0$, and we can take $P_i$ to be the minimal matching. \newline

\underline{\textbf{Case 7b)}}: Here $b_{\gamma' b} = -1$ and $b_{\gamma' a} = 1$. If the diagonal $\gamma'$ is positive then so is the diagonal $b$, hence $h_k = 0$ and we can take $P_i$ to be the minimal matching. \newline

\underline{\textbf{Case 8a)}}: Here $b_{\gamma a} = -1$. If the diagonal $\gamma$ is positive then so is the subsequent diagonal $a$, hence $h_k = 0$ and we can take $P_i$ to be the minimal matching. \newline

\underline{\textbf{Case 8b)}}: Here $b_{\gamma' a},b_{\gamma' c} \leq 0$. Moreover, there is precisely one tile with diagonal $\gamma'$, so $h_k = -1$.  Taking $P_i$ such that the diagonal $\gamma'$ and the subsequent diagonal $c$ are positive, and all other diagonals negative achieves $h_k = -1$. \newline

This completes the proof of Proposition \ref{local hsum}.
\end{proof}

\subsection{Proof of equation \texorpdfstring{\eqref{combinatorial Fmutation}}{(5.4)} in the Combinatorial Key Lemma}

\begin{thm}
\label{full comb key lemma}
Equation \eqref{combinatorial Fmutation} in Theorem \ref{thm: comb key lemma} holds true.
\end{thm}

\begin{proof}
Without loss of generality, 
suppose that the signature of $T$ is non-negative, and identify $T$ with the ideal triangulation $T^\circ$ representing it.
Consider all the local segments of $\alpha$ with respect to the flipping arc $\gamma:=\tau_k$, cf. Subsection \ref{subsec:local-segments}. Some of the arcs lying on the region surrounding these local segments, cf. Figure \ref{flips}, may be folded sides of self-folded triangles, even the dotted ones. Suppose first that $n>0$ such arcs are folded sides. We let $d_1, \ldots, d_m$ denote these folded sides, allowing repetition. For convenience, we also order $d_1, \ldots, d_m$ in the order that $\alpha$ intersects them (up to cyclic permutation). For each $j=1,\ldots,m$, we let $d_j^\circ\in T^\circ$ denote the loop that together with $d_j$ forms the corresponding self-folded triangle of $T^\circ$.

The~sum 
\[
\sum_{P \in \mathcal{P}(G(T,\alpha))} \overline{y}(P)
\]
decomposes as the sum of $2^{m}$ collections of perfect matchings -- each of these sums corresponding to a choice of orientation on each of the diagonals $d_1, \ldots, d_m$. Some of these sums may be empty; by convention we define an empty sum to be $0$. Specifically, if we let $(d_{i_1}, \ldots, d_{i_k})$ denote the collection of good matchings which induce positive orientation on $d_{i_1}, \ldots, d_{i_k}$ and negative orientation on the remaining diagonals, then we have the~equality
\[
\sum_{P \in \mathcal{P}(G(T,\alpha))} \overline{y}(P) =  
\sum_{(d_{i_1}, \ldots, d_{i_k})} \Big( \sum_{P \in (d_{i_1}, \ldots, d_{i_k})} \overline{y}(P) \Big).
\]
Moreover for any $(d_{i_1}, \ldots, d_{i_k})$ we have 
\begin{equation}
\label{product T}
\displaystyle \sum_{P \in (d_{i_1}, \ldots, d_{i_k})} \overline{y}(P) = \Big(\prod_{i=1}^{m} \Big(\sum_{P \in (d_{i_1}, \ldots, d_{i_k})_{i,i+1}} y(P)\Big)\Big) \frac{y_{d_{i_1}} \ldots y_{d_{i_k}}}{y_{d_{i_1}^{\circ}} \ldots y_{d_{i_k}^{\circ}}},
\end{equation}
where $(d_{i_1}, \ldots, d_{i_k})_{i,i+1}$ denotes the collection of perfect matchings, between the tiles $d_i$ and $d_{i+1}$, which induce the orientations on $d_i$ and $d_{i+1}$ dictated by $(d_{i_1}, \ldots, d_{i_k})$. By convention we define $(d_{i_1}, \ldots, d_{i_k})_{m,m+1} := (d_{i_1}, \ldots, d_{i_k})_{m,1}$.

Analogously, if we look at $T'$ we get 
\begin{equation}
\label{product T'}
\displaystyle \sum_{P \in (d_{i_1}, \ldots, d_{i_k})'} \overline{y}'(P) = \Big(\prod_{i=1}^{m} \Big(\sum_{P \in (d_{i_1}, \ldots, d_{i_k})'_{i,i+1}} y'(P)\Big)\Big) \frac{y_{d_{i_1}}' \ldots y_{d_{i_k}}'}{y_{d_{i_1}^{\circ}}' \ldots y_{d_{i_k}^{\circ}}'}
\end{equation}
where the collection $(d_{i_1}, \ldots, d_{i_k})'$ is defined in the same way as above, but now with respect to $\mathcal{P}(G(T',\alpha))$, rather than $\mathcal{P}(G(T,\alpha))$.

By Proposition \ref{local hsum} and the explicit computations carried out in Figure \ref{fig: SGmatching}, the proof of the identity 
\[
(y_k+1)^{h_k}\cdot
\sum_{P \in (d_{i_1}, \ldots, d_{i_k})} \overline{y}(P) = (y_k'+1)^{h_k'}\cdot 
\sum_{P \in (d_{i_1}, \ldots, d_{i_k})'} \overline{y}'(P)
\]
is reduced to showing that
$
\frac{y'_{d_{i_1}} \ldots y'_{d_{i_k}}}{y'_{d_{i_1}^{\circ}} \ldots y'_{d_{i_k}^{\circ}}} = \frac{y_{d_{i_1}} \ldots y_{d_{i_k}}}{y_{d_{i_1}^{\circ}} \ldots y_{d_{i_k}^{\circ}}}.
$
This follows from the equality $b_{\gamma \hspace{0.2mm} d_i} = b_{\gamma \hspace{0.2mm} d_i^{\circ}}$.
\begin{figure}[H]\caption{``Local'' verification of equation \eqref{combinatorial Fmutation} in the proof of Theorem \ref{full comb key lemma}.
}
\label{fig: SGmatching}
\begin{center}
\includegraphics[width=15cm]{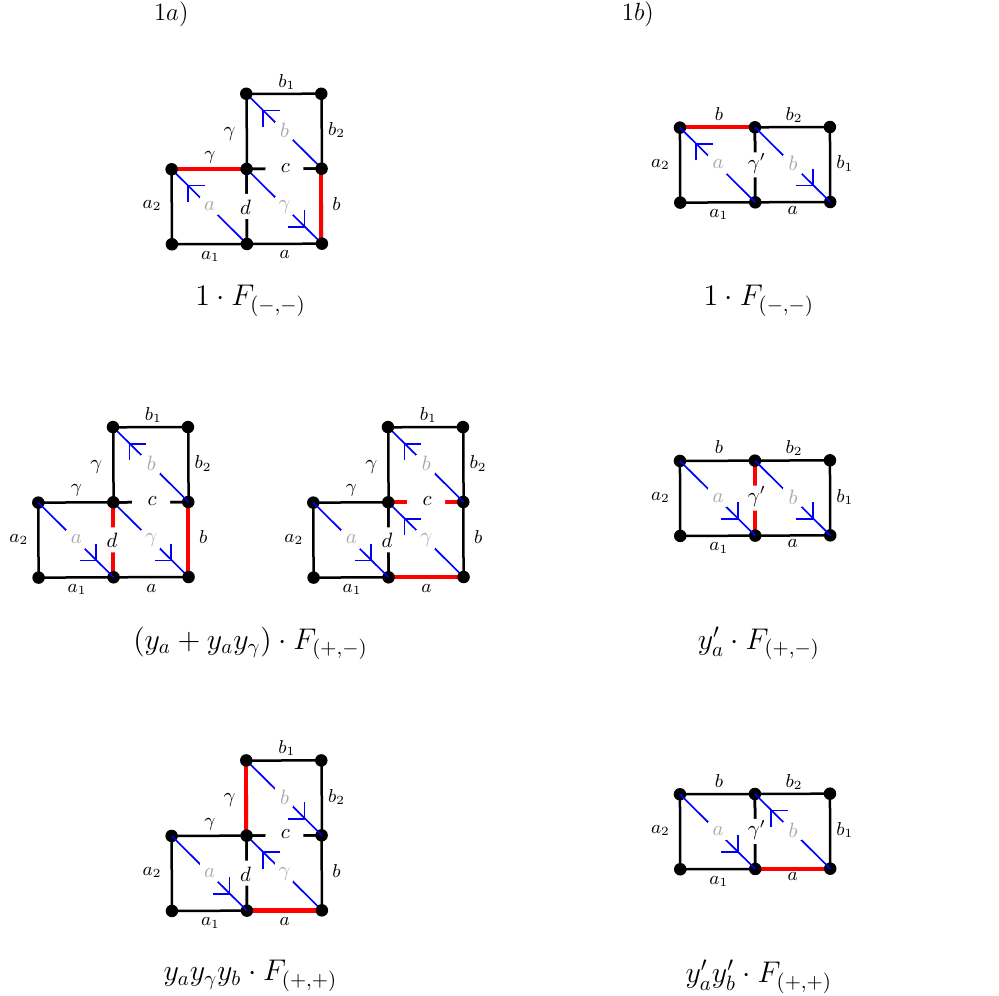}
\end{center}
\end{figure}

\begin{center}
\includegraphics[width=15cm]{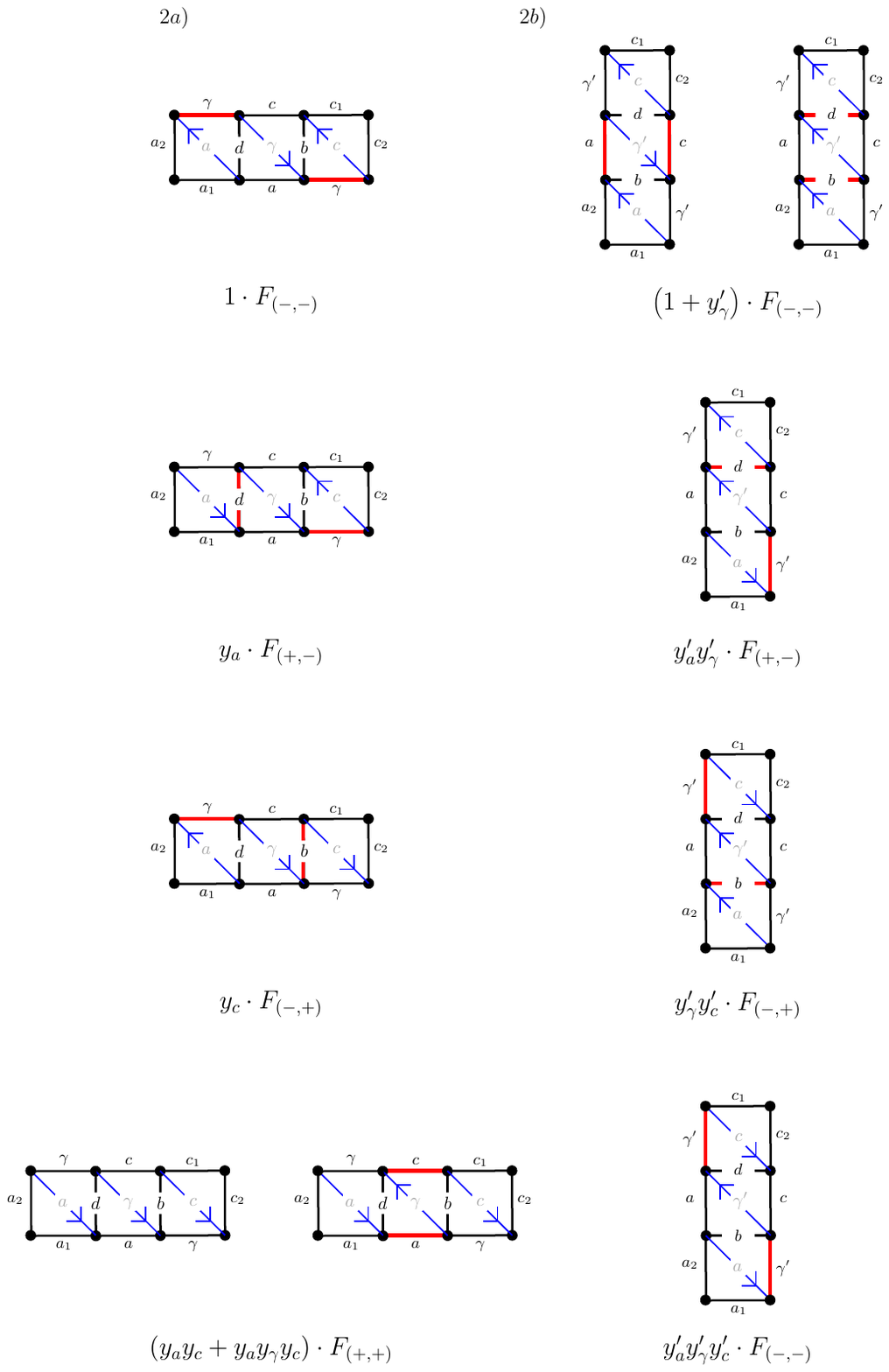} \newpage
\includegraphics[width=15cm]{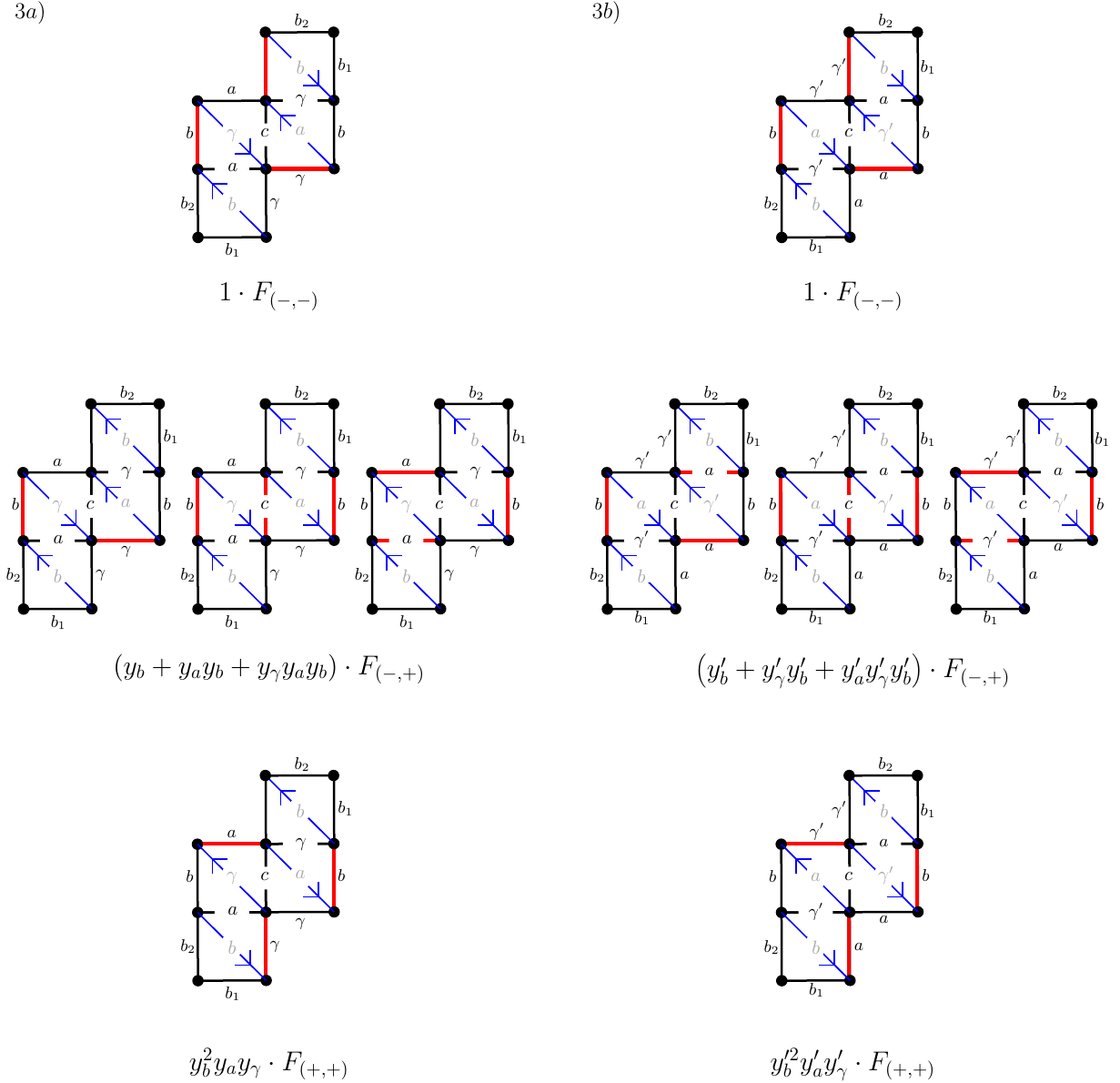} \newpage
\includegraphics[width=15cm]{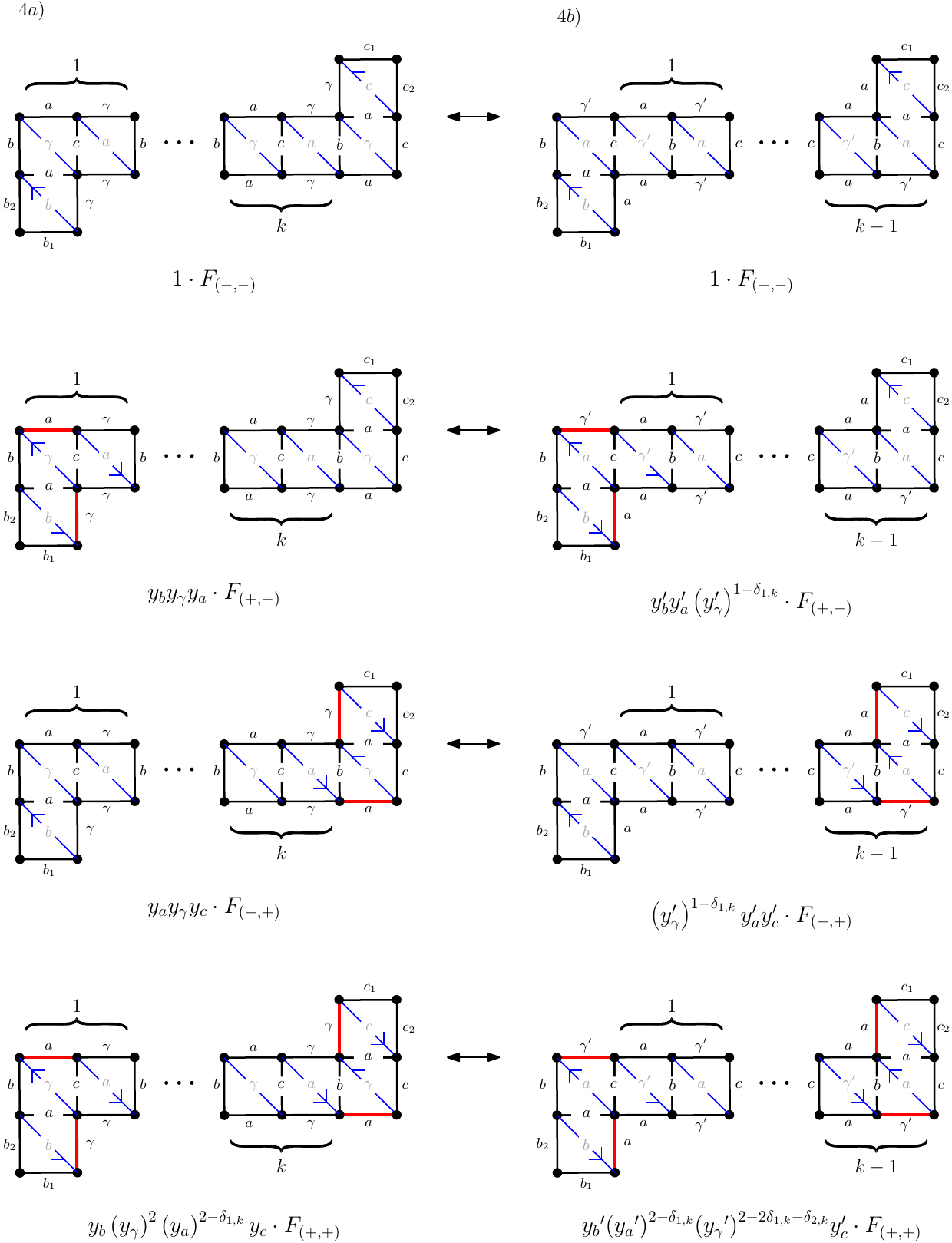} \newpage
\includegraphics[width=15cm]{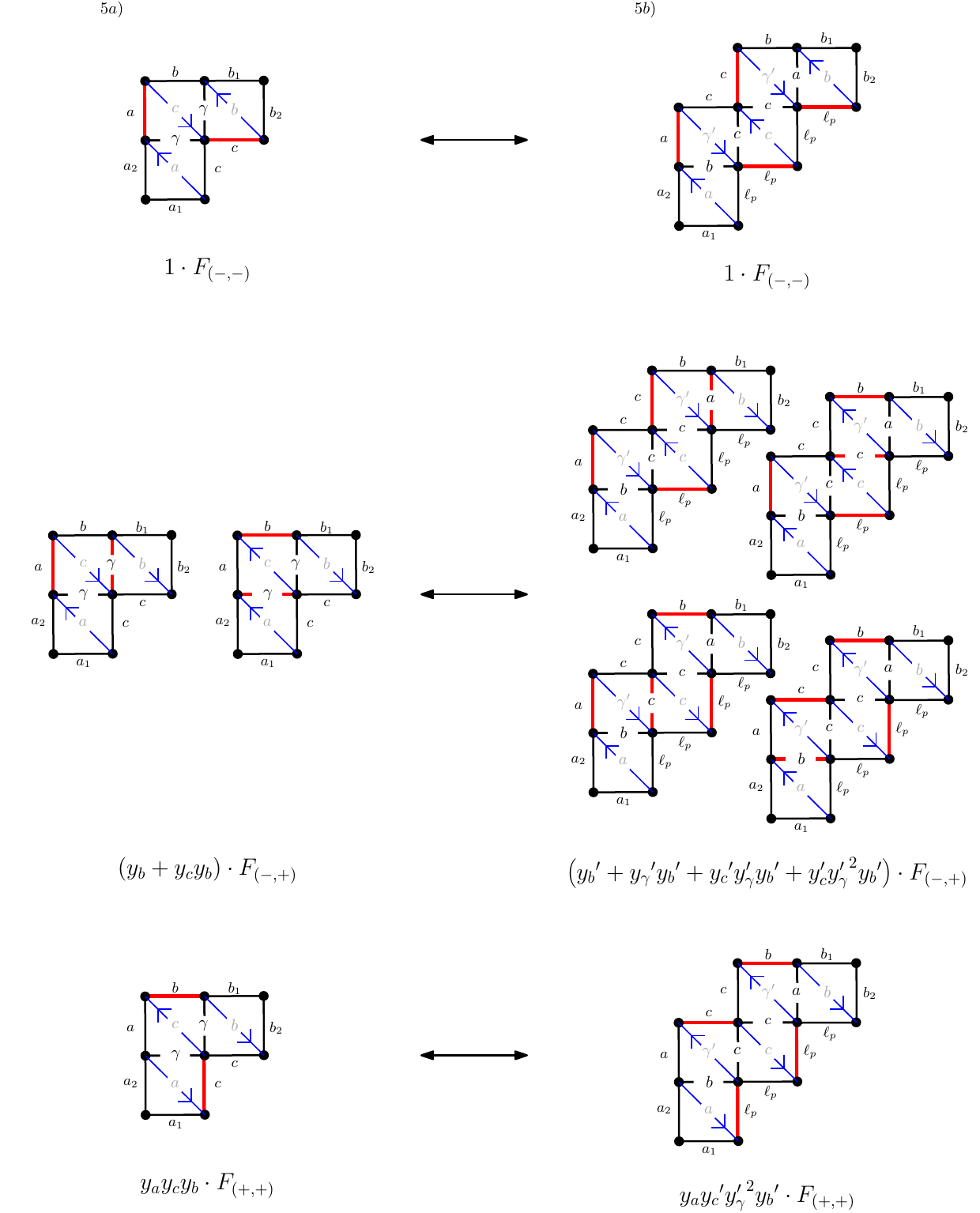} \newpage
\includegraphics[width=15cm]{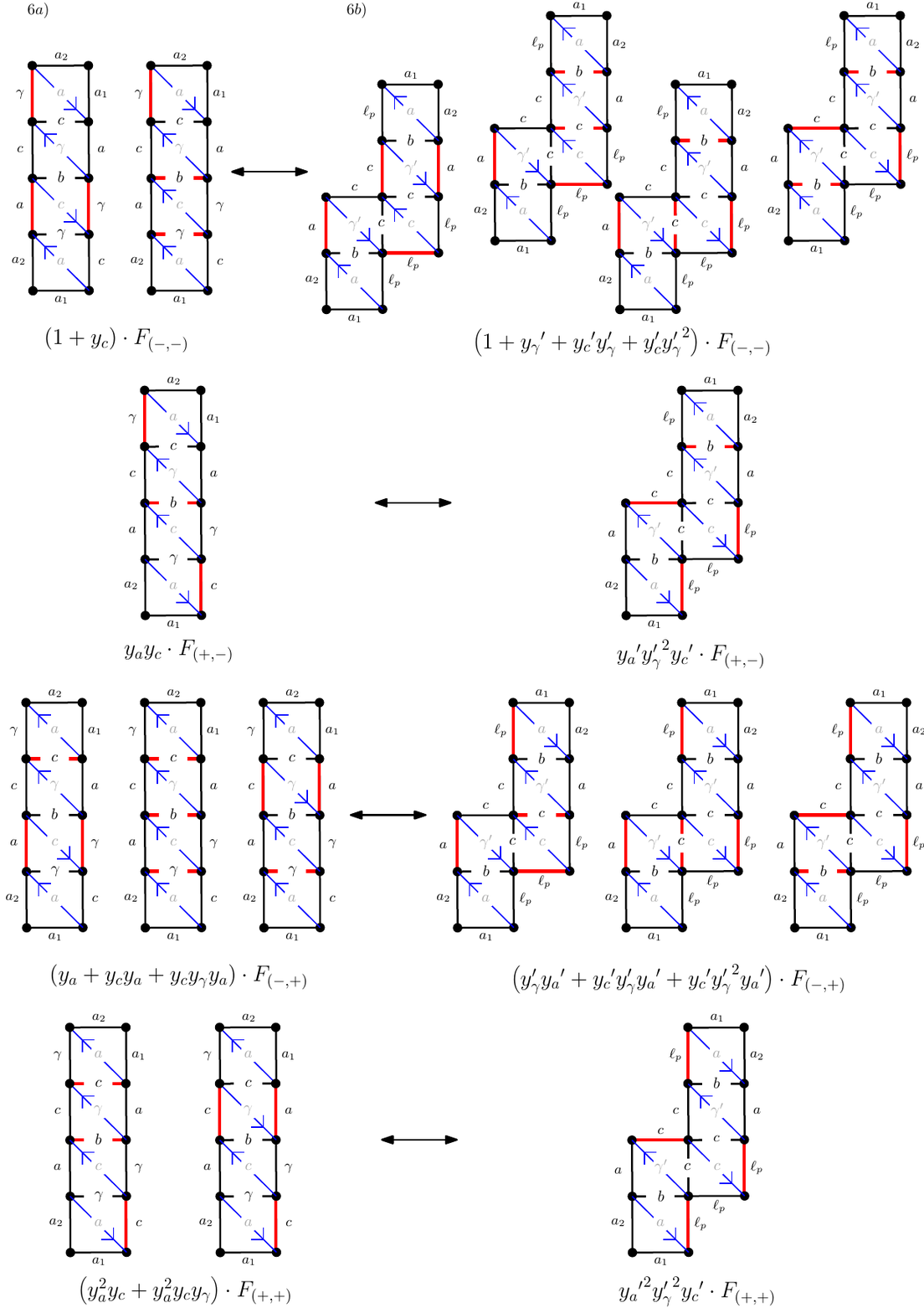} \newpage
\includegraphics[width=15cm]{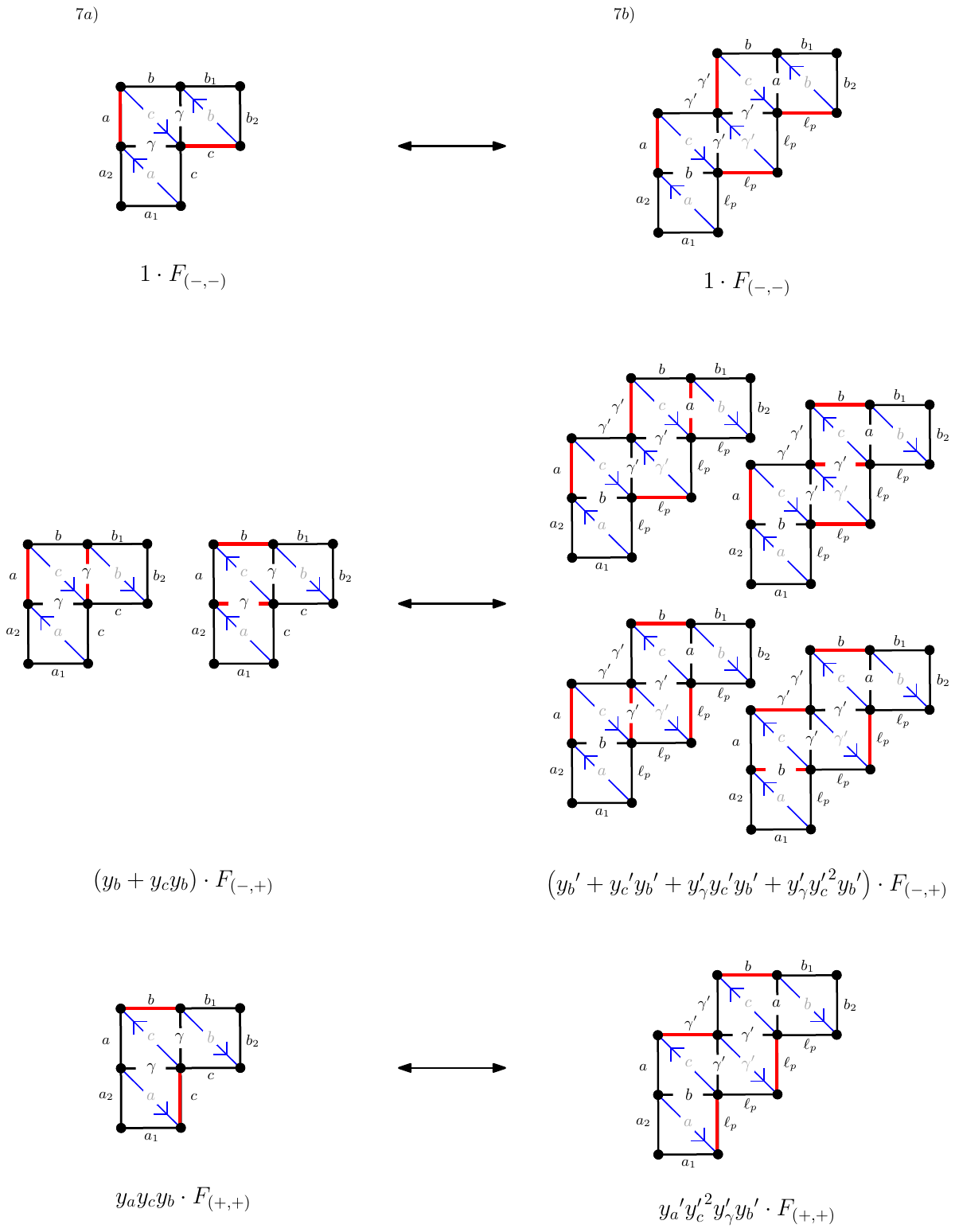} \newpage
\includegraphics[width=14cm]{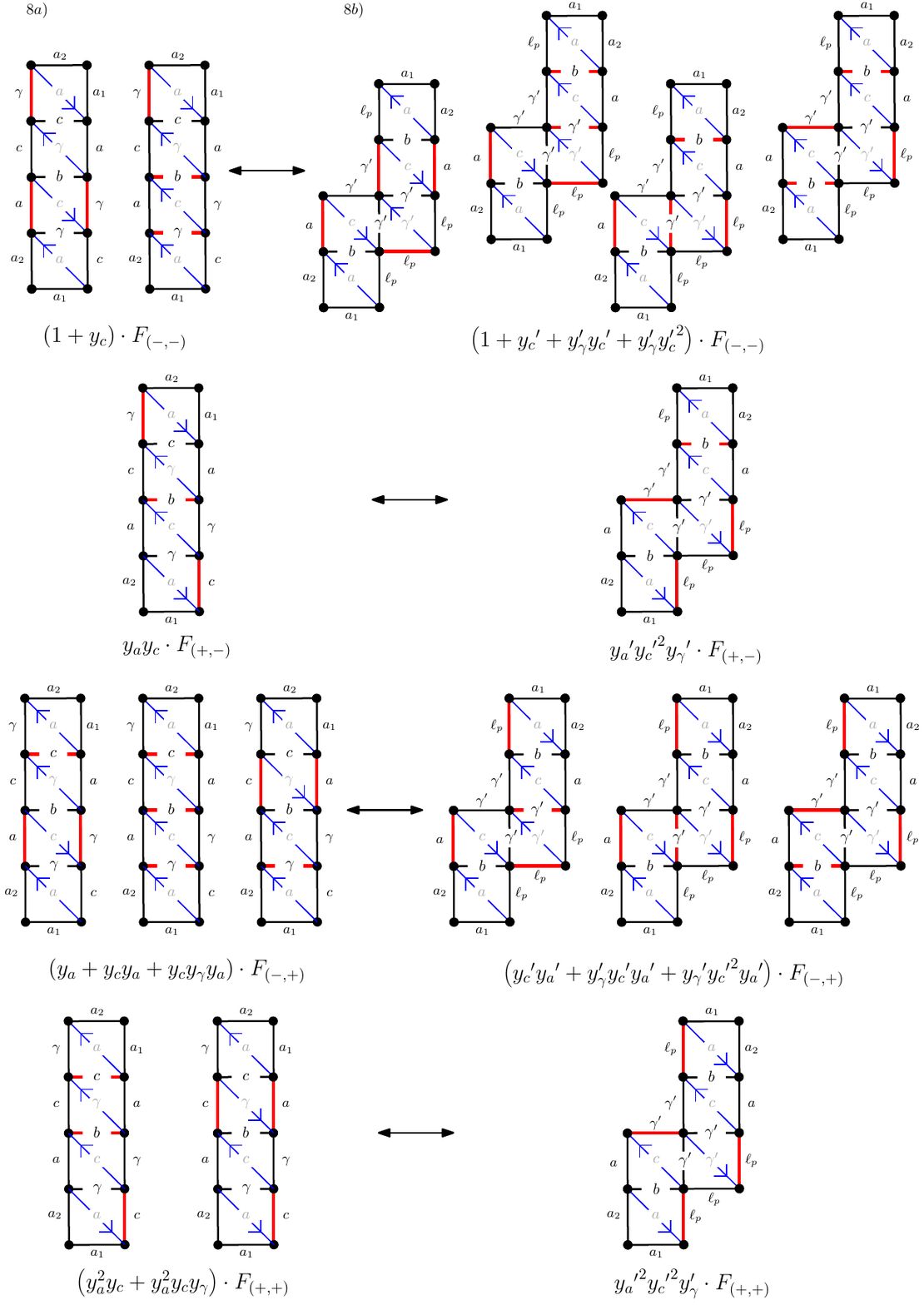}
\end{center}

When none of the boundary edges of the flip regions encountered corresponds to self-folded triangles, the proof of equation \eqref{combinatorial Fmutation}, is easier, because, in that case, the snake subgraphs of the band graphs $G(T,\alpha)$ that arise from the local segment do not interfere with each other at all. The details for that case are left to the reader.
\end{proof}

\subsection{Local \texorpdfstring{$\mathbf{g}$}{g}-vectors}

In this section we show that there are six fundamental configurations that dictate how snake $\textbf{g}$-vectors change under mutation -- these are listed below in Figure~\ref{local gvector flips}.  
\begin{figure}[H]\caption{Here we list all six combinatorial types of local curves for computing snake $\textbf{g}$-vectors, together with their corresponding graphs. The red edges arise from the associated minimal matching.}
\label{local gvector flips}

\end{figure}

\begin{center}
\includegraphics[width=11cm]{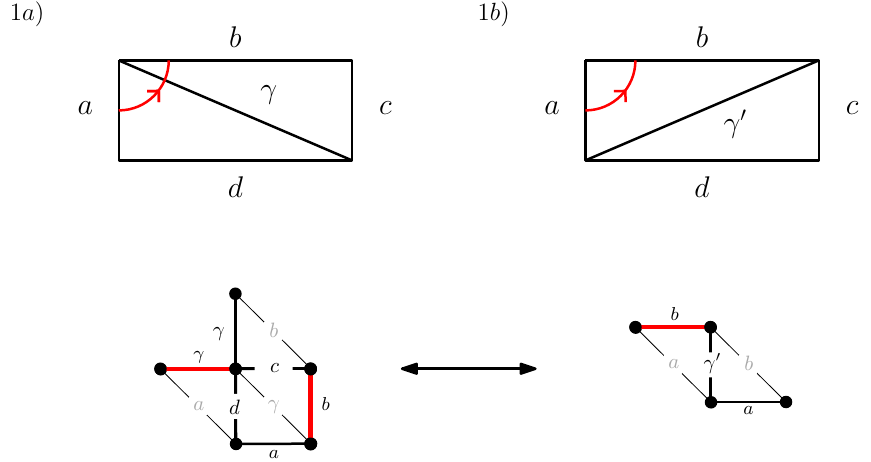}
\includegraphics[width=11cm]{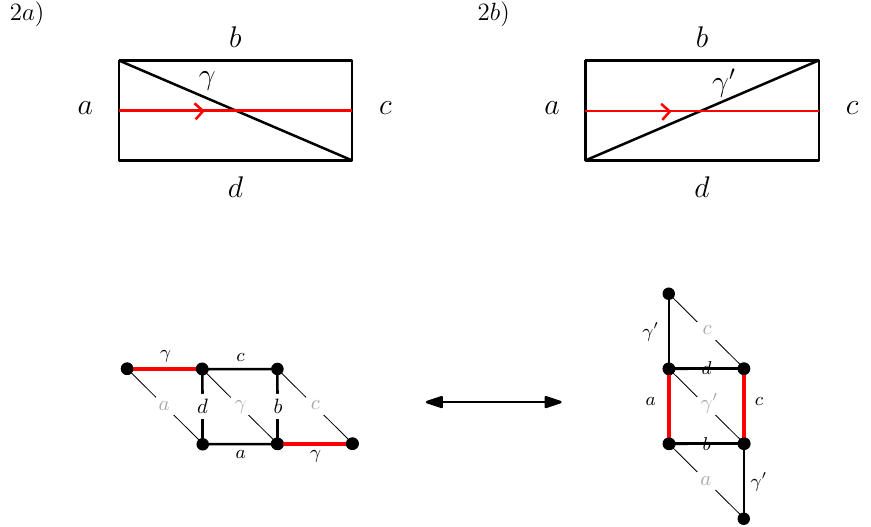}
\includegraphics[width=11cm]{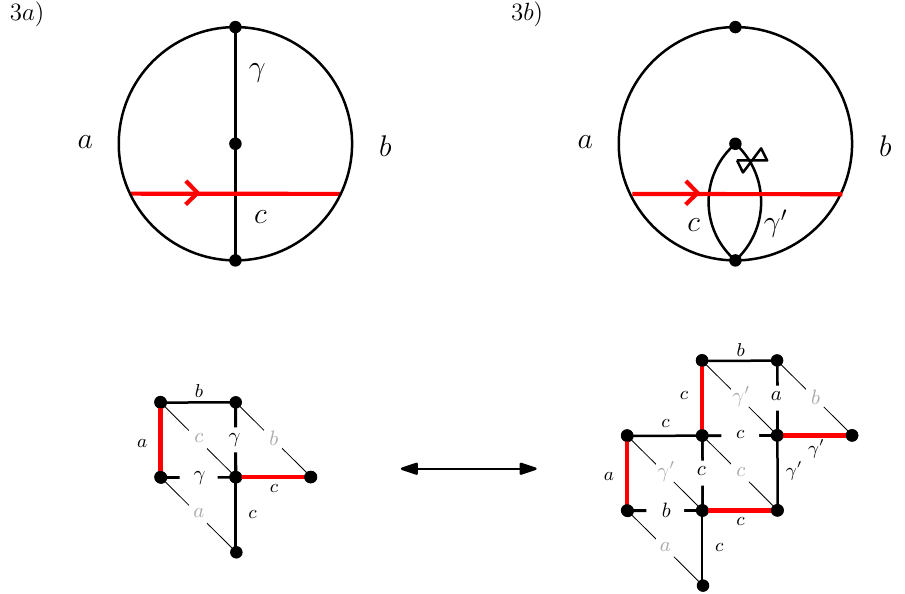}
\includegraphics[width=11cm]{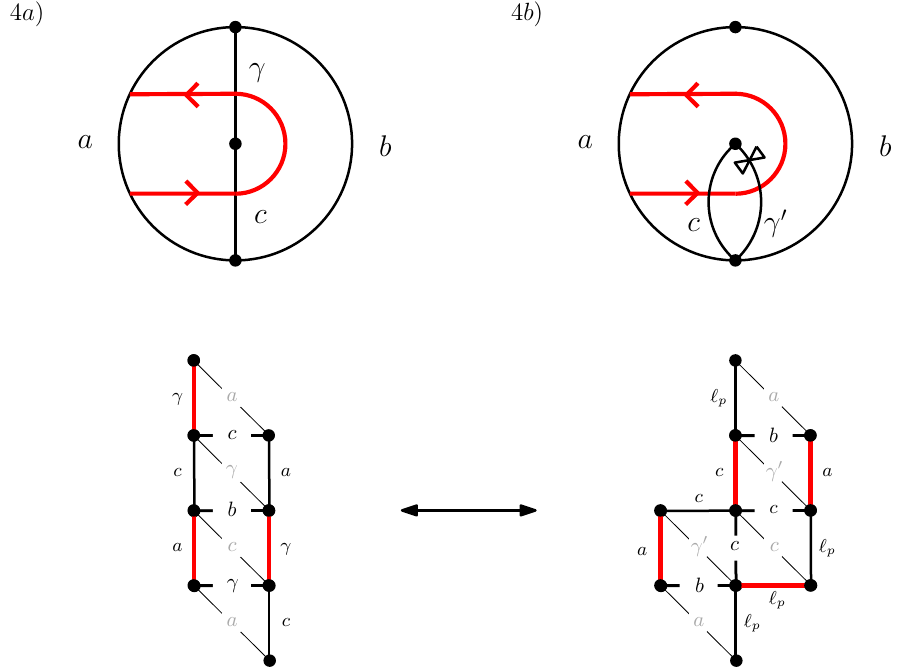}
\includegraphics[width=11cm]{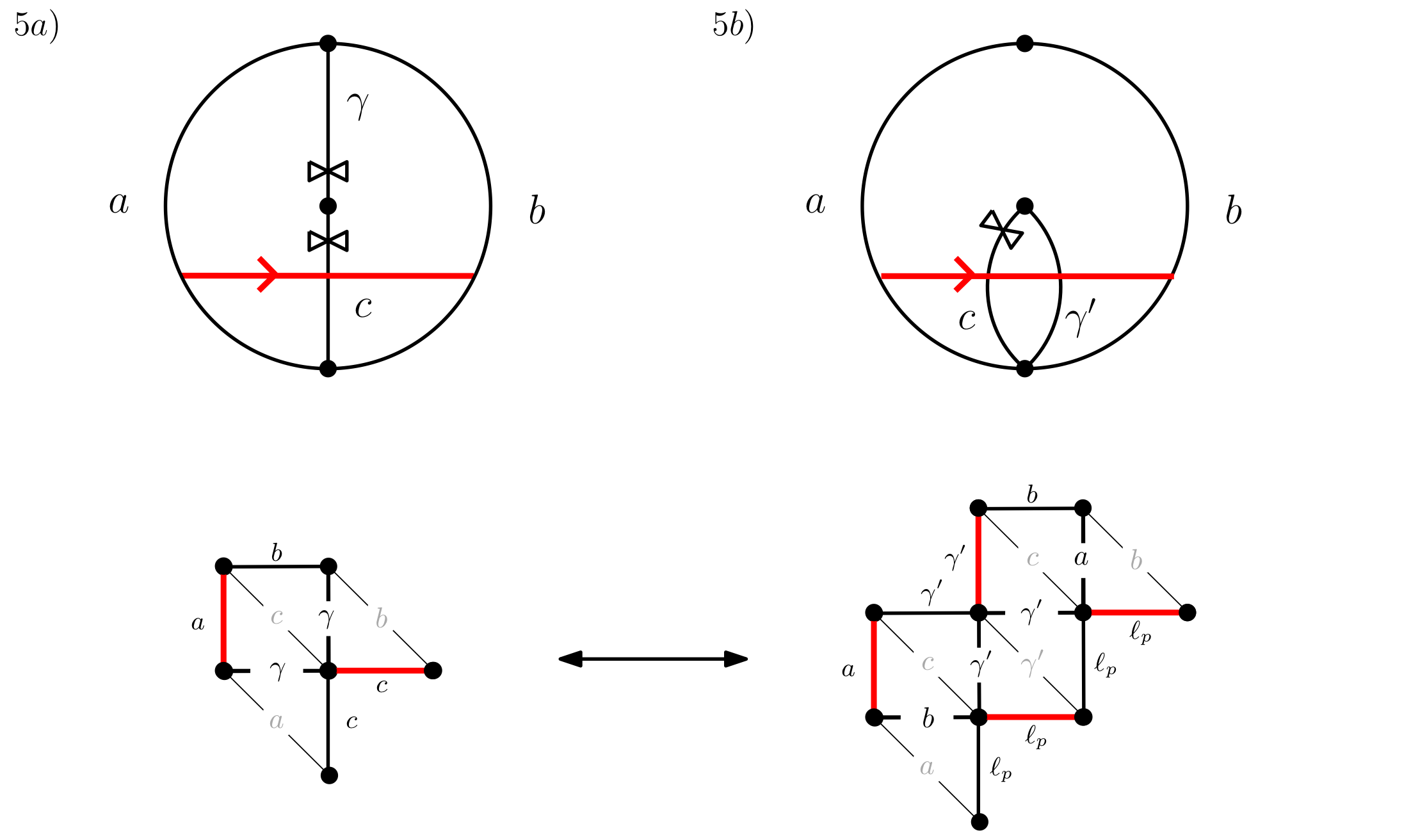}
\includegraphics[width=11cm]{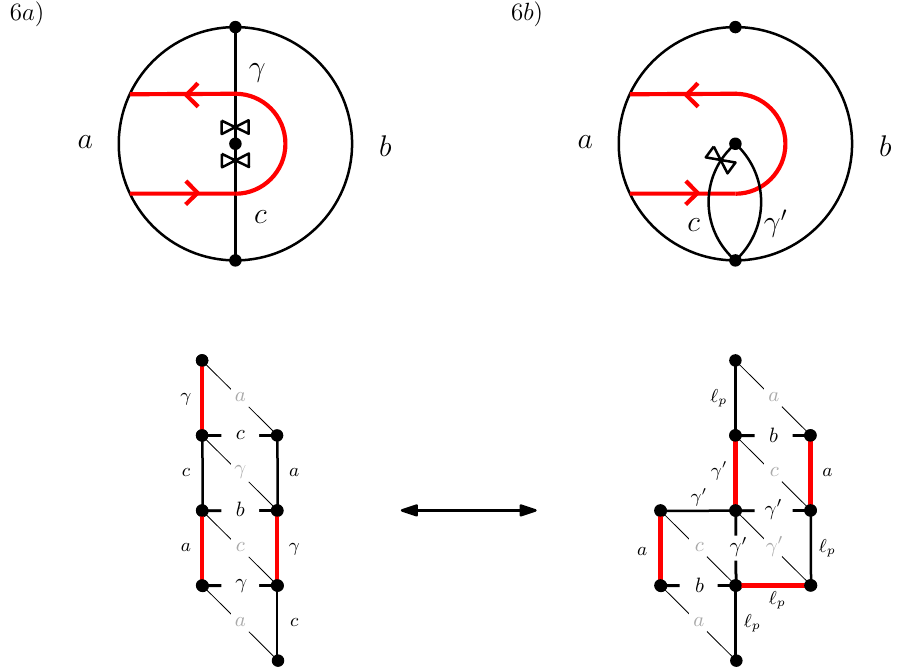}
\end{center}

\begin{lem}
\label{gvector difference}
Let $\alpha$, $T$ and $T'$ be as in the statement of Theorem \ref{thm: comb key lemma}, and let $\mathcal{L}_k$ denote the set of all local curves of $\alpha$ with respect to the flipping/mutating index $k$. Then the following equality holds:
\begin{equation}
\label{gvector sum}
g_{G(T,\alpha)} - \displaystyle \sum_{\beta \in \mathcal{L}_k} g_{G(T,\beta)} = g_{G(T',\alpha)} - \displaystyle \sum_{\beta \in \mathcal{L}_k} g_{G(T',\beta)}
\end{equation}
Moreover, the $k^{th}$ component of $g_{G(T,\alpha)} - \displaystyle \sum_{\beta \in \mathcal{L}_k} g_{G(T,\beta)}$ is $0$.
\end{lem}

\begin{proof}
By direct inspection of Figure \ref{local gvector flips} we see that the graphs 
\[
G(T,\alpha) \setminus \big(\displaystyle \coprod_{\beta \in \mathcal{L}_k} G(T,\beta)\big) \qquad \text{and} \qquad G(T,\alpha) \setminus \big(\displaystyle \coprod_{\beta \in \mathcal{L}_k} G(T,\beta)\big)
\]
are isomorphic. The validity of Equation (\ref{gvector sum}) immediately follows.

Furthermore, let $\gamma \in T$ be the arc corresponding to $k$. Note that $\gamma$ is an edge label of $G(T,\alpha)$ \emph{if and only if} this edge belongs to $G(T,\beta)$ for some $\beta \in \mathcal{L}_k$. Consequently, the $k^{th}$ component of $g_{G(T,\alpha)} - \displaystyle \sum_{\beta \in \mathcal{L}_k} g_{G(T,\beta)}$ is $0$.
\end{proof}

\begin{prop}
\label{comb gvector mutation 2}
Let $\alpha$, $T$ and $T'$ be as in the statement of Theorem \ref{thm: comb key lemma}.
Then:
\begin{equation}
\label{comb gvector mutation eq 2}
    \left[\begin{array}{c}-B(T')\\
    \phantom{-}\bg_{G(T',\alpha)}\end{array}\right]
    =\mu_k\left(
    \left[\begin{array}{c}-B(T)\\
    \phantom{-}\bg_{G(T,\alpha)}\end{array}\right]\right).
\end{equation}
\end{prop}

\begin{proof}
By Lemma \ref{gvector difference} it suffices to show equation (\ref{comb gvector mutation eq 2}) holds for the six configurations listed in Figure \ref{local gvector flips}. The result then follows by direct inspection.
Note that a little extra care is needed when treating the scenario that boundary edges of the flip region are labelled by some $\ell_p$. To this end, recall that $x_{\ell_p} = x_{\beta}x_{\beta^{(p)}}$ for some $\beta \in T$, and let us denote the arc corresponding to $k$ by $\gamma$. The validity of equation (\ref{comb gvector mutation eq 2}) then follows from the equality $b_{\beta \gamma} = b_{\beta^{(p)} \gamma}$.
\end{proof}

\begin{thm}\label{thm:shear-coords-equal-snake-g-vector-for-closed-curves}
Suppose $\bSigma=\surf$ is a surface with non-empty boundary.
Let $\alpha$ be a simple closed curve on $\bSigma$, and $T$ a tagged triangulation of $\bSigma$. Then $$\mathbf{g}_{G(T,\alpha)} = \Sh_T(\alpha).$$

\end{thm}

\begin{proof}
By Proposition \ref{prop:existence-ad-hoc-triangulation-for-closed-curve}, there exists a tagged triangulation $T_\alpha$ such that $\Sh_{T_\alpha}(\alpha) = \bg_{T_\alpha}(G(T_\alpha,\alpha))$. Since $T$ can be obtained from $T_\alpha$ by a finite sequence of flips, the theorem follows from Proposition \ref{comb gvector mutation 2} and Theorem \ref{thm:behavior-of-dual-shear-coords-under-matrix-mutation}.
\end{proof}

\begin{rmk}
    Theorem \ref{thm:shear-coords-equal-snake-g-vector-for-closed-curves} should be compared to \cite[Corollary 6.15-(2)]{musiker2013bases}, though the reader is advised to be wary of signs and conventions. Our techniques provide a new, independent proof of Theorem \ref{thm:shear-coords-equal-snake-g-vector-for-closed-curves}.
\end{rmk}

\subsection{Proof of equations (\ref{comb kth gvector}) and (\ref{comb gvector mutation}) in the Combinatorial Key Lemma}

\begin{lem}
\label{hk min}
For each $k \in \{1,\ldots, n\}$ the snake $g$- and $h$-vectors $\mathbf{g}_{G(T,\alpha)} = (g_{\alpha; 1},\ldots, g_{\alpha; n})$ and $\mathbf{h}_{G(T,\alpha)} = (h_{\alpha; 1},\ldots, h_{\alpha; n})$ of a simple closed curve $\alpha$ satisfy: 

\begin{equation}
h_{\alpha; k} = \min(0, g_{\alpha; k})
\end{equation}

\end{lem}

\begin{proof}
From Proposition \ref{local hsum} we have: 
\[
h_{\alpha; k} = \displaystyle \sum_{\alpha_{ij} \in \mathcal{L}} h_{\alpha_{ij}; k}.
\]
Directly from the definition of shear coordinates, for any (tagged) triangulation $T$ we have: 
\[
\Sh_T(\alpha) = \displaystyle \sum_{\alpha_{ij} \in \mathcal{L}} \Sh_T(\alpha_{ij}).
\]
Moreover, for any $k \in \{1,\ldots, n\}$ the following two statements hold: \begin{itemize}

\item $g_{\alpha; k} = \Sh_T(\alpha; k)$,

\item $\Sh_T(\alpha_{ij}; k) \geq 0$ for all $\alpha_{ij} \in \mathcal{L}$, or $\Sh_T(\alpha_{ij}; k) \leq 0$ for all $\alpha_{ij} \in \mathcal{L}$.

\end{itemize}
Therefore, to prove the lemma it suffices to show 
\[
h_{\alpha_{ij}; k} = \min\big(0, \Sh_T(\alpha_{ij}; k)\big)
\]
for every $\beta \in \mathcal{L}$. This follows by direct inspection, and from the explicit computations used in the proof of Proposition \ref{local hsum}.
\end{proof}

\begin{proof}[Proof of (\ref{comb kth gvector}) and (\ref{comb gvector mutation}) from Theorem \ref{thm: comb key lemma}.]

The validity of these two equations follow directly from Proposition \ref{comb gvector mutation 2} and Lemma \ref{hk min}.
\end{proof}


\section{Bangle functions are the generic basis}\label{sec:bangle-functions-are-the-generic-basis}

\subsection{The bangle function of a tagged arc belongs to the generic basis}

\begin{prop}\label{thm:bangle-of-tagged-arc-belongs-to-generic-basis}
Let $\bSigma=\surf$ be a surface with non-empty boundary, and let $\alpha$ be a tagged arc on $\bSigma$. For every tagged triangulation $T$ of $\bSigma$, the bangle function $\MSW(G(T,\alpha))$ is equal to the generic value taken by the coefficient-free Caldero-Chapoton function $CC_{A(T)}$ on the irreducible component $\pi_T(\alpha)\in\DecIrr^{\tau^-}(A(T))$.
\end{prop}

\begin{proof}
Since the boundary of $\Sigma$ is not empty, by \cite[Proposition~7.10]{fomin2008cluster} there exists a sequence of flips that transforms $T$ in a tagged triangulation $T_\alpha=f_{k_m}\cdots f_{k_1}(T)$ containing~$\alpha$. The negative simple representation $\mathcal{S}_\alpha^-(A(T_\alpha))$ (see \cite[the paragraph preceding Proposition 10.15]{derksen2008quivers} or~\cite[Equation (1.15)]{derksen2010quivers}) is the unique point in an irreducible component $Z_{T_\alpha,\alpha}$ for the Jacobian algebra $A(T_\alpha)$. This component $Z_{T_\alpha,\alpha}$ is obviously $\tau^-$-rigid, i.e., it is a $\tau^-$-regular component with $E$-invariant zero. By Proposition \ref{prp:JaMut} and the invariance of the (injective) $E$-invariant under mutations \cite[Theorem 7.1]{derksen2010quivers}, $Z_{T,\alpha}:=\widetilde{\mu}_{k_m}\cdots\widetilde{\mu}_{k_1}(Z_{T_\alpha,\alpha})$ is a $\tau^-$-regular component with $E$-invariant zero for $A(T)$, and the generic value $CC_{A(T)}(Z_{T,\alpha})$ is given by $CC_{A(T)}(\mathcal{M}(T,\alpha))$, where $\mathcal{M}(T,\alpha):=\mu_{k_m}\cdots\mu_{k_1}(\mathcal{S}_\alpha^-(A(T_\alpha)))$.

Now, by \cite[Cor.~6.3]{fomin2007cluster} and \cite[Eqn.~(2.14) and Thm.~5.1]{derksen2010quivers}, $CC_{A(T)}(\mathcal{M}(T,\alpha))$ is the cluster variable that corresponds to $\alpha$ according to \cite[Theorem 7.11]{fomin2008cluster}. On the other hand, by \cite[Theorems 4.9, 4.16 and 4.20]{musiker2011positivity}, $CC_{A(T)}(\mathcal{M}(T,\alpha))=\MSW(G(T,\alpha))$.
\end{proof}

\subsection{The bangle function of a closed curve belongs to the generic basis}

\begin{prop}\label{thm:bangle-of-simple-closed-curve-belongs-to-generic-basis}
    Let $\bSigma=\surf$ be a surface with marked points and with non-empty boundary, and let $\alpha$ be a simple closed curve on $\bSigma$. For every tagged triangulation $T$ of $\bSigma$, the bangle function $\MSW(G(T,\alpha))$ is equal to the generic value taken by the coefficient-free Caldero-Chapoton function $CC_{A(T)}$ on the irreducible component $\pi_T(\alpha)\in \DecIrr^{\tau^-}(A(T))$.
\end{prop}

\begin{proof}
Let $T$ be an arbitrary tagged triangulation of $\bSigma$, and let $T_\alpha$, and $M(T_\alpha,\alpha,\lambda,1)$ be as in Proposition \ref{prop:existence-ad-hoc-triangulation-for-closed-curve}. By Corollary \ref{coro:generic-values-on-adhoc-irred-comp-coincide-with-MSW}, the set
\[
    Z_{T_\alpha,\alpha}:=\overline{\bigcup_{\lambda\in\CC^*}\operatorname{GL}_{\mathbf{d}}(\CC)\cdot M(T_\alpha,\alpha,\lambda,1)}
\]
is a generically $\tau^-$-regular indecomposable irreducible component of the representation space $\operatorname{rep}(A(T_\alpha),\mathbf{d})$, where $\mathbf{d}:=\dimv(M(T_\alpha,\alpha,\lambda,1))$, with
\begin{equation}\label{eq:g-vector-matches-and-CC=MSW-for-closed-curve-and-ad-hoc-triangulation}
\bg_{A(T_\alpha)}(Z_{T_\alpha,\alpha})=\bg_{G(T_\alpha,\alpha)} \qquad \text{and} \qquad 
CC_{A(T_\alpha)}(Z_{T_\alpha,\alpha})=\MSW(G(T_\alpha,\alpha)).
\end{equation}
Thus,  $Z_{T_\alpha,\alpha}=\pi_{T_\alpha}(\alpha)$.
Moreover, the functions $\bg_{A(T_\alp)}$ and $CC_{A(T_\alp)}$ assume on
any point in $\bigcup_{\lambda\in\CC^*}\operatorname{GL}_{\mathbf{d}}(\CC)\cdot M(T_\alpha,\alpha,\lambda,1)$ 
the values \eqref{eq:g-vector-matches-and-CC=MSW-for-closed-curve-and-ad-hoc-triangulation}.

Since the boundary of $\Sigma$ is non-empty, there exists a finite sequence $(T_0,T_1,\ldots,T_m)$ of tagged triangulations, with $T_0=T_\alpha$ and $T_m=T$, such that for each $i=1,\ldots,n$, the triangulation $T_i$ is obtained from $T_{i-1}$ by flipping an arc $k_i\in T_{i-1}$. By \cite[Theorem 8.1]{labardini2016quivers} and \cite[Proposition 3.7]{derksen2008quivers}, the Jacobian algebra $A(T):=\cP_\CC(Q'(T),S'(T))$ is isomorphic to the Jacobian algebra of $\mu_{k_m}\cdots\mu_{k_1}(Q'(T_\alpha),S'(T_\alpha))$. By Proposition \ref{prp:JaMut}, 
\[
Z_{T,\alpha}:=\widetilde{\mu}_{k_m}\cdots\widetilde{\mu}_{k_1}(Z_{T_\alpha,\alpha})
\]
    is a $\tau^-$-regular indecomposable irreducible component of the representation varieties of $\cP_\CC(\mu_{k_m}\cdots\mu_{k_1}(Q'(T_\alpha),S'(T_\alpha)))\cong A(T)$. By Theorem \ref{thm:Lams-vs-taured-comps-iso}, we have $Z_{T,\alpha}=\pi_{T}(\alpha)$.

Since $CC_{A(T_\alpha)}(M(T_\alpha,\alpha,\lambda,1))$ is the generic value $CC_{A(T_\alpha)}(Z_{T_\alpha,\alpha})$, Proposition~\ref{prp:JaMut} implies that $CC_{A(T)}(M(T,\alpha,\lambda,1))$ is the generic value $CC(Z_{T,\alpha})$, where $$M(T,\alpha,\lambda,1):=\mu_{k_m}\cdots\mu_{k_1}(M(T_\alpha,\alpha,\lambda,1)).$$

By Proposition~\ref{prop:existence-ad-hoc-triangulation-for-closed-curve} and Corollary \ref{coro:generic-values-on-adhoc-irred-comp-coincide-with-MSW} we have 
    $$\bg_{A(T_\alpha)}(M(T_\alpha,\alpha,\lambda,1))=\bg_{A(T_\alpha)}(Z_{T_\alpha,\alpha})=\Sh_{T_\alpha}(\alpha)=\mathbf{g}_{G(T_\alpha,\alpha)},
    $$ 
    so Proposition \ref{prp:JaMut} and Theorems \ref{thm:Lams-vs-taured-comps-iso} and \ref{thm:shear-coords-equal-snake-g-vector-for-closed-curves} imply that
    $$\bg_{A(T)}(M(T,\alpha,\lambda,1))=\bg_{A(T)}(Z_{T,\alpha})=\Sh_{T}(\alpha)=\mathbf{g}_{G(T,\alpha)}.$$

At this point, we have shown that the identity $\bg_{A(T')}(Z_{T',\alpha})=\mathbf{g}_{G(T,\alpha)}$ holds for every tagged triangulation, which allows us to see that Theorem \ref{thm: comb key lemma} and \cite[Theorem 5.1 and Lemma 5.2]{derksen2010quivers} and the particular equality $F_{M(T_\alpha,\alpha,\lambda,1)}=F_{G(T_\alpha,\alpha)}$ that was established in Proposition \ref{prop:existence-ad-hoc-triangulation-for-closed-curve}, imply that $F_{M(T,\alpha,\lambda,1)}=F_{G(T,\alpha)}$.

We deduce that $CC_{A(T)}(Z_{T,\alpha})=CC_{A(T)}(M(T,\alpha,\lambda,1))=\MSW(G(T,\alpha))$ as elements of the Laurent polynomial ring $\mathbb{Z}[x_j^{\pm 1}\suchthat j\in T]$.
\end{proof}

\subsection{Main result}

Recall that for any surface with marked points $\bSig$ with 
a (tagged) triangulation $T$  the  
Caldero-Chapoton algebra~$\cC_{A(T)}(\bSig)$  is spanned by the 
Caldero-Chapoton functions of all decorated representations
of $A(T)$. We have a well-known chain of inclusions
$\cA(\bSig)\subset\cC_{A(T)}\subset\cU(\bSig)$. 
See for example~\cite{cerulli2015caldero-chapoton} for more
details.  
.  

\begin{thm}\label{thm:bangle-equals-generic-basis}
Let $\bSigma=\surf$ be a surface with marked points such that 
$\partial\Sig\neq\emptyset$. Then the following holds for each (tagged) triangulation $T$ of $\bSig$:
\begin{itemize}
\item[(a)]
We have 

\begin{itemize}
\item[(i)] $F_{G(T,\alpha)} = F_{\pi_T(L)}$ and
\item[(ii)] $\MSW(T, L)=CC_{A(T)}(\pi_T(L))$ 
\end{itemize}

\noindent for all laminations $L\in\Lamin(\bSig)$, where $\pi_T\colon \Lamin(\bSig)\ra\DecIrr^{\tau^-}(A(T))$ is the isomorphism of
partial KRS-monoids from Theorem~\ref{thm:Lams-vs-taured-comps-iso}. 
\item[(b)]
In particular, the set $\mathcal{B}^\circ(\bSigma, T)$ of coefficient-free bangle functions is equal to the generic basis  $\mathcal{B}_{A(T)}$ of the coefficient-free Caldero-Chapoton algebra $\cC_{A(T)}$.
\item[(c)] 
If $\Pu=\emptyset$ or $\abs{\Ma}\geq 2$ the set $\cB^\circ(\bSig,T)=\cB_{A(T)}$ is a basis of $\cA(\bSig)=\cU(\bSig)$. 
\end{itemize}
\end{thm}

\begin{proof} \hfill

(a)(i) Note that the case when $L$ consists entirely of tagged arcs is known due to the work of Derksen-Weyman-Zelevinsky~\cite[Theorem 5.1]{derksen2010quivers}. The case when $L$ contains a closed curve follows immediately from Theorem \ref{thm: comb key lemma}, Proposition \ref{prop:existence-ad-hoc-triangulation-for-closed-curve}, and ~\cite[Lemma 5.2]{derksen2010quivers}.

(a)(ii) By Theorem \ref{thm:Lams-vs-taured-comps-iso}, we know that there is a bijection between the set of single laminates in $\Lamin(\bSigma)$ and the set of indecomposable $\tau^-$-regular components of the representation spaces of the Jacobian algebra $A(T)$. Furthermore, we have proved in Propositions~\ref{thm:bangle-of-tagged-arc-belongs-to-generic-basis} and \ref{thm:bangle-of-simple-closed-curve-belongs-to-generic-basis} that for the indecomposable $\tau^-$-regular component $Z_{T,\alp}=\pi_T(\alp)$ associated to each single laminate $\alpha$, the coefficient-free generic Caldero-Chapoton function $CC_{A(T)}(Z_{T,\alpha})$ is equal to $\MSW(G(T,\alpha))$.

On the other hand, it follows easily from~\cite[Lemma~4.11]{cerulli2015caldero-chapoton} and the definitions, that when $Z=\overline{Z_1\oplus Z_2}$ for
$Z_1, Z_2\in\DecIrr^{\tau^-}(A(T)$ (with $E_{A(T)}(Z_1, Z_2)=0$) then we have $CC_{A(T)}(Z)= CC_{A(T)}(Z_1)\cdot CC_{A(T)}(Z_2)$.  
Since $\pi_T$ is an isomorphism of (tame) partial KRS-monids by Theorem~\ref{thm:Lams-vs-taured-comps-iso}, our claim follows from
these observations and the definition of $\MSW(T,L)$.

(b) Since $\pi_T$ is bijective we  get by (a) in particular
$\cB^\circ(\bSig, T)=\cB_{A(T)}$.  Next, suppose that $\bSig$ is not
the dreaded torus $\bSig_1$ (\emph{i.e.}\ we exclude the case $g(\Sig)=1, \Pu=\emptyset, \Ma=\{m\}$). Then, since $\partial\Sig\neq\emptyset$,
by~\cite[Thm.~1.4]{geiss2016the} the potential $S'(T)$ is up to right equivalence  the unique non-degenerate potential for $Q'(T)$. 
Thus, in this case $\cB_{A(T)}$ is by~\cite[Cor.~6.14]{geiss2020generic}
a basis of the Caldero-Chapoton algebra $\cC_{A(T)}$. 

In the only remaining case, the dreaded torus, by the main result of~\cite{Canakci2015OnCluster} we have in particular that
$\cB^\circ(\bSig_1, T)$ is a  basis of
$\cA(\bSig_1)=\cU(\bSig_1)$.  So our claim follows also in this case
since $\cA(\bSig)\subset\cC_{A(T)}\subset\cU(\bSig)$. 

(c) We have under this hypothesis $\cA(\bSig)=\cU(\bSig)$ either by the
main result from~\cite{Canakci2015OnCluster}, or because $\cA(\bSig)$
is locally acyclic by~\cite[Thm.~10.6 \& Thm.~4.1]{muller2013locally}. 
So our claim follows from~(b) together with  the inclusions $\cA(\bSig)\subset\cC_{A(T)}\subset\cU(\bSig)$.
\end{proof}

\section{Applications and open problems}\label{sec:remarks-and-problems}

During the proofs of Theorems \ref{thm:bangle-of-tagged-arc-belongs-to-generic-basis} and \ref{thm:bangle-of-simple-closed-curve-belongs-to-generic-basis}, we have shown that, given any lamination $\alpha \in\Lamin(\bSig)$, one can find a (tagged) triangulation $T_\alpha$ and a very concretely constructed (decorated) representation $\mathcal{M}(T_\alpha,\alpha)$ of the Jacobian algebra $A(T_\alpha)=\jacobalg{Q(T_\alpha),S(T_\alpha)}$ such that
\begin{enumerate}
\item $\bg_{A(T_\alpha)}(\mathcal{M}(T_\alpha,\alpha))=\mathbf{g}_{G(T_\alpha,\alpha)}$ \and $F_{M(T_\alpha,\alpha)}=F_{G(T_\alpha,\alpha)}.$
\item for any tagged triangulation $T=f_{k_m}\cdots f_{k_1}(T_\alpha)$, the decorated representation $\mathcal{M}(T,\alpha):=\mu_{k_m}\cdots\mu_{k_1}(\mathcal{M}(T_\alpha,\alpha))$ of the non-degenerate quiver with potential $(Q(T),S(T))$ satisfies
$\bg_{A(T)}(\mathcal{M}(T,\alpha))=\mathbf{g}_{G(T,\alpha)}$ \and $F_{M(T,\alpha)}=F_{G(T,\alpha)}.$
\end{enumerate}
Indeed, when $\alpha$ was a tagged arc, we picked $T_\alpha$ to be any tagged triangulation containing it, and $\mathcal{M}(T_\alpha,\alpha)$ to be the negative simple representation of $(Q(T_\alpha),S(T_\alpha))$ corresponding to $\alpha$; when $\alpha$ was a simple closed curve, we picked $T_\alpha$ to be a triangulation such that the Jacobian algebra $\Lambda_I$ of the restriction of $(Q(T_\alpha),S(T_\alpha))$ to the set $I$ of arcs crossed by $\alpha$ is gentle (see Proposition~\ref{prop:existence-ad-hoc-triangulation-for-closed-curve}), and $\mathcal{M}(T_\alpha,\alpha)$ to be the positive representation of $(Q(T_\alpha),S(T_\alpha))$ given by any quasi-simple band module $M(T_\alpha,\alpha,\lambda,1)$ arising from interpreting $\alpha$ as a band on $\Lambda_I$.

As a consequence, we get the following immediate application.

\begin{cor}

Let $\bSigma=\surf$ be a surface with marked points such that $\partial\Sig\neq\emptyset$, and let $T$ be a tagged triangulation of $\bSigma$.

Then for all laminations $\alpha \in\Lamin(\bSig)$, if $L$ is a closed curve (resp. a tagged arc) then the Euler-Poincaré characteristic $\chi(Gr_{e}(\mathcal{M}(T,\alpha)))$ equals the number of good matchings $P$ of the band graph (resp. loop graph) $G(T,\alpha)$ such that $y(P) = \prod_{i=1}^{n} y_i^{e_i}.$ 

\end{cor}

Our ability to obtain the above result is quite peculiar. Note that, although $\mathcal{M}(T_\alpha,\alpha)$ is a well-defined representation, we have not provided an explicit computation of $\mathcal{M}(T,\alpha)$ in general, but are still able to compute its Euler-Poincaré characteristic. Of course, when $\punct=\varnothing$, the Jacobian algebra $A(T)$ is gentle, and $\mathcal{M}(T,\alpha)$ can be written down explicitly. When $\punct\neq\varnothing$ but the signature of $T$ is zero, the Jacobian algebra $A(T)$ is skewed-gentle, and $\mathcal{M}(T,\alpha)$ can be computed explicitly as well, cf. \cite{crawley1989functorial,geiss2023onhomomorphisms,geiss2023laminations}. However, when $\punct\neq\varnothing$ and $T$ is arbitrary, the decorated representation $\mathcal{M}(T,\alpha)$ still remains to be explicitly computed in general. It should be noticed that the naive candidate for $\mathcal{M}(T,\alpha)$, namely, the obvious string or band representation of the quiver $Q(T)$ induced by $\alpha$, typically fails to be annihilated by the cyclic derivatives of the potential $S(T)$, see e.g. \cite[Example 6.2.7]{labardini2010quivers}. Explicit computations of $\mathcal{M}(T,\alpha)$ have been carried out by the second named author in \cite{labardini2010quivers} in the following situations:
\begin{itemize}
    \item when $T$ is a tagged triangulation of positive signature and $\alpha$ is a tagged arc with at most one notch;
    \item when $T$ is a tagged triangulation of positive signature and $\alpha$ is a simple closed curve (this is only implicit in \cite{labardini2010quivers}, but one can check that the results proved therein apply in this situation).
\end{itemize}

\begin{prob}
    Compute the decorated representation $\mathcal{M}(T,\alpha)$ of $(Q(T),S(T))$ in general. To appreciate the complexity of this task we also direct the reader to ~\cite{domínguez2017arc} .
\end{prob}

On an arguably more important matter, one of the reasons why in this paper we have not considered surfaces with empty boundary whatsoever, is that the proof given in \cite{geiss2020generic} of the linear independence of the set of generic Caldero-Chapoton functions can definitely not be applied for such surfaces.

\begin{prob}
For a tagged triangulation $T$ of a punctured surface with empty boundary $(\Sigma,\mathbb{P})$, is the set of generic Caldero-Chapoton functions over $\mathcal{P}(Q(T),S(T))$ linearly independent? Is it a basis for the Caldero-Chapoton algebra of $\mathcal{P}(Q(T),S(T))$, or better, for the (upper, coefficient-free) cluster algebra of $\Sigma$? What is its relation to Musiker--Schiffler--Williams' bangle functions? Here, $S(T)$ is the potential defined in \cite{labardini2016quivers}.
\end{prob}

\subsection*{Acknowledgments}
We are grateful to Jan Schröer for many illuminating discussions.
The first author acknowledges partial support from PAPIIT grant IN116723
(2023-2025). The work of this paper originated during the third author’s back-to-back visits at IMUNAM courtesy of the first author’s CONACyT-239255 grant and a DGAPA postdoctoral fellowship. He is grateful for the stimulating environment IMUNAM fostered, and for the additional generous support received from the second author’s grants: CONACyT-238754 and a C\'{a}tedra Marcos Moshinsky.

\bibliographystyle{plain}

\bibliography{Biblio_bases}

\begin{thebibliography}{10}

\bibitem{bobiński2025genericallytauregularirreduciblecomponents}
Grzegorz Bobiński and Jan Schröer.
\newblock Generically $\tau$-regular irreducible components of module
  varieties.
\newblock {\em arXiv:2502.13709}, 2025.

\bibitem{butler1987auslander}
Michael C.~R. Butler and Claus~Michael Ringel.
\newblock Auslander-{R}eiten sequences with few middle terms and applications
  to string algebras.
\newblock {\em Communications in Algebra}, 15(1-2):145--179, 1987.

\bibitem{Canakci2015OnCluster}
Ilke Canakci, Kyungyong Lee, and Ralf Schiffler.
\newblock On cluster algebras from unpunctured surfaces with one marked point.
\newblock {\em Proc. Amer. Math. Soc. Ser. B}, 2:35--49, 2015.

\bibitem{carroll2015on}
Andrew~T. Carroll and Calin Chindris.
\newblock On the invariant theory for acyclic gentle algebras.
\newblock {\em Trans. Amer. Math. Soc.}, 367(5):3481--3508, 2015.

\bibitem{cerulli2013linear}
Giovanni Cerulli~Irelli, Bernhard Keller, Daniel Labardini-Fragoso, and
  Pierre-Guy Plamondon.
\newblock Linear independence of cluster monomials for skew-symmetric cluster
  algebras.
\newblock {\em Compos. Math.}, 149(10):1753--1764, 2013.

\bibitem{cerulli2012quivers}
Giovanni Cerulli~Irelli and Daniel Labardini-Fragoso.
\newblock Quivers with potentials associated to triangulated surfaces, {P}art
  {III}: tagged triangulations and cluster monomials.
\newblock {\em Compos. Math.}, 148(6):1833--1866, 2012.

\bibitem{cerulli2015caldero-chapoton}
Giovanni Cerulli~Irelli, Daniel Labardini-Fragoso, and Jan Schr\"{o}er.
\newblock Caldero-{C}hapoton algebras.
\newblock {\em Trans. Amer. Math. Soc.}, 367(4):2787--2822, 2015.

\bibitem{crawley2002irreducible}
William Crawley-Boevey and Jan Schr\"{o}er.
\newblock Irreducible components of varieties of modules.
\newblock {\em J. Reine Angew. Math.}, 553:201--220, 2002.

\bibitem{crawley1989functorial}
William~W. Crawley-Boevey.
\newblock Functorial filtrations. {II}. {C}lans and the {G}el'fand problem.
\newblock {\em J. London Math. Soc. (2)}, 40(1):9--30, 1989.

\bibitem{derksen2008quivers}
Harm Derksen, Jerzy Weyman, and Andrei Zelevinsky.
\newblock Quivers with potentials and their representations. {I}. {M}utations.
\newblock {\em Selecta Math. (N.S.)}, 14(1):59--119, 2008.

\bibitem{derksen2010quivers}
Harm Derksen, Jerzy Weyman, and Andrei Zelevinsky.
\newblock Quivers with potentials and their representations {II}: applications
  to cluster algebras.
\newblock {\em J. Amer. Math. Soc.}, 23(3):749--790, 2010.

\bibitem{domínguez2017arc}
Salomón Domínguez.
\newblock Arc representations.
\newblock {\em arXiv:1709.09521 [math.RT]}, 2017.

\bibitem{fock2006moduli}
V.~Fock and A.~Goncharov.
\newblock Moduli spaces of local systems and higher {T}eichm{\"u}ller theory.
\newblock {\em Publications Math{\'e}matiques de l'IH{\'E}S}, 103:1--211, 2006.

\bibitem{fock2007dual}
V.~Fock and A.~Goncharov.
\newblock Dual {T}eichm{\"u}ller and lamination spaces.
\newblock {\em Handbook of Teichm{\"u}ller Theory, Volume I}, pages 647--684,
  2007.

\bibitem{fomin2010totalp}
Sergey Fomin.
\newblock Total positivity and cluster algebras.
\newblock In {\em Proceedings of the {I}nternational {C}ongress of
  {M}athematicians. {V}olume {II}}, pages 125--145. Hindustan Book Agency, New
  Delhi, 2010.

\bibitem{fomin2008cluster}
Sergey Fomin, Michael Shapiro, and Dylan Thurston.
\newblock Cluster algebras and triangulated surfaces. {P}art {I}. cluster
  complexes.
\newblock {\em Acta Mathematica}, 201(1):83--146, 2008.

\bibitem{fomin2018cluster}
Sergey Fomin and Dylan Thurston.
\newblock Cluster algebras and triangulated surfaces {P}art {II}: {L}ambda
  lengths.
\newblock {\em Mem. Amer. Math. Soc.}, 255(1223):v+97, 2018.

\bibitem{fomin2002cluster}
Sergey Fomin and Andrei Zelevinsky.
\newblock Cluster algebras {I}: foundations.
\newblock {\em Journal of the American Mathematical Society}, 15(2):497--529,
  2002.

\bibitem{fomin2003cluster}
Sergey Fomin and Andrei Zelevinsky.
\newblock Cluster algebras {II}. {F}inite type classification.
\newblock {\em Inventiones Mathematicae}, 154(1):63--121, 2003.

\bibitem{fomin2007cluster}
Sergey Fomin and Andrei Zelevinsky.
\newblock Cluster algebras {IV}. {C}oefficients.
\newblock {\em Compositio Mathematica}, 143(1):112--164, 2007.

\bibitem{geiss2023onhomomorphisms}
Christof Gei\ss.
\newblock On homomorphisms and generically $\tau$-regular components for
  skewed-gentle algebras.
\newblock {\em arXiv:2307.10306, to appear in J. Algebra}, 2023.

\bibitem{geiss2016the}
Christof Gei\ss, Daniel Labardini-Fragoso, and Jan Schr\"{o}er.
\newblock The representation type of {J}acobian algebras.
\newblock {\em Adv. Math.}, 290:364--452, 2016.

\bibitem{geiss2020generic}
Christof Gei{\ss}, Daniel Labardini-Fragoso, and Jan Schr\"oer.
\newblock Generic {C}aldero-{C}hapoton functions with coefficients and
  applications to surface cluster algebras.
\newblock {\em arXiv:2007.05483}, 2020.

\bibitem{geiss2022schemes}
Christof Gei\ss, Daniel Labardini-Fragoso, and Jan Schr\"{o}er.
\newblock Schemes of modules over gentle algebras and laminations of surfaces.
\newblock {\em Selecta Math. (N.S.)}, 28(1):Paper No. 8, 78, 2022.

\bibitem{geiss2023laminations}
Christof Gei{\ss}, Daniel Labardini-Fragoso, and Jon Wilson.
\newblock Laminations of punctured surfaces as $\tau$-regular irreducible
  components.
\newblock {\em arXiv:2308.00792, 47pp., to appear in IMRN}, 2023.

\bibitem{geiss2012generic}
Christof Geiss, Bernard Leclerc, and Jan Schr\"{o}er.
\newblock Generic bases for cluster algebras and the {C}hamber ansatz.
\newblock {\em J. Amer. Math. Soc.}, 25(1):21--76, 2012.

\bibitem{geiss2023semicontinuous}
Christof Geiß, Daniel Labardini-Fragoso, and Jan Schröer.
\newblock Semicontinuous maps on module varieties.
\newblock {\em Journal für die reine und angewandte Mathematik (Crelles
  Journal), doi:10.1515/crelle-2024-0049}, 2024.

\bibitem{gekhtman2005cluster}
Michael Gekhtman, Michael Shapiro, and Alek Vainshtein.
\newblock Cluster algebras and {W}eil-{P}etersson forms.
\newblock {\em Duke Math. J.}, 127(2):291--311, 2005.

\bibitem{gross2018canonical}
Mark Gross, Paul Hacking, Sean Keel, and Maxim Kontsevich.
\newblock Canonical bases for cluster algebras.
\newblock {\em J. Amer. Math. Soc.}, 31(2):497--608, 2018.

\bibitem{haupt2012euler}
Nicolas Haupt.
\newblock Euler characteristics of quiver {G}rassmannians and {R}ingel-{H}all
  algebras of string algebras.
\newblock {\em Algebr. Represent. Theory}, 15(4):755--793, 2012.

\bibitem{kang2018monidal}
Seok-Jin Kang, Masaki Kashiwara, Myungho Kim, and Se-jin Oh.
\newblock Monoidal categorification of cluster algebras.
\newblock {\em J. Amer. Math. Soc.}, 31(2):349--426, 2018.

\bibitem{keller2012cluster}
Bernhard Keller.
\newblock Cluster algebras and derived categories.
\newblock In {\em Derived categories in algebraic geometry}, EMS Ser. Congr.
  Rep., pages 123--183. Eur. Math. Soc., Z\"{u}rich, 2012.

\bibitem{keller2020green}
Bernhard Keller and Laurent Demonet.
\newblock A survey on maximal green sequences.
\newblock In {\em Representation theory and beyond}, volume 758 of {\em
  Contemp. Math.}, pages 267--286. Amer. Math. Soc., [Providence], RI, [2020]
  \copyright 2020.

\bibitem{krause1991maps}
Henning Krause.
\newblock Maps between tree and band modules.
\newblock {\em J. Algebra}, 137(1):186--194, 1991.

\bibitem{labardini2010quivers}
Daniel Labardini-Fragoso.
\newblock {\em Quivers with potentials associated with triangulations of
  {R}iemann surfaces}.
\newblock ProQuest LLC, Ann Arbor, MI, 2010.
\newblock Thesis (Ph.D.)--Northeastern University.

\bibitem{labardini2016on}
Daniel Labardini-Fragoso.
\newblock On triangulations, quivers with potentials and mutations.
\newblock In {\em Mexican mathematicians abroad: recent contributions}, volume
  657 of {\em Contemp. Math.}, pages 103--127. Amer. Math. Soc., Providence,
  RI, 2016.

\bibitem{labardini2016quivers}
Daniel Labardini-Fragoso.
\newblock Quivers with potentials associated to triangulated surfaces, part
  {IV}: removing boundary assumptions.
\newblock {\em Selecta Math. (N.S.)}, 22(1):145--189, 2016.

\bibitem{leclerc2010cluster}
Bernard Leclerc.
\newblock Cluster algebras and representation theory.
\newblock In {\em Proceedings of the {I}nternational {C}ongress of
  {M}athematicians. {V}olume {IV}}, pages 2471--2488. Hindustan Book Agency,
  New Delhi, 2010.

\bibitem{lusztig2000semicanonical}
G.~Lusztig.
\newblock Semicanonical bases arising from enveloping algebras.
\newblock {\em Adv. Math.}, 151(2):129--139, 2000.

\bibitem{mandel2023bracelets}
Travis Mandel and Fan Qin.
\newblock Bracelets bases are theta bases.
\newblock {\em arXiv:2301.11101}, 2023.

\bibitem{mills2017maximal}
Matthew~R. Mills.
\newblock Maximal green sequences for quivers of finite mutation type.
\newblock {\em Adv. Math.}, 319:182--210, 2017.

\bibitem{muller2013locally}
Greg Muller.
\newblock Locally acyclic cluster algebras.
\newblock {\em Adv. Math.}, 233:207--247, 2013.

\bibitem{muller2014A=U}
Greg Muller.
\newblock {$\mathcal {A}=\mathcal{ U}$} for locally acyclic cluster algebras.
\newblock {\em SIGMA Symmetry Integrability Geom. Methods Appl.}, 10:Paper 094,
  8, 2014.

\bibitem{musiker2011positivity}
Gregg Musiker, Ralf Schiffler, and Lauren Williams.
\newblock Positivity for cluster algebras from surfaces.
\newblock {\em Advances in Mathematics}, 227(6):2241--2308, 2011.

\bibitem{musiker2013bases}
Gregg Musiker, Ralf Schiffler, and Lauren Williams.
\newblock Bases for cluster algebras from surfaces.
\newblock {\em Compositio Mathematica}, 149(2):217--263, 2013.

\bibitem{nakanishi2012on}
Tomoki Nakanishi and Andrei Zelevinsky.
\newblock On tropical dualities in cluster algebras.
\newblock In {\em Algebraic groups and quantum groups}, volume 565 of {\em
  Contemp. Math.}, pages 217--226. Amer. Math. Soc., Providence, RI, 2012.

\bibitem{penner2012decorated}
R.~Penner.
\newblock {\em Decorated Teichm{\"u}ller theory}, volume~1.
\newblock European Mathematical Society Publishing House, ETH-Zentrum SEW A27,
  2012.

\bibitem{plamondon2013generic}
Pierre-Guy Plamondon.
\newblock Generic bases for cluster algebras from the cluster category.
\newblock {\em Int. Math. Res. Not. IMRN}, 2013(10):2368--2420, 2012.

\bibitem{qin2020dual}
Fan Qin.
\newblock Dual canonical bases and quantum cluster algebras.
\newblock {\em arXiv preprint arXiv:003.13674}, 2020.

\bibitem{qin2019bases}
Fan Qin.
\newblock Bases for upper cluster algebras and tropical points.
\newblock {\em J. Eur. Math. Soc. (JEMS)}, 26(4):1255--1312, 2024.

\bibitem{qiu2017cluster}
Yu~Qiu and Yu~Zhou.
\newblock Cluster categories for marked surfaces: punctured case.
\newblock {\em Compos. Math.}, 153(9):1779--1819, 2017.

\bibitem{wilson2020surface}
Jon Wilson.
\newblock Surface cluster algebra expansion formulae via loop graphs.
\newblock {\em arXiv preprint arXiv:2006.13218}, 2020.

\end{thebibliography}

\end{document}